\documentclass[a4paper]{amsart}

\pdfoutput=1

\usepackage[utf8]{inputenc}
\usepackage[T1]{fontenc}
\usepackage{lmodern}
\usepackage{amsthm, amssymb, amsmath, amsfonts, mathrsfs}
\usepackage[colorlinks=true, pdfstartview=FitV, linkcolor=blue, citecolor=blue, urlcolor=blue,pagebackref=false]{hyperref}
\usepackage{microtype}

\usepackage{booktabs} 

\usepackage{tikz} 

\usepackage{scalerel}
\newcommand{\pe}{\mathbin{\scaleobj{0.7}{\tikz \draw (0,0) node[shape=circle,draw,inner sep=0pt,minimum size=8.5pt] {\tiny $=$};}}}
\newcommand{\pne}{\mathbin{\scaleobj{0.7}{\tikz \draw (0,0) node[shape=circle,draw,inner sep=0pt,minimum size=8.5pt] {\tiny $\neq$};}}}
\newcommand{\pl}{\mathbin{\scaleobj{0.7}{\tikz \draw (0,0) node[shape=circle,draw,inner sep=0pt,minimum size=8.5pt] {\tiny $<$};}}}
\newcommand{\pg}{\mathbin{\scaleobj{0.7}{\tikz \draw (0,0) node[shape=circle,draw,inner sep=0pt,minimum size=8.5pt] {\tiny $>$};}}}
\newcommand{\ple}{\mathbin{\scaleobj{0.7}{\tikz \draw (0,0) node[shape=circle,draw,inner sep=0pt,minimum size=8.5pt] {\tiny $\leqslant$};}}}
\newcommand{\pge}{\mathbin{\scaleobj{0.7}{\tikz \draw (0,0) node[shape=circle,draw,inner sep=0pt,minimum size=8.5pt] {\tiny $\geqslant$};}}}

\usetikzlibrary{shapes.misc}
\usetikzlibrary{shapes.symbols}
\tikzset{
	dot/.style={circle,fill=black,draw=black, solid,inner sep=0pt,minimum size=0.5mm},
	yy/.style={circle,fill=gray!20,draw=black,inner sep=0pt,minimum size=0.8mm},
	>=stealth,
	}
\makeatletter
\def\DeclareSymbol#1#2#3{\expandafter\gdef\csname MH@symb@#1\endcsname{\tikz[baseline=#2,scale=0.15]{#3}}%
\expandafter\gdef\csname MH@symb@#1s\endcsname{\scalebox{0.6}{\tikz[baseline=#2,scale=0.15]{#3}}}}
\def\<#1>{\csname MH@symb@#1\endcsname}
\makeatother

\DeclareSymbol{X}{-2.4}{\node[dot] {};}
\DeclareSymbol{1}{0}{\draw[white] (-.4,0) -- (.4,0); \draw (0,0)  -- (0,1.2) node[dot] {};}
\DeclareSymbol{2}{0}{\draw (-0.5,1.2) node[dot] {} -- (0,0) -- (0.5,1.2) node[dot] {};}
\DeclareSymbol{3}{0}{\draw (0,0) -- (0,1.2) node[dot] {}; \draw (-.7,1) node[dot] {} -- (0,0) -- (.7,1) node[dot] {};}
\DeclareSymbol{30}{-3}{\draw (0,0) -- (0,-1); \draw (0,0) -- (0,1.2) node[dot] {}; \draw (-.7,1) node[dot] {} -- (0,0) -- (.7,1) node[dot] {};}
\DeclareSymbol{31}{-3}{\draw (0,0) -- (0,-1) -- (1,0) node[dot] {}; \draw (0,0) -- (0,1.2) node[dot] {}; \draw (-.7,1) node[dot] {} -- (0,0) -- (.7,1) node[dot] {};}
\DeclareSymbol{32}{-3}{\draw (0,0) -- (0,-1) -- (1,0) node[dot] {}; \draw (0,0) -- (0,-1) -- (-1,0) node[dot] {}; \draw (0,0) -- (0,1.2) node[dot] {}; \draw (-.7,1) node[dot] {} -- (0,0) -- (.7,1) node[dot] {};}
\DeclareSymbol{20}{-3}{\draw (0,0) -- (0,-1);\draw (-.7,1) node[dot] {} -- (0,0) -- (.7,1) node[dot] {};}
\DeclareSymbol{22}{-3}{\draw (0,0.3) -- (0,-1) -- (1,0) node[dot] {}; \draw (0,0.3) -- (0,-1) -- (-1,0) node[dot] {};\draw (-.7,1) node[dot] {} -- (0,0.3) -- (.7,1) node[dot] {};}
\DeclareSymbol{31p}{-3}{\draw (0,0) -- (0,-1) -- (1,0) node[dot] {}; \draw (0,0) -- (0,1.2) node[dot] {}; \draw (-.7,1) node[dot] {} -- (0,0) -- (.7,1) node[dot] {}; \draw (0,-1) node{\scaleobj{0.5}{\pe}}; }
\DeclareSymbol{32p}{-3}{\draw (0,0) -- (0,-1) -- (1,0) node[dot] {}; \draw (0,0) -- (0,-1) -- (-1,0) node[dot] {}; \draw (0,0) -- (0,1.2) node[dot] {}; \draw (-.7,1) node[dot] {} -- (0,0) -- (.7,1) node[dot] {}; \draw (0,-1) node{\scaleobj{0.5}{\pe}};}
\DeclareSymbol{22p}{-3}{\draw (0,0.3) -- (0,-1) -- (1,0) node[dot] {}; \draw (0,0.3) -- (0,-1) -- (-1,0) node[dot] {};\draw (-.7,1) node[dot] {} -- (0,0.3) -- (.7,1) node[dot] {}; \draw (0,-1) node{\scaleobj{0.5}{\pe}};}

\newtheorem{thm}{Theorem}[section]
\newtheorem{prop}[thm]{Proposition}
\newtheorem{lem}[thm]{Lemma}
\newtheorem{cor}[thm]{Corollary}
\theoremstyle{remark}
\newtheorem{rem}[thm]{Remark}
\theoremstyle{definition}

\renewcommand{\le}{\leqslant}
\renewcommand{\leq}{\leqslant}
\renewcommand{\ge}{\geqslant}
\renewcommand{\geq}{\geqslant}
\newcommand{\ls}{\lesssim}
\renewcommand{\subset}{\subseteq}

\newcommand{\msf}{\mathsf}

\newcommand{\msc}{\mathscr}
\newcommand{\N}{\mathbb{N}}
\newcommand{\na}{\nabla}
\newcommand{\Ll}{\left}
\newcommand{\Rr}{\right}
\newcommand{\lhs}{left-hand side}
\newcommand{\rhs}{right-hand side}
\newcommand{\1}{\mathbf{1}}
\newcommand{\R}{\mathbb{R}}

\newcommand{\Z}{\mathbb{Z}}

\newcommand{\E}{\mathbb{E}}
\newcommand{\ov}{\overline}
\newcommand{\un}{\underline}
\newcommand{\td}{\widetilde}

\newcommand{\eps}{\varepsilon}
\renewcommand{\d}{{\mathrm{d}}}
\newcommand{\dr}{\partial}

\newcommand{\al}{\alpha}
\newcommand{\be}{\beta}
\newcommand{\ga}{\gamma}
\newcommand{\de}{\delta}
\newcommand{\ze}{\zeta}

\newcommand{\si}{\sigma}

\newcommand{\B}{\mathcal{B}}

\newcommand{\F}{\msc{F}}

\newcommand{\dk}{\delta_k}

\newcommand{\tn}{|\!|\!|} 

\newcommand{\De}{\Delta}
\newcommand{\Dd}{\Delta}
\newcommand{\uc}{\un c}
\DeclareMathOperator{\supp}{Supp}

\newcommand{\com}{\mathsf{com}}

\newcommand{\la}{\left\langle}
\newcommand{\ra}{\right\rangle}

\newcommand{\Pl}{\mathcal{P}_\ell}
\newcommand{\Plp}{\mathcal{P}_\ell^{\perp}}
\newcommand{\per}{\mathsf{per}}

\numberwithin{equation}{section}

\title{The dynamic $\Phi^4_3$ model comes down from infinity}
\author{Jean-Christophe Mourrat, Hendrik Weber}

\address[Jean-Christophe Mourrat]{Ecole normale supérieure de Lyon, CNRS, Lyon, France}
\email{jean-christophe.mourrat@ens-lyon.fr}

\address[Hendrik Weber]{University of Warwick, Coventry, United Kingdom}
\email{hendrik.weber@warwick.ac.uk}

\begin{document}

\begin{abstract}

We prove an a priori bound for the dynamic $\Phi^4_3$ model on the torus which is independent of the initial condition. In particular, this bound rules out the possibility of finite time blow-up of the solution. It also gives a uniform control over solutions at large times, and thus allows to construct invariant measures via the Krylov-Bogoliubov method. It thereby provides a new dynamic construction of the Euclidean $\Phi^4_3$ field theory on finite volume. 
Our method is based on the local-in-time solution theory developed recently by Gubinelli, Imkeller, Perkowski and Catellier, Chouk. The argument relies entirely on deterministic PDE arguments (such as embeddings of Besov spaces and interpolation), which are combined to derive energy inequalities.

\thanks{JCM is partially supported by the ANR
Grant LSD (ANR-15-CE40-0020-03). HW acknowledges support by an EPSRC First Grant, a Royal Society University Research Fellowship and the Mathematical Sciences Research Institute where part of this work was completed.}

\bigskip


\medskip


\end{abstract}
\maketitle
%
%
%
%
%

\section{Introduction}
\label{s:intro}
%


The aim of this paper is to prove an a priori bound for the dynamic $\Phi^4_3$ model on the torus. This model is formally given by the stochastic partial differential equation
\begin{equation}
\label{e:eqX}
\Ll\{
\begin{array}{ll}
\dr_t X = \Dd X - X^{3} + m X +  \xi, \qquad \text{on } \R_+ \times [-1,1]^3, \\
X(0,\cdot) = X_0,
\end{array}
\Rr.
\end{equation}
where  $\xi$ denotes a white noise over $\R \times [-1,1]^3$, and $m$ is a real parameter. Our main result, Theorem~\ref{t:apriori} below, implies that for every $p < \infty$ and $\eps>0$ sufficiently small, we have 
\begin{align*}
\E \Ll[ \sup_{0<t \leq 1} \,  \sup_{X_0 \in \B_\infty^{-\frac12-\eps}} \Ll( \sqrt{t} \, \| X(t)\|_{\B_{\infty}^{-\frac12-\eps}}\Rr)^{p}\Rr] < \infty.
\end{align*}
Here and below, for $\al > 0$, we denote by $\B_\infty^{-\alpha}$ the Besov-H\"older space of 
negative regularity $-\alpha$ (see Appendix~\ref{s:Besov}).
This bound is not only strong enough to prove the global existence of solutions for \eqref{e:eqX}, but can also be used to construct invariant measures via the Krylov-Bogoliubov method. 
This last point is particularly interesting, because  equation \eqref{e:eqX} describes the natural reversible dynamics for the $\Phi^4_3$ quantum field theory, which is formally given by the expression 
\begin{equation}\label{QFT-measure}
\mu \propto \exp\Ll(- 2 \int_{[-1,1]^3} \Ll[\frac 1 2 | \nabla X|^2 + \frac{1}{4}X^4 -  \frac m 2 X^2 \Rr] \Rr) \prod_{x \in [-1,1]^3} \d X(x).
\end{equation}
The construction of this  measure was a major result in the programme of constructive quantum field theory, accomplished in the late 60s and 70s \cite{Glimm1,EO,GJultimo,FeldmanOsterwalder,Feldman}. 
Our main result yields an alternative construction through the dynamics \eqref{e:eqX}.


\smallskip

The construction of the dynamics \eqref{e:eqX} in two  and three dimensions was proposed in \cite{ParisiWu}, but in the more difficult three dimensional case very little progress was made before Hairer's 
recent breakthrough results on \emph{regularity structures}; the construction of local-in-time solutions to \eqref{e:eqX} was one of the two principal applications of the theory presented in \cite{Martin1}. Hairer's work triggered a lot of activity: 
Catellier and Chouk \cite{catcho} were able to reproduce a similar local-in-time well-posedness result based on the notion of 
\emph{paracontrolled distributions} put forward by Gubinelli, Imkeller and Perkowski in \cite{Gubi}.  Yet another approach to obtain solutions for short times, based on Wilsonian \emph{renormalisation group} analysis, was given by Kupiainen in \cite{kuppi}. The analysis presented in this article is  based on the paracontrolled approach of \cite{Gubi,catcho}. The emphasis is on deriving an a priori estimate that complements the local solution theory and rules out the possibility of finite time blow-up. Our method relies solely on PDE arguments, such as energy inequalities and parabolic regularity theory.

\smallskip

The main difficulty in dealing with \eqref{e:eqX} or \eqref{QFT-measure} is the 
 irregularity of $X$, which in turn stems from the roughness of the white noise term $\xi$. Realisations of $X$  are distribution valued, so that there is a priori no canonical 
interpretation of the non-linear terms $X^3$ in \eqref{e:eqX} and $X^4$ in \eqref{QFT-measure}. The construction ultimatively involves a renormalisation procedure which amounts to subtracting some infinite counter-terms.
The first important observation which is used to implement this renormalisation,   and which lies at the foundation of all of the local solution theories, is the  \emph{subcriticality} of \eqref{e:eqX} in three dimensions. To explain this property, let us momentarily consider this equation over $\R^d$ for an arbitrary $d \geq 1$. Formally rescaling the equation via
\begin{align*}
\hat{t}= \lambda^2 t, \qquad \hat{x} = \lambda x, \qquad \hat{\xi} = \lambda^{\frac{d+2}{2}} \xi, \qquad \hat{X} = \lambda^{\frac{2-d}{2}} X, \qquad \hat{m} = \lambda^2 m,
\end{align*}
yields
\begin{equation}
\label{e:eqXs}
\dr_{\hat{t}} \hat{X} = {\Dd} \hat{X} - \lambda^{4-d} \hat{X}^{3} + \hat{m} \hat{X} +  \hat{\xi},
\end{equation}
where $\hat{\xi}$ is a space-time white noise with the same law as $\xi$.
This suggests that for $d < 4$, the influence of the non-linear term should vanish as we consider smaller and smaller scales. This corresponds to the well-known fact that the $\Phi^4_d$ theory is \emph{superrenormalisable} in dimension $d < 4$.
 
\smallskip
 
Based on this observation, the first step to implement the renormalisation 
in both the approaches using regularity structures or paracontrolled distributions  is the explicit construction of several terms based on the solution of the \emph{linear} stochastic heat equation\footnote{Throughout the article, we adopt Hairer's convention to denote the terms in the expansion by trees: here the symbol $\<1>$ should be interpreted as a graph with a single vertex at the top which corresponds to the white noise, and with a line below corresponding to a convolution with the heat kernel.  
This graphical notation is extremely useful to keep track of a potentially large number of explicit stochastic objects. } 
\begin{equation}
\label{e:lin}
(\dr_t - \Dd)\<1> =  \xi.
\end{equation}
The renormalisation, that is, the subtraction of diverging counter-terms, is implemented
at this stage.
 For example, the simplest stochastic objects constructed from $\<1>$ are 
$\<2>$ and $\<3>$, which formally play the role of ``$\<1>^2$'' and ``$\<1>^3$''. These objects are constructed by considering a regularised version $\<1>_\delta$ of $\<1>$, e.g. the solution obtained by replacing $\xi$ by its convolution with a smoothing kernel on scale $\delta$, and then taking the limits as 
$\delta $ tends to zero of 
\begin{align}\label{e:Wick-approx}
\<1>_\delta^2 - C_\delta \qquad \text{and} \qquad  \<1>_\delta^3 - 3 C_\delta \<1>_\delta,
\end{align}
for a suitable choice of diverging constant $C_\delta$. The proof of convergence
of these objects makes strong use of explicit representations of the covariance of  $\<1>$ and of its Gaussianity.

\smallskip

In both theories, the full non-linear system \eqref{e:eqX} is only treated in a second step. This step is completely deterministic, with the random terms constructed in the first step treated as an input. 
The solution $X$ is sought in a space of distributions whose small-scale behaviour is described in detail by the explicit stochastic objects. In both theories, this is implemented by replacing the scalar field $X$ by a vector-valued 
function whose components correspond to the different ``levels of regularity'' of $X$.  The scalar equation \eqref{e:eqX} then turns into a coupled system of equations. 
 This point is at the heart of both methods. The approaches via regularity structures and via paracontrolled distributions then differ significantly.
In the regularity structures approach, a local description of the solution $X$ in ``real space''  is given, whereas the paracontrolled approach uses tools from Fourier analysis.  However, in both approaches, local-in-time solutions $X$ are found by performing a Picard iteration for the system of equations interpreted in the mild sense. We stress 
that the renormalisation is completely treated at the level of the construction of the stochastic objects based on \eqref{e:lin}, and that no ``infinite constants'' appear in the deterministic analysis. 

\smallskip

All three approaches mentioned above, i.e. regularity structures \cite{Martin1}, paracontrolled distributions \cite{Gubi,catcho} and renormalisation group \cite{kuppi},  focus on the problems arising in the analysis of \eqref{e:eqX} on small scales, and devise a powerful method to deal with the so-called ultra-violet divergences. However, extra ingredients are necessary to obtain information on large scales. 
This already becomes apparent from the fact that the ``good'' sign of the term $-X^3$ is not used in the construction of local solutions. In fact, the theories would allow for the construction of local-in-time solutions of \eqref{e:eqX} with the sign of the non-linear term reversed, and solutions of this modified equation 
are expected to blow up in finite time. Moreover, the scaling analysis above suggests that it is the non-linear term $-X^3$ which dominates the dynamics on large scales, so that it can no longer be treated as a perturbation. 

\smallskip

In situations where the noise is less irregular, there are well-known tools available to obtain large scale information on non-linear equations such as \eqref{e:eqX}. In the deterministic case $\xi =0$, the non-linear term is known to have a strong damping effect, and the non-linear equation 
satisfies better bounds than the linearised version: for solutions of \eqref{e:eqX} with $\xi=0$ (started with an $L^\infty$ initial datum, say), a simple argument based on the comparison principle and the behaviour of the ODE $\dot{x} = - x^3 +mx$ immediately yields $\| X(t)\|_{L^{\infty}} \lesssim t^{-\frac12}+1$, where the implicit constant does not depend on the initial datum.
 Other standard tools to extract information on the non-linear term involve testing the equation against $X$ or powers of $X$. In this paper, we show how comparable arguments can be implemented in the context of the system of equations arising in the paracontrolled solution theory of~\eqref{e:eqX}.

\subsection{Formal derivation of a system of equations}

The obvious difficulty in developing a solution theory for  \eqref{e:eqX} is the fact that the solution $X$ will be a distribution, and that it is unclear how to 
interpret the non-linear expression $-X^3$. However, as we have explained in the previous section, on small scales $X$ is expected to ``behave like'' 
the  Gaussian process $\<1>$; more precisely, we expect that $X - \<1>$ has better regularity than each of the terms separately. Moreover, the detailed knowledge of the covariance and the Gaussianity of $\<1>$ can be used to define the ``renormalised'' products 
$$
(\<1>)^2 \leadsto \<2> \quad \text{and} \quad (\<1>)^3 \leadsto \<3>,
$$
via \eqref{e:Wick-approx}.
In this section, we present a formal computation in the spirit of \cite{catcho} to reorganise \eqref{e:eqX} into a system that we are able to solve, assuming that 
we can define the products of the explicit stochastic terms, even if they are distributions of low regularity. For the moment, we will ignore the ``infinite constants''
and manipulate the equation formally, adopting the following rules:
\begin{itemize}
\item Every term has a regularity exponent associated with it. We will say, for example, that the terms $X$ and $\<1>$ have regularity $(-\frac12)^-$, i.e.  regularity $\frac12-\eps$ 
for $\eps$ arbitrarily small. All regularities are derived from the regularity of the white noise $\xi$, which is $(-\frac52)^ -$.
\item A function of regularity $\alpha_1>0$ can be multiplied with a distribution of regularity $\alpha_2<0$ if $\alpha_1 + \alpha_2>0$, resulting in a distribution of 
regularity $\alpha_2$. 
\item Convolution with the heat kernel of $\partial_t - \Delta$ increases the regularity by $2$.
\item Explicit stochastic objects can \emph{always} be multiplied, irrespective of their regularity. The product of stochastic objects of regularity $\alpha_1$ and $\alpha_2$ has 
regularity $\min\{\alpha_1, \alpha_2, \alpha_1+ \alpha_2 \}$. 
\end{itemize}
In Section \ref{s:renormalisation}, we will give a precise meaning to these statements and discuss in particular how the last of these rules has to be 
interpreted. There, we will give a rigorous link between the system we derive formally in this section and the original equation \eqref{e:eqX}.

\smallskip

For illustration, we briefly show this calculation in the two-dimensional case $d = 2$, sketching a method introduced by Da Prato and Debussche in \cite{DPD}.
In dimension $2$, the noise $\xi$ has regularity $(-2)^-$, so both $X$ and $\<1>$ have regularity $0^{-}$. According to the rules above, we cannot define $X^3$ directly (the regularity 
being negative), but we can define the square $\<2>$ and the cube $\<3>$ of $\<1>$, both of which also have regularity $0^{-}$. If we  make the ansatz $X = \<1> + Y$, then $Y$ solves
\begin{equation}
\label{e:defY}
(\dr_t - \Dd) Y = -Y^3 - 3 Y^2 \<1> - 3 Y \<2> - \<3> + m(\<1> + Y).
\end{equation}
Convolution with the heat kernel increases regularity by $2$, so that we expect $Y$ to have regularity $2^-$, which in turn allows to define all the products on the right hand side. 
Hence, we can solve \eqref{e:defY}, at least locally in time. We define the solution we seek, as a replacement for \eqref{e:eqX}, to be $X := \<1> + Y$. 

\smallskip

We now come back to our original problem, posed in three space dimensions. As stated above, in this case $\xi$ has regularity $(-\frac52)^-$, so that $X$ and $\<1>$
have regularity $(-\frac12)^-$, $\<2>$ has regularity $(-1)^-$ and $\<3>$ has regularity $(-\frac32)^-$. Therefore, the simple procedure leading to \eqref{e:defY} does not 
suffice, as it would lead to $Y$ being of regularity $(\frac12)^-$, which is not enough to define the products on the right-hand side of \eqref{e:defY}. The most irregular term we 
encounter in this approach, limiting the regularity of $Y$ to $(\frac12)^-$,  is the term $\<3>$, so we use it to define the next-order term in our expansion. We introduce
 $\<30>$,   the solution of
\begin{equation}\label{e:def30}
(\dr_t - \Dd)\<30> =  \<3>,
\end{equation}
which has regularity $(\frac 1 2)^-$,
and postulate an expansion of the form
\begin{equation}
\label{e:defu}
X = \<1> - \<30> + u,
\end{equation}
for some hopefully more regular $u$. 
Analogously to the two-dimensional case, we write the formal equation satisfied by $u$:
\begin{align*}
(\dr_t - \Dd) u & = -(u + \<1> - \<30>)^3 + m(u + \<1> - \<30>) - \<3> \\
& = -u^3 - 3(u - \<30> )\, \<2> +  Q(u),
\end{align*}
where we introduced the notation  
$$
Q(u) = b_0 + b_1 u + b_2 u^2,
$$
with
\begin{align*}
b_0 & = m(\<1> - \<30>)+ (\<30>)^3 - 3 \<1> \, (\<30>)^2, \\
b_1 & = m + 6 \, \<30> \, \<1> - 3 (\<30>)^2, \\
b_2 & = -3 \, \<1> + 3 \,  \<30>.
\end{align*}
All of these coefficients have regularity $(-\frac12)^-$.
Since the regularity of $\<2>$ is $(-1)^-$, the regularity of $u$ is expected to be $1^-$, so that the product $u \, \<2>$ is still ill-defined a priori. 

\smallskip

In order to solve  
this problem, we use the notion of paraproducts, following \cite{Gubi}. Roughly speaking, the paraproduct of $f$ and $g$, which we denote by $f \pl g$, carries the high-frequency modes of $g$, modulated by the low-frequency modes of $f$. The product $fg$ can be written
\begin{equation}
\label{e.prod}
fg = f \pl g + f \pe g + f \pg g,
\end{equation}
where $f \pe g$ carries the resonant interactions between $f$ and $g$. 
The striking property of paraproducts is that, on the one hand, the quantities $f \pl g$ and $f \pg g$ are always well-defined, and only the resonant term $f \pe g$ can fail to be defined. But on the other hand, whenever the resonant term is well-defined, its regularity is given by the \emph{sum} of regularities of $f$ and $g$ (as opposed to the minimum).
We refer to the appendix for a more precise discussion, in particular Proposition~\ref{p:mult}. We use \eqref{e.prod} with $f = u - \<30>$ and $g = \<2>$, and decompose $u$ into $v + w$ solving
\begin{align}
(\dr_t - \Dd) v & = -3(v+w- \<30> ) \pl \<2> , \label{e:defv1}\\
(\dr_t - \Dd) w & =  -(v+w)^3 - 3(v+w - \<30>) \pge \<2> + Q(v+w) \label{e:defw1},
\end{align}
where we write $\pge = \pg + \pe$ for concision. 
The idea is that $v$ carries the same local irregularity as $u$, while $w$ should have better regularity, namely $(\frac 3 2)^-$ instead of $1^-$. The paraproduct in the right side of \eqref{e:defv1} contains the high-frequency modes of $\<2>$ modulated by the low-frequency modes of $(v+w-\<30>)$. It  is always well-defined and has regularity $(-1)^-$. 
The paraproduct $(v+w- \<30> ) \pg \<2>$ is also well-defined and has regularity $(-\frac 1 2)^{-}$. It remains to consider the resonant term
$$
(v+w- \<30> ) \pe \<2>,
$$
which cannot be made sense of classically (the criterion being the same as for the product of course, that is, the sum of regularities should be strictly positive). 
As was pointed out above, this term should have regularity given by the sum of the regularities of each term, that is, regularity $(-\frac 1 2)^{-}$ in our case.
Since $w$ is expected to have regularity $(\frac 3 2)^-$, the term $w \pe \<2>$ can be made sense of classically. In extension of our rules, we postulate that we can define  $\<30> \pe \<2>=:  \<32p>$ as a 
distribution of regularity $(-\frac12)^{-}$. 

It remains to treat the term $v \pe \<2>$. The key advantage of the decomposition using paraproducts lies in the following commutator estimates, which 
allow to rewrite this term using explicit graphical terms of low regularity and more regular objects involving $v$ and $w$. As a first step,
we denote by $\<20>$ the solution of
\begin{equation}
\label{e:def20}
(\dr_t - \Dd) \<20> = \<2> \qquad (\<20>(t=0) = 0),
\end{equation}
that is,
\begin{equation}
\label{e:def20bis}
 \<20>(t) = \int_0^t e^{(t-s)\Delta} \<2>(s) \, \d s. 
\end{equation}
We also write \eqref{e:defv1} in the mild form
\begin{align*}
v(t) = e^{t \Delta }v_0 -3  \int_0^t e^{(t-s)\Delta} \Ll[   (v+w- \<30> ) \pl \<2> \Rr](s) \, \d s. 
\end{align*}
The behaviour of the heat kernel suggests that the local irregularity of $v$ is  that of  $-3(v+w - \<30>) \pl \<20>$. In other words, the difference 
\begin{equation}
\begin{split}
\label{e:def:com1}
\msf{com}_1(v,w)(t) :=e^{t \Delta }v_0 - 3 \int_0^t e^{(t-s)\Delta}   \Ll[ (v+w- \<30> ) \pl \<2>\Rr] (s) \, \d s \\
 + 3\Ll[(v+w - \<30>) \pl \<20>\Rr](t)
 \end{split}
\end{equation}
has better regularity than $v$ itself. (Justifying this relies on Proposition~\ref{p:comm2} and on suitable \emph{time} regularity of $v$, $w$ and $\<30>$.) 
We thus decompose $v \pe \<2>$ into
$$
v \pe \<2> = -3\Ll[(v+w - \<30>) \pl \<20>\Rr] \pe \<2> + \msf{com}_1(v,w) \pe \<2>.
$$
The second of these terms is defined classically, and it only remains to control the first term. Recall that $(v+w - \<30>) \pl \<20>$ carries the high-frequency modes of $\<20>$, modulated by the low-frequency modes of $(v+w-\<30>)$. Hence, it is reasonable to expect $\Ll[(v+w - \<30>) \pl \<20>\Rr] \pe \<2>$ to have the same local irregularity as 
$$
(v+w-\<30>) \<22p>,
$$
where $\<22p>$ is a postulated version of the resonant term $\<20> \pe \<2>$. To be more precise, the domain of the commutation operator
$$
[\pl,\pe] : (f,g,h) \mapsto (f \pl g) \pe h - f  (g \pe h)
$$
can be extended to cases for which the terms appearing in the definition are not well-defined separately (see Proposition~\ref{p:comm1}), so that
\begin{equation}
\label{e:def:com2}
\msf{com}_2(v+w) := [\pl,\pe]\Ll( -3(v+w-\<30>), \<20>,\<2> \Rr)
\end{equation}
is well-defined. Our renormalisation rule is thus given by
$$
-3\Ll[(v+w - \<30>) \pl \<20>\Rr] \pe \<2> \leadsto -3(v+w-\<30>)  \<22p> + \msf{com}_2(v+w),
$$
that is,
$$
v \pe \<2> \leadsto -3(v+w - \<30>)  \<22p> + \msf{com}(v,w),
$$
where 
\begin{equation}
\label{e:def:com}
\msf{com}(v,w) := \msf{com}_1(v,w) \pe \<2> + \msf{com}_2(v+w).
\end{equation}

To sum up, we are interested in solutions of the system
\begin{equation}
\label{e:eqvw}
\Ll\{
\begin{array}{lll}
(\dr_t - \Dd) v & = & F(v+w),\\
(\dr_t - \Dd) w & = & G(v,w),
\end{array}
\right.
\end{equation}
where $F$ and $G$ are defined by
\begin{align}
\label{e:defF}
F(v+w) & := -3(v+w- \<30> ) \pl \<2>, \\
\label{e:defG}
G(v,w) & := -(v+w)^3  - 3 \mathsf{com}(v,w) \\
& \qquad \qquad -3w \pe \<2> - 3(v+w-\<30>) \pg \<2> + P(v+w), \notag
\end{align}
with
\begin{equation}
\label{e:defP}
P(v+w) = a_0 + a_1(v+w) + a_2(v+w)^2,
\end{equation}
\begin{align*}
& a_0  = b_0  +3 \<30> \pe \<2> - 9 \<30> \<22p>
&a_1  = 
 b_1 +9\, \<22p>, \\
&a_2  = b_2
\end{align*}
with $\msf{com}$ defined by \eqref{e:def:com}, \eqref{e:def:com1} and \eqref{e:def:com2}. 
%

\subsection{Renormalised system}\label{s:renormalisation}
We now turn to giving a precise meaning to the discussion of the previous section. From now on, we refer to processes represented by diagrams as ``the diagrams''. For such a process, we understand the notion of ``being of regularity $\al$'' as meaning that it belongs to $C([0,\infty), \B_\infty^\alpha)$.  This definition would have to be modified for $\xi$ and $\<3>$, which only make sense as space-time distributions, but we will not refer to these any longer.
We refer the reader to Appendix~\ref{s:Besov} for the definition and some properties of the Besov spaces $\B^{\al}_p$. These spaces are more commonly denoted by $\B^\al_{p,q}$, but since we do not make use of fine properties encoded by the second integrability index $q$, we will always set it equal to $\infty$ and  drop it in the notation. For the diagram $\<30>$, some additional information on its time regularity will also be needed. 

\smallskip

We now discuss briefly in which way the system \eqref{e:eqvw} can be linked to the original 
equation rigorously, and in particular in which sense the products (and resonant terms) of the diagrams of low regularity should be interpreted. 
The diagrams entering our equations for $v$ and $w$ are 
\begin{equation}
\label{e:fund-diag}
\<1>, \<2>, \<30>, \<31p>, \<32p>, \<22p>,
\end{equation}
as well as $\<20>$, which is defined as the solution of \eqref{e:def20}, that is, as a function of $\<2>$. 
These quantities, together with their regularity exponent, are summarized in Table~\ref{tab:diag}.
{\small
\begin{table}
\centering
\renewcommand{\arraystretch}{1.5}
\begin{tabular}{ccccccc}
\toprule
$\tau$ & $\<1>$& $\<2>$& $\<30>$& $\<31p>$& $\<32p>$& $\<22p>$ 
\\
\midrule
$ \ \alpha_\tau \ $ & $\ -\frac 1 2 - \eps \ $ & $ \ -1 - \eps \ $ & $ \ \frac 1 2 - \eps \ $ & $ \ -\eps \ $ & $ \ -\frac 1 2 - \eps \ $ & $ \ -\eps \ $
\\
\bottomrule
\end{tabular}
\bigskip
\caption{The list of relevant diagrams, together with their regularity exponent, where $\eps > 0$ is arbitrary. }
\label{tab:diag}
\end{table}
}

\smallskip

The two remaining ambiguous terms in our formal derivation, namely $\<1> (\<30>)^2$ and $\<30> \,  \<1>$, can be defined classically in terms of the more fundamental object $\<31p>$. For $\<30> \,  \<1>$, we can set
$$
\<30> \, \<1> := \<30> \pne \<1> + \<31p>.
$$

As for  $ \<1> \,  (\<30>)^2$, we only need to define $\<1> \pe (\<30>)^2$. 
This term can be formally decomposed into
$$
2 \, \<1> \pe \Ll[ \<30> \pl \<30> \Rr]  + \<1> \pe \Ll[ \<30> \pe \<30> \Rr],
$$
and only the first term is ill-defined. The commutator
$$
[\pl,\pe](\<30>,\<30>,\<1>)
$$
is well-defined, and we can thus set
$$
\<1> \pe \Ll[ \<30> \pl \<30> \Rr] := \<30> \, \<31p> + [\pl,\pe](\<30>,\<30>,\<1>),
$$
that is,
$$
\<1> \,  (\<30>)^2 :=  \<1> \pne  (\<30>)^2 + \<1> \pe \Ll[ \<30> \pe \<30> \Rr] + 2 \, \<30> \, \<31p> + 2[\pl,\pe](\<30>,\<30>,\<1>).
$$
%
In this way, the coefficients $a_0$, $a_1$ and $a_2$ appearing in \eqref{e:defP} can be re-expressed  
as 
\begin{align*}
a_0 & = m(\<1> - \<30>)+ (\<30>)^3 - 3\Ll[ \<1> \pne  (\<30>)^2 + \<1> \pe \Ll[ \<30> \pe \<30> \Rr] + 2 \<30> \, \<31p> + 2[\pl,\pe](\<30>,\<30>,\<1>) \Rr] \\
& \qquad \qquad -9 \, \<30> \, \<22p> + 3\, \<32p>,\\
a_1 & = m + 6 \Ll[ \<30> \pne \<1> + \<31p> \Rr] - 3 (\<30>)^2 + 9\, \<22p>, \\
a_2 & = -3 \, \<1> + 3 \,  \<30>.
\end{align*}
Throughout the article, we will never make use of the explicit form of these coefficients, but only that they are of regularity $(-\frac 1 2)^-$.

\smallskip

A natural approach to construct the diagrams in \eqref{e:fund-diag} is via regularisation: if $\xi$ is replaced by a smooth approximation $\xi_\delta$, then these terms 
have a canonical interpretation:
One can define
$\widetilde{\<1>}_\delta$ as the solution to \eqref{e:lin} with $\xi$ replaced by $\xi_\delta$, $\widetilde{\<2>}_\delta := \widetilde{\<1>}_\delta^2$, $\ \widetilde{\<3>}_\delta := \widetilde{\<1>}_\delta^3$, and
$\widetilde{\<20>}_\delta$ and $\widetilde{\<30>}_\delta$ as solutions of \eqref{e:def20} and \eqref{e:def30} with right hand sides $\widetilde{\<2>}_\delta$ and $\widetilde{\<3>}_\delta$.
Furthermore, one can then define $\widetilde{\<31p>}_\delta = \widetilde{\<30>}_\delta \pe \widetilde{\<1>}_\delta$, $\widetilde{ \<32p>}_\delta :=  \widetilde{ \<30>}_\delta \pe \widetilde{\<2>}_\delta$ and 
$\widetilde{\<22p>}_\delta:=  \widetilde{ \<20>}_\delta \pe \widetilde{\<2>}_\delta$. Finally, if 
$(\widetilde{v}_\delta,\widetilde{w}_\delta)$ solves \eqref{e:eqvw}, with diagrams interpreted in this way, then indeed, $\widetilde{X}_\delta = \widetilde{\<1>}_\delta - \widetilde{\<30>}_\delta + \widetilde v_\delta + \widetilde w_\delta$
solves \eqref{e:eqX} (with $\xi$ replaced by $\xi_\delta$). 

\smallskip

However, these ``canonical'' diagrams fail to converge as the regularisation parameter~$\delta$ is sent to zero. Given 
their low regularity, this is not surprising. Yet, the first striking fact about renormalisation is that these terms \emph{do} converge in the relevant spaces if they are modified in a rather mild way. Indeed, if we set
\begin{align*}
\<1>_\delta = \widetilde{\<1>}_\delta, \qquad \<2>_\delta = \widetilde{\<2>}_\delta-C_\delta^{(1)}, \qquad  \<3>_\delta = \widetilde{\<3>}_\delta-3 C_\delta^{(1)}  \widetilde{\<1>}_\delta ,
\end{align*}
for a suitable choice of diverging constant $C_\delta^{(1)}$, then define $\<20>_\delta$ and $\<30>_\delta$ as solutions of \eqref{e:def20} and \eqref{e:def30} with right hand sides $\<2>_\delta$ and $\<3>_\delta$,
and finally 
\begin{align*}
\<31p>_\delta =\<30>_\delta \pe  \<1>_\delta , \qquad  \<32p>_\delta :=   \<30>_\delta \pe  \<2>_\delta - 3 C^{(2)}_\delta \<1>_\delta, \qquad  
\<22p>_\delta:=   \<20>_\delta \pe \<2>_\delta - C^{(2)}_\delta.
\end{align*}
for another choice of diverging constant $C^{(2)}_\delta$, then these terms converge to non-trivial limiting objects. This is shown in \cite{catcho}, and a very similar result is already contained in \cite[Sec.\ 10]{Martin1} (see also \cite{mourrat2016construction} for a pedagogical presentation of these calculations).
We stress once more that these results rely heavily on explicit calculations involving variances of the terms involved, which allow to capture stochastic cancellations. 

\smallskip

The second striking fact is that the ``renormalisation'' of these diagrams translates into a simple 
transformation of the original equation. Indeed, if $(v_\delta,w_\delta)$ solves \eqref{e:eqvw}, with diagrams interpreted in the renormalised way, then  $X_\delta = \<1>_\delta - \<30>_\delta + v_\delta + w_\delta$
solves the identical equation \eqref{e:eqX}, with $\xi$ replaced by $\xi_\delta$ but with \emph{renormalised massive term} $m_\delta := m + 3C_{\delta}^{(1)}- 9 C_\delta^{(2)} $. 
Since the solution theory for \eqref{e:eqvw} is stable under convergence of the diagrams, we can conclude that the solution $X_\delta$ to this renormalised equation does 
converge to a non-trivial limit, denoted by $X$, as $\de$ tends to~$0$. 

\smallskip

The fact that we have modified the equation we intended to solve may be discomforting at first. 
That this modification is the ``correct'' one is ultimately justified by the fact that the solutions thus defined are indeed the physically relevant ones.
In particular, these solutions arise as scaling limits of models of statistical mechanics near criticality. The connexion between renormalised fields and statistical mechanics has been studied at least since the 60s (see e.g. \cite{GlimmJaffeBook,GuerraRosenSimon,GlimmJaffe}
 and the references therein). We showed in \cite{Ising} that the $\Phi^4_2$ model can be obtained as the scaling limit of Ising-Kac models near criticality, as anticipated in \cite{GLP}. Related results were obtained for the KPZ equation, first in \cite{bergia} via a Cole-Hopf transformation, and then, following \cite{MartinKPZ}, in a series of works including
  \cite{gonjar,FunakiQuastel,MartinHao,KPZReloaded,EnergyUnique,HairerQuastel}. See also the survey articles \cite{HairerSurvey,LectureNotesAjay} for a summary of the work on the $\Phi^4$ model 
  with regularity structures.

\subsection{Main result}

Our aim is to derive an a priori bound on solutions of \eqref{e:eqX}. 
We will only be concerned with the analysis of the deterministic system. 
Before we do so, we make a modification to the system \eqref{e:eqvw}. We give ourselves a (large) constant $c \ge 0$, and consider instead the system
\begin{equation}
\label{e:eqvwc}
\Ll\{
\begin{array}{lll}
(\dr_t - \Dd) v & = & F(v+w) - c v,\\
(\dr_t - \Dd) w & = & G(v,w) + c v,
\end{array}
\right.
\end{equation}
with $F$ and $G$ as in \eqref{e:defF} and \eqref{e:defG} respectively, and with initial condition 
\begin{equation}
\label{e:init}
v(0) = v_0, \qquad w(0) = w_0.
\end{equation}
Naturally, this modification changes the definitions of $v$ and $w$, but we stress that it does \emph{not} change the sum $v + w$, and therefore the final solution $X$. This can easily be seen on the level of the regularised solution  $(v_\delta, w_\delta)$ discussed in the previous section. Since $(v,w)$ is the limit of the $(v_\delta,w_\de)$,
it follows that $v+w$ itself does not depend on the choice of $c$. Therefore, it is ultimately enough to show the  existence of a constant $c$ for which the a priori bound holds.
 For the same reason, the solution $X$ depends on $v_0$ and $w_0$ only through the sum $v_0 + w_0$. 
 
 \smallskip
 
 We seek solutions to \eqref{e:eqvwc} in the space $\mathfrak{X}$ defined as the set of pairs $(v,w)$ in 
 \begin{align*}
& \Ll[ 
	C\Ll([0,1] , \mathcal{B}_\infty^{-\frac35}\Rr) \cap 
	C\Ll((0,1] , \mathcal{B}_\infty^{\frac12+2\eps}\Rr) 
	 \cap C^{\frac18}\Ll((0,1],L^{\infty}\Rr) 
\Rr] \\
& \qquad 
\times
\Ll[ 
	C\Ll([0,1] , \mathcal{B}_\infty^{-\frac35}\Rr) \cap 
	C\Ll((0,1] , \mathcal{B}_\infty^{1+2\eps}\Rr)  \cap 
	C^{\frac18}\Ll((0,1],L^{\infty}\Rr) 
\Rr] 
\end{align*}
for which the norm
 \begin{align}
\notag
&\| (v,w) \|_{\mathfrak{X}} \\
\notag
&:= \max \bigg\{
\sup_{0 \leq t \leq 1} \|v(t) \|_{\B^{-\frac35}_\infty},  
\sup_{0 < t \leq 1} t^{\frac{3}{5}} \|v(t) \|_{\B^{\frac12+2\eps}_\infty},  
\sup_{0 < s<t \leq 1} s^{\frac12} \frac{\|v(t) - v(s) \|_{L^{\infty}}}{ |t-s|^{\frac18}} , \\
\notag
& \qquad  
\sup_{0 \leq t \leq 1} \|w(t) \|_{\B^{-\frac{3}{5}}_{\infty}},  
\sup_{0 < t \leq 1} t^{\frac{17}{20}} \|w(t) \|_{\B^{1+2\eps}_\infty},
\sup_{0 < s<t \leq 1} s^{\frac12} \frac{\|w(t) - w(s) \|_{L^{\infty}}}{ |t-s|^{\frac18}}  
\bigg\}
\end{align}
is finite.
%
%
%
Here is our main result. 

\begin{thm}[global existence and a priori bound]
\label{t:apriori}
For each $p \in [24,\infty)$ and $\eps > 0$ sufficiently small, there exist constants $c_0 < \infty$ (depending only on $p$), $C < \infty$ and an exponent $\kappa < \infty$ such that the following holds. Let $K \in [1,\infty)$ and let $\<1>$, $ \<2>$, $ \<30>$, $ \<31p>$, $ \<32p>$, $ \<22p>$ be distributions such that for every pair $(\tau,\al_\tau)$ as in Table~\ref{tab:diag}, we have 
\begin{equation*}  
\tau \in C([0,1],\B_\infty^{\al_\tau}), \qquad \sup_{0 \le t \le 1} \| \tau(t) \|_{\B_\infty^{\al_\tau}} \le K,
\end{equation*}
as well as
\begin{equation*}
\sup_{0 \le s < t \le 1} \frac{\|\<30>(t) - \<30>(s)\|_{\B_\infty^{\frac 1 4 - \eps}}}{|t-s|^\frac 1 8} \le K.
\end{equation*}
Assume furthermore that the constant $c$ in \eqref{e:eqvwc} is chosen according to
\begin{equation}
\label{e.main.bound.c}
c = c_0 K^{30p}.
\end{equation}
We set $v_0 := 0$. For every $w_0 \in \B^{-\frac 3 5 }_{\infty}$, there exists a unique pair $(v,w) \in \mathfrak X$ solution to \eqref{e:eqvwc}-\eqref{e:init}.
Moreover, for every $t \in (0,1]$, we have
\begin{equation*}
\| w(t) \|_{L^{3p-2}} \le  \frac{C K^{\kappa}}{\sqrt{t}} \quad \text{and} \quad \| v(t) \|_{\B^{-3\eps}_{2p}} \le C K^\kappa.
\end{equation*}
\end{thm}
%
%
We now explain how to apply this result to the renormalized solution of \eqref{e:eqX}. Note first that the diagrams based on the solution to \eqref{e:lin} unfortunately do not satisfy uniform bounds such as, for every $p < \infty$,
\begin{equation}  
\label{e.random.diag1}
\sup_{s \ge 0} \, \E \Ll[\sup_{s \le t \le s+1} \| \tau(t) \|_{\B_\infty^{\al_\tau}}^p \Rr] < \infty,
\end{equation}
or 
\begin{equation}
\label{e.random.diag2}
\sup_{r \ge 0} \, \E \Ll[\sup_{r \le s, t \le r+1} \frac{\|\<30>(t) - \<30>(s)\|^p_{\B_\infty^{\frac 1 4 - \eps}}}{|t-s|^\frac p 8}\Rr] < \infty.
\end{equation}
However, this problem is very simple to solve: it suffices to add a massive term to the linear equation \eqref{e:lin}, that is, to redefine $\<1>$ as the solution to
\begin{equation*}  
(\partial_t - \Delta + 1)\<1> = \xi.
\end{equation*}
The addition of a massive term in the definition of the diagrams only modifies the system \eqref{e:eqvwc} very superficially, and it is elementary to verify that Theorem~\ref{t:apriori} also applies to this modified system. Moreover, the diagrams defined with a massive term do satisfy \eqref{e.random.diag1}-\eqref{e.random.diag2} for every $p < \infty$. Indeed, this is an elementary extension of the results of Catellier and Chouk~\cite{catcho}; see also \cite[Sec.\ 10]{Martin1} and \cite{mourrat2016construction}. We can then apply Theorem~\ref{t:apriori} iteratively to construct a solution to \eqref{e:eqX} over $[0,\infty)$ as follows. We first apply Theorem~\ref{t:apriori} to define $X =  \<1> - \<30>  + v+w$, where $(v,w)$ solves \eqref{e:eqvwc} with $c$ sufficiently large and $v_0 = 0, w_0 = X_0 \in \B^{-\frac 1 2 - \eps}_\infty$. This defines $X$ up to time~$1$, and ensures that for every $p < \infty$,  
\begin{equation*}
\E \Ll[\sup_{0< t\le 1 } \, \sup_{X_0 \in \B^{-\frac12-\eps}_\infty} \, \Ll( \sqrt{t} \,  \| X(t)  \|_{\B_{\infty}^{-\frac12-\eps}}\Rr)^p \Rr] < \infty,
\end{equation*}
since for $p \ge 24$, the spaces $L^{3p-2}$ and $\B^{-3\eps}_{2p}$ are continuously embedded in $\B^{-\frac 1 2 - \eps}_\infty$, see Proposition~\ref{p:embed} and Remark~\ref{r:Besov-vs-Lp}. 
We then apply Theorem~\ref{t:apriori} iteratively at times $t_1 = \frac 1 2$, $t_2 = 1$, etc, each time with the new initial condition $v(t_k) = 0$ and $w(t_k)$ given by the sum of the $v$ and $w$ at time $t_k$ obtained from the previous iteration. Recall that this reallocation of the initial condition does not change the sum $v + w$; nor does a modification of the value of the constant $c$ change this sum. We thus obtain a solution $X$ over $[0,\infty)$ which satisfies, for every $p < \infty$,
\begin{equation}
\label{e:final-bound-expectation}
\sup_{s \ge 0} \, \E \Ll[\sup_{s< t\le s+1 } \, \sup_{X_0 \in \B^{-\frac12-\eps}_\infty} \,  \Ll((\sqrt{t} \wedge 1) \| X(t)  \|_{\B_{\infty}^{-\frac12-\eps}}\Rr)^p \Rr] < \infty.
\end{equation}
This bound can then be used as the basis for a 
 Krylov-Bogoliubov procedure for the construction of an invariant measure, see
 \cite[Section 4]{tsatsoulis2016spectral} for the implementation of this argument in the case of the two-dimensional torus.

\smallskip

The two-dimensional analysis in \cite{tsatsoulis2016spectral} actually yields a stronger statement, namely the exponential convergence to equilibrium with respect to the total variation norm, uniformly over all initial data. The key ingredients are a non-linear dissipative bound akin to \eqref{e:final-bound-expectation}, complemented by the strong Feller property as well as a support theorem. The strong Feller property for \eqref{e:eqX} has in the meantime been established in \cite{HairerMattingly} in the framework of regularity structures, and a support theorem is part of the forthcoming work \cite{HairerSchonbauer}. We expect that the combination of our main result with these two additional ingredients will indeed imply exponential convergence to equilibrium also in the three-dimensional case.

\begin{rem}
 In the simpler two-dimensional case, a comparable analysis  was performed in \cite{JCH}. There, we were able to push the analysis further and show global existence of solutions if the equation is 
 posed on the full space $\R^2$. The full-space setting is physically more relevant, but also more difficult to analyse, because the stochastic terms lack any decay at infinity, which mandates an analysis in weighted distribution spaces. It would be interesting to investigate whether the methods of~\cite{JCH} can be combined with those of the present article to yield a solution theory for the dynamical $\Phi^4$ equation in the full space $\R^3$. 
\end{rem}

\begin{rem}
\label{r.put.v0.zero}
At first glance, the choice of initial datum $(v_0,w_0) = (0,X_0)$ may seem surprising. However, we cannot
expect to obtain a strong non-linear dissipative bound  for the system~\eqref{e:eqvwc} uniformly over all initial data in, say $\B^{-\frac12-\eps}_\infty \times
\B^{-\frac12-\eps}_\infty $. As we are ultimately only interested in the sum $X= \<1> - \<30>+ v+w$, this does not impose any restrictions on the level of the process $X$. 
\end{rem}

\begin{rem}
Convergence of lattice approximations to \eqref{e:eqX} was shown in \cite{MartKonst} and \cite{ZhuTwins}. This was used in \cite{MartKonst} to implement an argument in the spirit of Bourgain's work on non-linear Schr\"odinger  equations (see e.g. \cite{Bourgain}) to show that for almost every initial datum with respect to the measure
\eqref{QFT-measure}, solutions to \eqref{e:eqX} do not explode. This result relies on the analysis of  the measure \eqref{QFT-measure} performed in  \cite{bfs}. It can then be upgraded using the strong Feller property shown in \cite{HairerMattingly} to  obtain the global well-posedness for any initial datum of suitable regularity. We stress however that the spirit of this method is completely different from the method presented here. There, a priori information on the invariant measures is used to rule out finite time blow-up of solutions. Our argument on the other hand relies only on the dynamics, and yields information on the invariant measure as a result. 
\end{rem}

\begin{rem}\label{remcatcho} 
The notion of solution derived in  \cite{catcho} is closely related to \eqref{e:eqvw}, but slightly different: there, our ansatz 
\begin{equation*}
X = \<1> - \<30> + v+w
\end{equation*}
is replaced by 
\begin{equation*}
X = \<1> - \<30> + \Phi' \pl \<20>+\Phi^{\sharp},
\end{equation*}
and a system of equations for $\Phi'$ and the remainder $\Phi^{\sharp}$ is solved. The term $\Phi' \pl \<20>$ in this decomposition
corresponds to $v$ up to a commutator term. Although these approaches are very similar, ours makes the equations solved by $v$ and $w$ more explicit.
\end{rem}

%

\subsection{Sketch of proof and organisation of the paper}
We present a local existence and uniqueness result based on a Picard iteration in Section~\ref{s:loc}. This result is essentially contained in \cite{catcho}, although we use slightly different norms (see also Remark~\ref{remcatcho}). 
The bulk of our argument is contained in Sections \ref{s.apriori-v} to \ref{s:conc}, and we now proceed to explain the strategy. We start by recalling the deterministic argument we aim to mimic. If $X$ solves the deterministic PDE
\begin{align*}
\Ll\{
\begin{array}{lll}
\partial_t X - \Delta X &=& -X^3 + \cdots,\\
X(0, \cdot) &=& X_0,
\end{array}
\Rr.
\end{align*}
where $\cdots$ denotes a collection of lower order terms which is bounded, say in $L^{\infty}$, by $K \geq 1$, one can simply test the equation against $X^{3p-3}$ for an even integer $p$ to get the differential inequality
\begin{equation}
\label{e.ode.arg}
\partial_t \| X (t)\|_{L^{3p-2}}^{3p-2} + \| X(t) \|_{L^{3p}}^{3p} \lesssim \langle \cdots , X^{3p-3} \rangle
\lesssim K \| X(t) \|_{L^{3p-3}}^{3p-3}.
\end{equation}
In fact, an additional ``good term'' $\|  X^{3p-3} |\nabla X|^2(t)\|_{L^1}$ which comes from the Laplacian $-\Delta X$  on the left-hand side also appears, but we can choose to ignore it. By Young's and Jensen's inequalities, the term $\| X(t) \|_{L^{3p-3}}^{3p-3}$ on the right-hand side of the inequality above can 
be absorbed into the term $\| X(t) \|_{L^{3p}}^{3p}$, and then a simple comparison argument for ODEs
yields that for every $t > 0$,
\[
\| X(t) \|_{L^{3p-2}} \lesssim t^{-\frac12}+K^{\frac 1 3}.
\]
This bound is uniform over all initial data $X_0$. A yet simpler manifestation of this phenomenon is the well-known fact that solutions of the ODE $\dot x = - x^{3}$ satisfy
$x(t) \leq (2t)^{-\frac12}$ uniformly over all initial data. 

\smallskip

We aim to implement a similar testing argument for the system \eqref{e:eqvwc}, which we 
restate here in the form
\begin{align}
\label{e:eqvwcv}
(\dr_t - \Dd+c) v & =  -3(v+w- \<30> ) \pl \<2> ,\\
\label{e:eqvwcw}
(\dr_t - \Dd) w & =   -(v+w)^3  - 3  \com_1(v,w) \pe \<2>    -3w \pe \<2>\\
\notag
&  \qquad \qquad +a_2(v+w)^2 +\ldots,
\end{align}
where we use the suggestive  convention to write
\begin{equation*}
\ldots = -3 \msf{com}_2(v+w)   - 3(v+w-\<30>) \pg \<2> + a_0 +a_1(v+w) + c v
\end{equation*}
for a collection of lower order terms which do not cause any particular difficulty in the analysis.
One quickly realises that the testing must be performed on the level of the 
equation for $w$. First of all, it is where the ``good'' cubic term, which 
is the crucial ingredient for the testing, appears. Second, testing the equation 
\eqref{e:eqvwcv} against $v$ would produce a ``good'' term proportional to
 $\| \nabla v(t) \|_{L^2}^2$ on the 
left hand side, but this term is infinite, since the best regularity exponent we can expect for $v$ is below $1$.
Moreover, as already hinted at in Remark~\ref{r.put.v0.zero}, since the damping terms $(-\Delta + c) v$ in \eqref{e:eqvwcv} are linear, we cannot expect $v$ to relax to equilibrium faster than exponentially.  
This motivates our choice of initial condition $v_0 = 0$ (although in several steps of the argument, it will be useful to estimate the behaviour of $v$ for arbitrary initial datum $v_0$).

\smallskip

We proceed to test the equation for 
$w$ against $w^{3p-3}$ for some large even integer~$p$. Ideally, we would like to get a closed
expression which permits to invoke an ODE comparison argument similar to the one sketched below \eqref{e.ode.arg}. However,  several problems present themselves. First, the equations for $v$ and $w$ are coupled, so we need to estimate the influence of $v$ on $w$ and vice versa.
Second, even if we controlled all terms involving $v$ on the right-hand side of
\eqref{e:eqvwcw}, the testing would not lead to a closed expression:
several terms involve higher order regularity information on $w$ which 
is not controlled by the ``good'' term $\| w^{3p-2} |\nabla w|^2 \|_{L^1}$ appearing when testing the equation.
These are the terms left explicit on the right-hand side of 
\eqref{e:eqvwcw}, namely the terms
$\msf{com}_1(v,w) \pe \<2>$  and  $w \pe \<2>$.
Indeed, the estimation 
of $\msf{com}_1(v,w) \pe \<2>$ requires information on the time regularity 
of $w$, while the term $w \pe \<2>$ requires to control at least $1+2\eps$
derivatives of $w$. The quadratic  term $a_2(v+w)^2$ also requires 
some care because it calls for a control of $\frac12+2\eps$ derivatives of $v$ 
and $w$, and it is quadratic rather than linear.

\smallskip

A bound on $v$  is presented in Section \ref{s.apriori-v}. The key observation is that although the terms on the right-hand side of \eqref{e:eqvwcv} contain paraproducts with $\<2>$, solutions are relatively easy to control, because both $v$ and $w$ only appear linearly. We thus use a Gronwall-type lemma to obtain several estimates on $v$ in terms of the initial datum $v_0$ and $w$.
 These estimates are used in the following sections to replace all expressions involving $v$ when manipulating the equation for $w$. 
The extra massive term $-cv$ appearing on the left-hand side of \eqref{e:eqvwcv} permits to get small constants in this argument. This feature is crucially useful in the testing argument to show that when testing $(v+w)^3$ 
against $w^{3p-3}$, the terms involving $v$ are dominated by the ``good term'' $\| w\|_{L^{3p}}^{3p}$; see Lemma~\ref{lem:cubic2}.

\smallskip

In order to address the appearance of higher regularity norms of $w$ in the testing argument, which 
ultimately controls the large scale behaviour of solutions, we use parabolic regularity  estimates. More precisely, the mild form of the equation is used in Sections~\ref{s:apriori-dw} and \ref{s:Gronwall-w}
to derive bounds on $\delta_{st} w := w(t) -w(s)$ and $\| w(t) \|_{\B^{\gamma}_p}$ for some $\gamma >1$. Both sections aim to control the ``small-scale behaviour'' of solutions, and thus it is 
natural that the ``good sign'' of the cubic non-linearity is not used in these sections. The bounds on $v$  derived in Section~\ref{s.apriori-v} are used in 
these two sections to replace terms involving $v$ by terms involving $w$. In the end, both 
$\sup_{s \neq t} |t-s|^{-\frac18}\| \delta_{st} w \| $ and $\| w(t) \|_{\B^{\gamma}_p}$ can be bounded in terms of 
\begin{align*}
\Big( \int_0^t \| w(s) \|_{L^{3p}}^{3p}  \d s \Big)^{\frac{1}{p}} , \;
\Big(\int_0^t \| w(s) \|_{\B^{1+4\eps}_p}^p \d s\Big)^{\frac{1}{ p}}, \; 
\| v_0 \|_{\B_{2p}^{-3\eps}}^3 , 
\end{align*}
as well as a suitable norm for $w_0$. In Section~\ref{s:testing-w}, the equation for $w$ is tested against $w^{3p-3}$. We use the bounds on $v$ and $\de_{st} w$ from Sections~\ref{s.apriori-v} and \ref{s:apriori-dw} systematically to obtain a bound on $\int_s^t \| w(r )\|_{L^{3p}}^{3p} \d r$. 

\smallskip

In the concluding Section~\ref{s:conc}, this bound is  combined with the higher regularity bound on $\| w(t) \|_{\B^{\gamma}_p}$ from Section~\ref{s:Gronwall-w} to finish the proof of our main result, Theorem~\ref{t:apriori}. 
 We first derive a self-contained bound on quantities involving $w$, see Lemma~\ref{l.key.bound.sec8}. In this estimate, some norm of $v$ appears on the right-hand side. In order to conclude by mimiking the ODE argument explained below \eqref{e.ode.arg}, we rely on the assumption that 
 $v_0 = 0$. This is the only place where this assumption is used. We apply the estimate from Lemma~\ref{l.key.bound.sec8} up to the first time $\tau$ such that
$\| v(\tau)  \|_{\B^{-3\eps}_{2p}}$ exceeds a suitable norm of $w$. This argument then yields the desired estimate on $w(t)$ for all $t\leq \tau$. In order to remove the restriction on times to be less than $\tau$, we use that $t^{-\frac12}$ is integrable and Theorem~\ref{t:apriori-v} to get a bound on $\|v(\tau) \|_{\B^{-3\eps}_{2p}}$,  and thus deduce that suitable norms of $w(\tau)$ must be small (irrespectively of the possible smallness of $\tau$). This final part of the argument only works if $v$ is measured in a 
low regularity norm (we work with $\| \cdot \|_{\B^{-3\eps}_{2p}}$; see \eqref{e.now.is.time.to.understand}) and this is the reason why throughout the paper we measure the initial datum of the equation for $v$ in this norm.

%
%
%
%
%
\section{Local existence and uniqueness}
\label{s:loc}


The aim of this section is to provide a local existence and uniqueness result for the system \eqref{e:eqvwc}.
A similar local theory was already presented in \cite{catcho} in a slightly different formulation (see Remark~\ref{remcatcho}).
 The value of the constant $c$  plays no role for the results presented in this section.

\smallskip
We interpret the system \eqref{e:eqvwc} in the mild sense:
\begin{align}
\label{e:mild-v}
v(t) & =e^{t (\Dd - c)} v_0 + \int_0^t e^{(t-s) (\Dd - c)} F(v(s)+w(s)) \, \d s,\\
\label{e:mild-w}
w(t) & =e^{t \Dd} w_0 + \int_0^t e^{(t-s) \Dd} [G(v(s),w(s)) + c v(s)] \, \d s,
\end{align}
and assume our initial condition $(v_0,w_0) \in \B_{\infty}^{-\frac35} \times \B_{\infty}^{-\frac35}$. (This choice is somewhat arbitrary. Any initial condition of regularity strictly better than $-\frac23$ would work.)
For $T\in (0,1]$,  we define $\mathfrak{X}_{T} $ as the space of pairs $(v,w)$ in
\begin{align*}
&\big[ 
	C([0,T] , \mathcal{B}_\infty^{-\frac35}) \cap 
	C((0,T] , \mathcal{B}_\infty^{\frac12+2\eps}) 
	 \cap C^{\frac18}((0,T],L^{\infty}) 
\big] \\
& \qquad 
\times
\big[ 
	C([0,T] , \mathcal{B}_\infty^{-\frac35}) \cap 
	C((0,T] , \mathcal{B}_\infty^{1+2\eps})  \cap 
	C^{\frac18}((0,T],L^{\infty})
\big]
\end{align*}
for which the norm
\begin{align}
\label{e:def:XMT}
&\| (v,w) \|_{\mathfrak{X}_T} \\
\notag
&:= \max \bigg\{
\sup_{0 \leq t \leq T} \|v(t) \|_{\B^{-\frac35}_\infty},  
\sup_{0 < t \leq T} t^{\frac{3}{5}} \|v(t) \|_{\B^{\frac12+2\eps}_\infty},  
\sup_{0 < s<t \leq T} s^{\frac12} \frac{\|v(t) - v(s) \|_{L^{\infty}}}{ |t-s|^{\frac18}} , \\
\notag
& \qquad  
\sup_{0 \leq t \leq T} \|w(t) \|_{\B^{-\frac{3}{5}}_{\infty}},  
\sup_{0 < t \leq T} t^{\frac{17}{20}} \|w(t) \|_{\B^{1+2\eps}_\infty},
\sup_{0 < s<t \leq T} s^{\frac12} \frac{\|w(t) - w(s) \|_{L^{\infty}}}{ |t-s|^{\frac18}}  
\bigg\}
\end{align}
is finite.
The main result of this section is the following.
\begin{thm}\label{thm:local-theory}
Let $\eps > 0$ be sufficiently small, let $K \in [1,\infty)$, and let $\<1>$, $ \<2>$, $ \<30>$, $ \<31p>$, $ \<32p>$, $ \<22p>$ be distributions such that for every pair $(\tau,\al_\tau)$ as in Table~\ref{tab:diag}, we have 
\begin{equation}
\label{e:hyp_diag}
\tau \in C([0,1],\B_\infty^{\al_\tau}), \qquad \sup_{0 \le t \le 1} \| \tau(t) \|_{\B_\infty^{\al_\tau}}  \le K
\end{equation}
as well as 
\begin{equation}
\label{e:hyp_diag2}
\sup_{0 \le s, t \le 1} \frac{\|\<30>(t) - \<30>(s)\|_{\B_\infty^{\frac 1 4 - \eps}}}{|t-s|^\frac 1 8} \le K.
\end{equation}
\emph{(1)} For every pair of initial conditions $(v_0, w_0)\in \B_{\infty}^{-\frac35} \times \B_\infty^{-\frac35}$, 
there exists $T^{\star} \in (0, 1]$ such that the system \eqref{e:mild-v}--\eqref{e:mild-w} has a solution $(v,w)$
defined on $[0,T^{\star})$.
This time $T^{\star}$ can be chosen maximal, 
in the sense that either the solution is global, i.e. $T^{\star}=1$ and the solution can be extend to time $t=1$ and 
takes values in $\mathfrak{X}_1$, or  
 $\lim_{t \uparrow T^{\star}}  \| v(t) \|_{\B_{\infty}^{-\frac35}} \vee \| w(t) \|_{\B_\infty^{-\frac35}}  = \infty$,
 in which case the restriction of $(v,w)$ to any compact interval $[0,T] \subseteq [0,T^{\star})$ takes values in $\mathfrak{X}_{T}$. The choice of maximal existence time $T^{\star}$ and solution $(v,w)$ with these properties is unique.

\noindent \emph{(2)} If $(v_0,w_0) \in \B^{\frac12+2\eps}_\infty \times \B^{1+2\eps}_{\infty}$, then the solution pair $(v,w)$ constructed in (1) is continuous at the initial time, in the sense that $\mathfrak{X}_T$ in the above statement can 
be replaced by
\begin{align*}
\overline{\mathfrak{X}}_{T} = &
\big[ 
	C([0,T^{\star}] , \mathcal{B}_\infty^{\frac12+2\eps}) 
	 \cap C^{\frac18}([0,T^{\star}],L^{\infty}) 
\big] \\
&
\times
\big[ 
	C([0,T^{\star}] , \mathcal{B}_\infty^{1+2\eps})  \cap 
	C^{\frac18}([0,T^{\star}],L^{\infty})
\big].
\end{align*}
\end{thm}

We start by isolating a bound on the commutator $\msf{com}_1$ defined in \eqref{e:def:com1}, which we will use again in subsequent sections. We introduce the difference operator
\begin{equation}
\label{e:def:dst}
\de_{st} f := f(t) - f(s).
\end{equation}
\begin{prop}[First commutator estimate]
\label{p:com1}
Let $\eps > 0$, $\be \in (4\eps,1+2\eps]$, $p \in [1,\infty]$ and $T > 0$. Under the assumption \eqref{e:hyp_diag}--\eqref{e:hyp_diag2}, we have for every $(v,w) \in \mathfrak{X}_T$ and $t \in [0, T)$,
\begin{align*}
 \|\msf{com}_1(v,w)(t) - e^{t\Delta}v_0\|_{\B^{1+2\eps}_p} 
& \ls  K^2
%
 + \int_0^t \frac{K}{(t-s)^{1+2\eps-\frac \be 2 }} \Ll(\|v(s)\|_{\B^\be_p} + \|w(s)\|_{\B^\be_p} \Rr) \, \d s \\
& \quad + \int_0^t \frac{K}{(t-s)^{1+2\eps}} \|\de_{st} (v+w)\|_{L^p}  \, \d s,
\end{align*}
where the implicit multiplicative constant depends on $\eps$ and $p$.
\end{prop}
\begin{proof}
Recall the definition of $\com_1$ in \eqref{e:def:com1}. 
We introduce the commutation operator
\begin{equation}
\label{e:def:comm2}
[e^{t\Dd},\pl] : (f,g) \mapsto e^{t\Dd}(f\pl g) - f \pl(e^{t\Dd} g),
\end{equation}
so that
\begin{multline}
\label{e.decomp-somename}
e^{(t-s)\Dd} [(v+w-\<30>) \pl \<2>](s) \\ = (v+w-\<30>)(s) \pl \Ll[ e^{(t-s)\Dd} \, \<2>(s) \Rr] + [e^{(t-s)\Dd},\pl]\Ll((v+w-\<30>)(s),\<2>(s)\Rr).
\end{multline}
We start by estimating the last term in the sum above. The contribution of $\<30>$ can be estimated using Proposition~\ref{p:comm2}:
\begin{align*}
\Ll\|\int_0^t [e^{(t-s)\Dd},\pl]\Ll(\<30>(s),\<2>(s)\Rr) \, \d s \Rr\|_{\B^{1+2\eps}_{p}} & \ls \int_0^t \Ll\|[e^{(t-s)\Dd},\pl]\Ll(\<30>(s),\<2>(s)\Rr) \Rr\|_{\B^{1+2\eps}_p} \, \d s \\
& \ls \int_0^t \frac {K^2} {(t-s)^{\frac 3 4 + 2 \eps}} \, \d s \ls K^2.
\end{align*}
By the same reasoning, we have
\begin{equation*}
\Ll\|\int_0^t [e^{(t-s)\Dd},\pl]\Ll((v+w)(s),\<2>(s)\Rr)  \, \d s \Rr\|_{\B^{1+2\eps}_{p}} \ls 
\int_0^t \frac{K}{(t-s)^{\frac{2 + 3\eps -\be}{2}}} \|(v+w)(s)\|_{\B^{\be}_p} \, \d s.
\end{equation*}

We now turn to the first term in the \rhs\ of \eqref{e.decomp-somename}, which we will combine with the last term in \eqref{e:def:com1}. 
Recalling \eqref{e:def20bis},
we observe that
\begin{multline*}
 [(v+w-\<30>) \pl \<20>](t) - \int_0^t (v+w-\<30>)(s) \pl \Ll[ e^{(t-s)\Dd} \, \<2>(s) \Rr]  \, \d s  \\
 = \int_0^t \Ll[\de_{st}(v+w-\<30>)\Rr] \pl \Ll[ e^{(t-s)\Dd} \, \<2>(s) \Rr]  \, \d s  .
\end{multline*}
By Proposition~\ref{p:mult}, the $\|\cdot\|_{\B^{1+2\eps}_p}$ norm of the integral above is bounded by a constant times
$$
\int_0^t \|\de_{st} (v+w-\<30>)\|_{L^p} \,  \|e^{(t-s)\Dd}\<2>(s)\|_{\B^{1+2\eps}_\infty} \, \d s \lesssim \int_0^t \frac{K}{(t-s)^{1 + \frac{3\eps}{2}}} \|\de_{st} (v+w-\<30>)\|_{L^p}  \, \d s,
$$
where we used Proposition~\ref{p:smooth-besov} and the fact that $\|\<2>(s)\|_{\B^{-1-\eps}_\infty} \ls 1$ in the last step. By the assumption of H\"older regularity in time on $\<30>$ (with exponent $\frac 1 8$), this last integral is bounded by a constant times
\begin{equation*}
K^2 + \int_0^t \frac{K}{(t-s)^{1 +\frac{3\eps}{2}}} \|\de_{st} (v+w)\|_{L^p}  \, \d s,
\end{equation*}
which completes the proof.
\end{proof}

\begin{proof}[Proof of Theorem~\ref{thm:local-theory}]
We follow the usual strategy to first solve the system for some small but strictly positive $T\in (0,1]$ using a Picard iteration. In a second step, solutions are  restarted iteratively 
to obtain maximal solutions. 

\smallskip

For every $T>0$ and $M > 0$, we define the ball
\begin{align*}
\mathfrak{X}_{T,M} := \{ (v,w) \in \mathfrak{X}_T \colon \| (v,w) \|_{\mathfrak{X}_T} \leq M  \}.
\end{align*}
For dealing with the case of regular initial data we also introduce the ball 
\begin{align*}
\overline{\mathfrak{X}}_{T,M} := \{ (v,w) \in \overline{\mathfrak{X}}_T \colon \| (v,w) \|_{\overline{\mathfrak{X}}_T} \leq M  \},
\end{align*}
where $\| (v,w) \|_{\overline{\mathfrak{X}}_T}$  is defined in an analogous way to $\| (v,w) \|_{\mathfrak{X}_T} $without allowing for blow-up near time $0$, i.e.
\begin{align*}
&\| (v,w) \|_{\overline{\mathfrak{X}}_T} \\
&\qquad := \max \bigg\{ 
\sup_{0 \leq t \leq T} \|v(t) \|_{\B^{\frac12+2\eps}_\infty},  
\sup_{0 \leq s < t \leq T} \frac{\|v(t) - v(s) \|_{L^{\infty}}}{ |t-s|^{\frac18}} , \\
& \qquad  \qquad
\sup_{0 \leq t \leq T}\|w(t) \|_{\B^{1+2\eps}_\infty}
\sup_{0 \leq s< t \leq T} \frac{\|w(t) - w(s) \|_{L^{\infty}}}{ |t-s|^{\frac18}}  
\bigg\}.
\end{align*}
Furthermore, we denote by $\Psi$ the fixed point map, i.e.\ the mapping which associates to $(v,w) \in \mathfrak{X}_T$ the function $t \mapsto (\Psi^V[v,w], \Psi^W[v,w])(t)$, where
\begin{align*}
\Psi^V[v,w](t) & = e^{t (\Dd - c)} v_0 + \int_0^t e^{(t-s) (\Dd - c)} F(v(s)+w(s)) \, \d s,\\
\Psi^W[v,w](t) & = e^{t \Dd  } w_0 + \int_0^t e^{(t-s) \Dd } \big( G(v(s),w(s)) +cv(s)\big) \, \d s.
\end{align*}
We now show that for a suitable $M$ and for $T$ small enough, $\Psi$ maps  the ball $\mathfrak{X}_{T,M}$ into
 itself and $\overline{\mathfrak{X}}_{T,M}$ into itself. The core ingredients are the following bounds, which we formulate as a lemma.
\begin{lem}\label{lem:loc-theory-lemma1} 
There exists a constant $C$ depending only on $c$ and $K$ (defined in the assumption of Theorem~\ref{thm:local-theory}) such that the following holds. For every 
\begin{equation}
\label{e.assumption.M}
M \ge \max\{1, \| v_0\|_{\B^{-\frac35}_\infty}, \| w_0\|_{\B^{-\frac35}_\infty} \} ,
\end{equation}
$T \in (0,1]$,  $(v,w) \in \mathfrak{X}_{M,T}$ and $s \in [0,T]$, we have
\begin{align}
\label{e:local-lemma1a}
\| F(v(s)+w(s)) \|_{\B^{-1 -\eps}_\infty} &\leq C M s^{-\frac{33}{100}}, \\
 \| G(v(s),w(s)) +cv(s)\|_{\B^{-\frac12 -2\eps}_\infty } &\leq C M^3 s^{-\frac{99}{100}} .
\label{e:local-lemma1} 
\end{align}
If 
\begin{equation*}  
M \ge \max \{1, \| v_0\|_{\B^{\frac 1 2 +2\eps}_\infty}, \| w_0\|_{\B^{1 + 2\eps}_\infty} \} 
\end{equation*}
and $(v,w) \in \overline{\mathfrak{X}}_{M,T}$, then the same bound holds without blow-up near zero, i.e.
\begin{align}
\label{e:local-lemma2a}
\| F(v(s)+w(s)) \|_{\B^{-1 -\eps}_\infty} &\leq C M\\
 \| G(v(s),w(s)) +cv(s) \|_{\B^{-\frac12 -2\eps}_\infty } &\leq C M^3  .
\label{e:local-lemma2} 
\end{align}
\end{lem}
We momentarily admit this lemma, and first use it to establish that $\Psi$ is a contraction from $\mathfrak{X}_{T,M}$ into itself, and also from $\overline{\mathfrak{X}}_{T,M}$ into itself. We focus on the statement concerning $\mathfrak{X}_{T,M}$, the proof for $\overline{\mathfrak{X}}_{T,M}$ using \eqref{e:local-lemma2a}-\eqref{e:local-lemma2} instead of \eqref{e:local-lemma1a}-\eqref{e:local-lemma1} being similar, only simpler.

\smallskip

We start by deriving bounds on $\Psi^V$. For every $M$ satisfying \eqref{e.assumption.M}, using Proposition~\ref{p:smooth-besov} and \eqref{e:local-lemma1a}, we get that for every $t \leq T$ and $\beta \in \{-\frac35, \frac12+2\eps \}$,
\begin{align*}
\| \Psi^V[v,w](t) \|_{\B_\infty^{\beta}} &\leq  \| e^{t (\Dd - c)} v_0\|_{\B_\infty^{\beta}} + \int_0^t \frac{1}{(t-s)^\frac{\beta +1+\eps}{2}}  \| F(v(s)+w(s)) \|_{\B_\infty^{-1-\eps}} \, \d s \\
& \lesssim  t^{-\frac12(\beta+\frac35)} \|  v_0\|_{\B_\infty^{-\frac35}} +t^{\frac{17}{100} -\frac{\beta}{2}-\frac{\eps}{2}} M.
\end{align*}
Note that the exponents appearing in these bounds are compatible with the exponents appearing in the definition \eqref{e:def:XMT} of 
$\mathfrak{X}_{T,M}$. Indeed, for $\beta = -\frac35$ there is no power of~$t$ appearing in front of $\|  v_0\|_{\B_\infty^{-\frac35}} $, and the second $t$ exponent evaluates to $\frac{47}{100}-\frac{\eps}{2}>0$. For $\beta = \frac12+2\eps$, the $t$ exponent before $\|  v_0\|_{\B_\infty^{-\frac35}} $ evaluates to $-\frac{11}{20}-\eps > - \frac35$ and the exponent appearing in the second term is $-\frac{2}{25}-\frac{3\eps}{2}>-\frac35$.

\smallskip

To bound  time differences, we make use of the identity 
\begin{align*}
\Psi^V[v,w](t) - \Psi^V[v,w](s) =& (e^{(t-s) (\Dd - c)} -  \mathrm{Id} ) e^{s (\Dd - c)}  v_0 \\
&+ (e^{(t-s )(\Dd - c)} -  \mathrm{Id} )  \int_0^s e^{(s-r) (\Dd - c)} F(v(r) +w(r)) \d r\\
&+   \int_s^t e^{(t-r) (\Dd - c)} F(v(r) +w(r)) \d r,
\end{align*}
which holds for any $0 \leq s \leq t$. This allows us to write, using Proposition~\ref{p:smooth-besov} and \eqref{e:local-lemma1a} again, 
\begin{align*}
&\| \Psi^V[v,w](t) - \Psi^V[v,w](s)\|_{L^{\infty}} \\
&\qquad  \lesssim  (t-s)^{\frac 1 8}s^{-\frac{17}{40}-\eps}  \|v_0\|_{\B^{-\frac35}_{\infty}} \\
&\qquad \qquad + (t-s)^{\frac 1 8} \int_0^t \frac{1}{(s-r)^{\frac12( \frac14 +1 +2\eps)}} \| F(v(r) +w(r))\|_{\B_\infty^{-1 -\eps}} \d r \\
&\qquad \qquad + \int_s^t \frac{1}{(t-r)^{\frac12(1+2\eps)}}   \| F(v(r) +w(r))\|_{\B_\infty^{-1 -\eps}} \d r \\
&\qquad  \lesssim (t-s)^{\frac 1 8} \big( s^{-\frac{17}{40}-\eps}  \|v_0\|_{\B^{-\frac35}_\infty} + t^{\frac{9}{200}-\eps}  M \big).
\end{align*}
Note that the $s$-exponent ${-\frac{17}{40}-\eps} $ appearing in front of $\|v_0\|_{\B^{-\frac35}_\infty}$ is strictly 
larger than the exponent $-\frac12$ which appears in the definition of $\mathfrak{X}_{M,T}$.

\medskip
The argument for $\Psi^W$ is similar: we get  
\begin{align}
\notag
&\| \Psi^W[v,w](t) \|_{\B_\infty^{1+2\eps}} \\
\notag
&\leq  \| e^{t \Dd } w_0\|_{\B_\infty^{1+2\eps}} + \int_0^t \frac{1}{(t-s)^{\frac{1+2\eps}{2} + \frac 1 4 + \eps}}  \| G(v(s),w(s))+cv(s) \|_{\B_\infty^{-\frac12-2\eps}} \, \d s \\
\label{e:w-regularity-local1}
& \lesssim  t^{-\frac{4}{5}-\eps} \|  w_0\|_{\B_{\infty}^{-\frac35}} +t^{\frac{1}{100} -\frac34-2\eps} M^3 ,
\end{align}
as well as 
\begin{align*}
\| \Psi^W[v,w](t) \|_{\B^{\frac35}_\infty} 
& \lesssim   \|  w_0\|_{\B_{\infty}^{-\frac35}} +t^{\frac{1}{100} } M^3 ,
\end{align*}
and  
\begin{align}
\notag
&\| \Psi^W[v,w](t) - \Psi^W[v,w](s)\|_{L^{\infty}} \\
\notag
&\qquad  \lesssim  (t-s)^{\frac 1 8} s^{-\frac{17}{40}-\eps} \|w_0\|_{\B^{-\frac35}_\infty} 
+ (t-s)^{\frac 1 8} M^3 s^{\frac{1}{100}- \frac38 - 2\eps}  \\  
\notag
&\qquad \qquad +\int_s^t \frac{1}{(t-r)^{\frac14+2\eps}} M^3 r^{-\frac{99}{100}} \d r \\
\label{e:w-regularity-local2}
&\qquad  \lesssim (t-s)^{\frac 1 8} \big(s^{-\frac{17}{40}-\eps}  \|w_0\|_{\B^{-\frac35}_{\infty}} + M^3 s^{\frac{1}{100}- \frac38 - 2\eps}  \big),
\end{align}
where to bound the last integral we have made use of the simple estimate $r^{-\frac{99}{100}} \leq (r-s)^{-\frac58+2\eps} s^{\frac{1}{100}-\frac{3}{8}-2\eps}$.
Summarising, we conclude that there exists a constant $C^{\star}$ depending only on $K$ and $c$, as well as an exponent $\theta>0$ such that for all $T\leq 1$, $(v,w) \in \mathfrak{X}_{M,T}$ and $M\geq \max\{1, \| v_0\|_{\B^{-\frac35}_\infty}, \| w_0\|_{\B^{-\frac35}_\infty} \} $, we have
\begin{align*}
\| (\Psi^V[v,w], \Psi^W[v,w] ) \|_{\mathfrak{X}_T} \leq C^{\star} \max\{  \| v_0\|_{\B^{-\frac 3 5}_\infty}, \| w_0\|_{\B^{-\frac35}_\infty}, T^{\theta} M^3  \}.
\end{align*}
Hence, if we choose $M=  C^{\star}\max\{1,  \| v_0\|_{\B^{-\frac 3 5}_\infty}, \| w_0\|_{\B^{-\frac35}_\infty} \}$ and $T =({C^\star M^2})^{-\frac{1}{\theta}}$, we can conclude that $\Psi $ indeed maps $\mathfrak{X}_{M,T}$ into itself. The fact that it is also a contraction on this ball can be established with the same method and we omit the proof.

\smallskip

At this point, we can conclude that for every initial data $(v_0, w_0) \in \B^{-\frac35}_\infty \times  \B^{-\frac35}_\infty $ and every choice of processes $\tau$ satisfying \eqref{e:hyp_diag}--\eqref{e:hyp_diag2}, there exists a strictly positive time $0<T_1\leq 1$ such that \eqref{e:mild-v}--\eqref{e:mild-w} has a unique solution over $[0,T_1]$. 
If the initial datum is regular (i.e. $(v_0, w_0) \in \B^{-\frac12+2\eps}_\infty \times  \B^{1+2\eps}_\infty $), then a contraction argument in $\overline{\mathfrak{X}}_{M,T_1}$ implies that this solution is continuous 
 all the way to $t=0$ without blowup. Furthermore, any upper bound on $\| v_0 \|_{\B^{-\frac{3}{5}}_\infty}$ and $\| w_0 \|_{\B_\infty^{-\frac35}}$  provides a lower bound on $T_1$. Our argument also implies that $\| v(T_1) \|_{\B^{\frac12+2\eps}_\infty}$ and  $\| w(T_1) \|_{\B_\infty^{1+2\eps}}$ are finite. In particular, we have $\| v(T_1) \|_{\B^{-\frac35}_\infty}< \infty$ and  $\| w(T_1) \|_{\B_\infty^{-\frac35}}< \infty$,
 which makes these functions admissible initial conditions to repeat the 
argument to obtain solutions on $[0,T_1+T_2]$ for some strictly positive $T_2$. A priori, the contraction mapping principle in 
$\mathfrak{X}_{T_2,M}$ would not ensure the continuity in the stronger norms of $\B^{\frac12+2\eps}_\infty$ for $v$ and $\B^{1+2\eps}_\infty$ for $w$ at time $T_1$. However
 one could  also use the contraction mapping principle on $\overline{\mathfrak{X}}_{\bar{T}_2,M}$ for some possibly smaller time $\bar{T}_2$ to find a solution for which these norms are continuous. By uniqueness of solutions in $\mathfrak{X}_{T_2,M}$, these solutions coincide, which ensures the continuity at $T_1$ of the original solution.
 By induction, one can now iterate this construction. In this way, either eventually the whole interval $[0,1]$ is covered,
 or one has  $T^{\star} = \sum_{k=1}^\infty T_k \leq 1$. By the previous observation, this  can only happen if at least one of the quantities $\| v(t) \|_{\B^{-\frac35}_\infty}$ or  $\| w(t) \|_{\B_\infty^{-\frac35}}$
blows up as $t \uparrow T^{\star}$. 

\smallskip 

There remains to argue about uniqueness of solutions to the system \eqref{e:mild-v}--\eqref{e:mild-w}. This follows from the local contractivity of the fixed point map by classical arguments (see e.g.\ Step 3 of the proof of \cite[Theorem~6.2]{JCH}). 
\end{proof}

\begin{proof}[Proof of Lemma \ref{lem:loc-theory-lemma1}]
We only treat the case $(v,w) \in \mathfrak{X}_{M,T}$, the case $(v,w) \in \overline{\mathfrak{X}}_{M,T}$ being only simpler. Throughout the calculations we make extensive use of the fact that the bounds
\[
\| v(s) \|_{\B^{-\frac35}_\infty} \leq M \qquad \text{and} \qquad
\| v(s) \|_{\B^{\frac12+2\eps}_\infty} \leq M s^{-\frac35} \; 
\]
can be interpolated, using Proposition~\ref{p:interpol}, to yield
\begin{align*}
\| v(s) \|_{\B^{\gamma}_\infty} \lesssim M s^{-\frac35 (\frac{10\gamma+6}{11+20\eps})} \; ,
\end{align*}
for all $-\frac{3}{5} \leq \gamma \leq  \frac12+2\eps$. We will in particular use this for $\gamma = \eps$. By Remark~\ref{r:Besov-vs-Lp}, this yields a bound on the $L^{\infty}$ norm of $v$:
\begin{equation}
\label{e.bound.v.Linfty.local}
\| v(s) \|_{L^{\infty}} \lesssim Ms^{-\frac{18+30\eps}{55+100\eps}} \lesssim M
s^{-\frac{33}{100}},
\end{equation}
for $\eps$ small enough (of course this exponent is somewhat arbitrary; it is only important that it is less than a third). 
In the same way, we get
\begin{align*}
\| w(s) \|_{L^{\infty}}  \lesssim M s^{-\frac{33}{100}}   \quad \text{and} \quad
\| w(s) \|_{\B^{\frac12+2\eps}} \lesssim M s^{-\frac35}
\end{align*}
for $\eps$ small enough.

\smallskip

According to the definition of $F$ in \eqref{e:defF} and Proposition~\ref{p:mult}, we have (dropping the time argument $s$ in the first expressions to lighten the notation)
\begin{align*}
\| F(v+w) \|_{\B^{-1 - \eps}_\infty}  &=  3\| (v+w - \<30>) \pl \<2>  \|_{\B^{-1 - \eps}_\infty} \ls \| v +  w -  \<30> \big\|_{L^\infty}  \| \<2>\|_{\B_{\infty}^{-1 - \eps}} \\
&\lesssim M K^2 s^{-\frac{33}{100}}.
\end{align*}

\smallskip 

We now proceed to bound $G(v(s) +w(s)) +cv(s)$ in  \eqref{e:local-lemma1}.
The term $cv(s)$ can be estimated using \eqref{e.bound.v.Linfty.local}. 
We now recall the definition of $G$ in \eqref{e:defG}:
\begin{align*}
G(v,w)  & =   -(v+w)^3 - 3 \mathsf{com}(v,w)   -3w \pe \<2> 
 - 3(v+w-\<30>) \pg \<2> + P(v+w),  
\end{align*}
where the polynomial $P$ is defined in \eqref{e:defP}. We  proceed by using the triangle inequality and bounding the terms on the right-hand side above one by one. 
The least regular term is  $a_2 (v+w)^2$ arising in the polynomial $P$. We use Proposition~\ref{p:mult} 
and Corollary~\ref{c:mult}
to bound this term:
\begin{align*}
\| a_2 (v+w)^2 \|_{\B^{-\frac12-\eps}_\infty } &  
 \lesssim \| a_2 \|_{\B^{-\frac{1}{2}-\eps}_{\infty} } \| (v + w)^2 \|_{\B^{\frac 1 2 +2\eps}_\infty} \\
& \lesssim \| a_2 \|_{\B^{-\frac{1}{2} -\eps}_{\infty} } \big(\| v+w \|_{\B^{\frac12+2\eps}_\infty} \|v+ w \|_{L^{\infty}} \big)
\lesssim K M^2 s^{-\frac{93}{100}}.
\end{align*}
%
%
For the remaining terms in the polynomial $P$, we get
\begin{align*}
\| a_1 (v+w) + a_0 \|_{\B^{-\frac12 -\eps}_\infty} \lesssim \| a_1 \|_{\B_{\infty}^{-\frac12 -\eps}} \| v+w \|_{\B_{\infty}^{\frac12 +2\eps}}    + \| a_0 \|_{\B_{\infty}^{-\frac12 -\eps}} \lesssim K M s^{-\frac{3}{5}}.
\end{align*}
 Another rather irregular term is  given by
\begin{align*}
\| 3 (v+w-\<30>) \pg \<2> \|_{\B^{-\frac12-2\eps}_\infty} 
&\lesssim  \big( \| v\|_{\B_\infty^{\frac{1}{2}-\eps}} + \| w\|_{\B_\infty^{\frac{1}{2}-\eps}} +\| \<30>\|_{\B_\infty^{\frac{1}{2}-\eps}} \big) \| \<2> \|_{\B^{-1-\eps}_\infty}\\
&\lesssim K^2 M s^{-\frac35} ,
\end{align*}
where we used Proposition~\ref{p:mult} once more. 

The remaining terms appearing  in the definition of $G$ can be bounded in stronger norms. Indeed, we have
\begin{align*}
\| (v +w)^3 \|_{L^\infty} \lesssim \| v \|_{L^\infty}^3 + \| w \|_{L^\infty}^3 \lesssim M^3 t^{-\frac{99}{100}}.
\end{align*} 
Note that  it is here where it is crucial that the blowup exponent for the $L^{\infty}$-norm of $v$ and $w$ 
is strictly less than $\frac13$, which requires the initial conditions $v_0$ and $w_0$ to be better than $\B^{-\frac23}_\infty$. Furthermore,
\begin{align*}
\| 3 w \pe \<2> \|_{L^\infty} \lesssim \| w \|_{\B^{1+2 \eps}_\infty} \| \<2> \|_{\B^{-1-\eps}_\infty} \lesssim M K t^{-\frac{17}{20}}.
\end{align*}

Finally, we recall that according to \eqref{e:def:com}, we have
\begin{align*}
\msf{com}(v,w) = \msf{com}_1(v,w) \pe \<2> + \msf{com}_2(v+w),  
\end{align*}
and use Proposition~\ref{p:mult} and Proposition~\ref{p:com1} to  write 
\begin{align*}
\|  \msf{com}_1(v,w) \pe \<2> (s)\|_{L^\infty} 
&\lesssim  \|  \msf{com}_1(v,w) (s)\|_{\B^{1+2\eps}_\infty} \| \<2> (s)\|_{\B^{-1 - \eps}_\infty} \\
&\lesssim K \|  \msf{com}_1(v,w) (s)\|_{\B^{1+2\eps}_\infty}.
\end{align*}
and
\begin{align*}
\|  \msf{com}_1(v,w) (s)\|_{\B^{1+2\eps}_\infty}
& \lesssim \| e^{s}v_0\|_{\B^{1+2\eps}_{\infty}}+
K^2 \\
&+ \int_0^s \frac{K}{(s-r)^{ \frac{3}{4}+\eps }} \big(\|v(r)\|_{\B^{\frac12+2\eps}_\infty} + \|w(r)\|_{\B^{\frac{1}{2}+2\eps}_\infty} \big) \, \d r \\
&  + \int_0^s \frac{K}{(s-r)^{1+2\eps}} \|\de_{rs} (v+w)\|_{L^{\infty}}  \, \d r\\
&\lesssim Ms^{-\frac45+\eps } +K^2  +K M {s^{-\frac{7}{20}-\eps}} + KM s^{-\frac38 - 2\eps} M .
\end{align*}
 Note in particular that we have made use of the control on the H\"older regularity in time of $(v,w)$ in order to 
treat the last integral.  
 For the second commutator term (defined in \eqref{e:def:com2}), we use Proposition~\ref{p:comm1} to obtain
\begin{align*}
\| \msf{com}_2 (v+w) \|_{L^{\infty}} \ls K^3(1+ \| v+w \|_{\B^{3 \eps}_2}) \ls K^3 M t^{-\frac35}.
\end{align*}
This completes the argument.
\end{proof}
We conclude this section by making two important observations, and then setting the stage for the derivation of the a priori bound.

\smallskip

\begin{rem}
\label{r.wellposed}
The first observation is that the solution pair obtained in Theorem~\ref{thm:local-theory} depends continuously on the initial condition. This is indeed clear from the construction of the solution pair via a fixed point argument. As a consequence, it suffices to show Theorem~\ref{t:apriori} for smooth initial datum. Indeed, once the theorem is established for $v_0 =0$ and smooth $w_0$, one can recover the general case
$v_0 =0$ and $w_0 \in \B^{-\frac35}_{\infty}$, by regularising  $w_0$ and solving the system with this initial datum, then applying the result for this solution to get a bound which holds uniformly in the regularisation parameter, and finally passing to the limit.
\end{rem}

\begin{rem}
\label{r.postprocess}
The second point we wish to make is that the norms of the spaces $\mathfrak{X}_T$ and $\overline{\mathfrak{X}}_T$, although convenient to work with in order to show Theorem~\ref{thm:local-theory}, can be improved a posteriori.
Indeed, a slight modification of \eqref{e:w-regularity-local1} yields the bound
\begin{align}
\notag
&\| \Psi^W[v,w](t) \|_{\B_\infty^{\gamma}} 
\notag
 \lesssim \|  w_0\|_{\B_{\infty}^{\gamma}} +
 \| (v,w) \|_{\overline{\mathfrak{X}}_T}^3 ,
\end{align}
for each $\gamma< \frac32-2\eps$. Similarly, a small modification of \eqref{e:w-regularity-local2} gives
\begin{align}
\notag
&\| \Psi^W[v,w](t) - \Psi^W[v,w](s)\|_{L^{\infty}} \\
\notag
&\qquad  \lesssim (t-s)^{\frac {\gamma} 2 - \eps} \big( \|w_0\|_{\B^{\gamma}_{\infty}} + \| (v,w) \|_{\overline{\mathfrak{X}}_T}^3 \big),
\end{align}
for every $\gamma<\frac{3}{2}-4\eps$.
In particular, if $w_0 \in \B^{\frac 4 3}_\infty$, then we can conclude that for any compact interval $I \subset [0,T^{\star})$ we have
\begin{equation}
\label{e.smooth.consequence}
\sup_{t\in I}\| w(t) \|_{\B^{\frac 4 3}_\infty} + \sup_{s,t \in I} \frac{\| w(t) - w(s)\|_{L^{\infty}}}{|t-s|^{\frac 2 3-\eps}} <\infty.
\end{equation}
(If the initial condition is only assumed to be in $\B_{\infty}^{-\frac35} \times \B_\infty^{-\frac35}$, then the same conclusion holds if $I$ is a compact subinterval of $(0,T^\star)$.)
\end{rem}
\smallskip

\noindent \textbf{Setup of the rest of the paper.} From now on, we give ourselves processes $\<1>$, $ \<2>$, $ \<30>$, $ \<31p>$, $ \<32p>$, $ \<22p>$ satisfying the bounds \eqref{e:hyp_diag} and \eqref{e:hyp_diag2} for some constant $K \in [1, \infty)$. Note that we assume $K \ge 1$. This assumption is a convenience allowing us to simplify bounds by using inequalities such as $K \le K^2$, etc. For the same reason, we also assume throughout that the constant $c$ appearing in \eqref{e:eqvwc} satisfies $c \ge 1$. As the argument proceeds, we will also assume a stronger lower bound on $c$, see \eqref{e.lowerbound.c}, and then simply fix $c$ according to \eqref{e.fix.c}.

\smallskip

We also give ourselves $v_0 \in \B^{\frac 1 2 +2\eps}_\infty$ and $w_0 \in \B^{\frac 4 3}_\infty$. In view of Remark~\ref{r.wellposed}, it suffices to prove Theorem~\ref{t:apriori} with this choice of initial datum (and the additional constraint $v_0 = 0$, which we do not assume for the moment). By Theorem~\ref{thm:local-theory}, this defines a solution $(v,w)$ to \eqref{e:eqvwc} over a maximal time interval $[0,T)$, where $T \in (0,1]$. Our final estimate implies in particular that
\begin{equation*}  
\lim_{t \uparrow T} \Ll(\|v(t)\|_{\B^{-\frac 3 5}_\infty} + \|w(t)\|_{\B^{-\frac 3 5}_\infty}\Rr) < \infty,
\end{equation*}
and therefore that $T = 1$ and $(v,w) \in \mathfrak{X}_1$. 

\smallskip

In the course of the argument, various norms of $v$ and $w$ will be involved. We know beforehand that all these norms are finite. Indeed by the assumed smootheness of the initial datum, we have $(v,w) \in \overline{\mathfrak{X}}_T$. Moreover, in view of Remark~\ref{r.postprocess}, we also have \eqref{e.smooth.consequence} for every compact interval $I \subset [0,T)$.

%

%
%
%
%
%
%
%


\section{A priori estimate on \texorpdfstring{$v$}{v}}
\label{s.apriori-v}

In this section, we derive a priori estimates on $v$. Theorem~\ref{t:apriori-v} below will be used many times to replace quantities involving $v$ by quantities involving $w$ only (and the initial condition $v_0$). The estimate becomes better as $c$ increases. The possibility to choose $c$ sufficiently large is used crucially in Lemma~\ref{lem:cubic2} below. The constraint on $c$ is then propagated to Theorem~\ref{t:apriori-w} and then throughout the concluding Section~\ref{s:conc}. Lemma~\ref{lem:cubic2} is part of an argument where we test the equation for $w$ against $w^{3p-3}$, and focuses on the terms arising from the cubic non-linearity $-(v+w)^3$. This testing produces the quantities
\begin{equation*}  
- \int \Ll( w^{3p} + 3 w^{3p-1} v + 3w^{3p-2} v^2 + w^{3p-3} v^3 \Rr) \; ;
\end{equation*}
fixing $c$ large allows to argue that the term $w^{3p}$ (which has the right sign if $p$ is an even integer) dominates the other terms. 
We recall our notation $\de_{st} v = v(t) - v(s)$, and that we assume $c \ge 1$.

\begin{thm}[A priori estimate on $v$]
\label{t:apriori-v}
Let $p,p', q \in [1,\infty]$ satisfy $p' \le q$, $p \le q$, let $\eps > 0$ be sufficiently small, let $\be \in \Ll[-3\eps,1 - 4\eps\Rr)$ be such that
\begin{equation*}  
{\frac{\beta+3\eps}{2}+ \frac 3 2 \Ll( \frac 1 {p'} - \frac 1 q \Rr)} < 1
\end{equation*}
and
\begin{equation}
\label{e.def.sigma}
\sigma := \frac{{\be} + 1+\eps}{2}  + \frac 3 2 \Ll( \frac 1 p - \frac 1 q \Rr) < 1,
\end{equation}
and let
\begin{equation}
\label{e:def:unc}
\uc := c-1-\Ll[K\Gamma \Ll( 1-\si \Rr)\Rr] ^{\frac 1 {1-\si}},
\end{equation}
where $\Gamma$ is Euler's Gamma function. 
For every $t < T$, we have 
\begin{equation}
\label{e:apriori-v1}
\|v(t)\|_{\B^{\be}_q} \lesssim  \frac{e^{- \uc t}}{t^{\frac{\be + 3\eps}{2}+\frac 3 2 \Ll( \frac 1 {p'} - \frac 1 q \Rr) }} \|v_0\|_{\B^{-3\eps}_{p'}} 
+ K \int_0^t \frac{e^{-\uc (t-u)}}{(t-u)^{\sigma}} \Ll(\|w(u)\|_{L^p} + K \Rr) \, \d u ,
\end{equation}
as well as
\begin{equation}
\label{e:apriori-v1.bis}
\|v(t) - e^{t(\Delta - c)} v_0\|_{\B^{\be}_q} \lesssim t^{1-\sigma}  \|v_0\|_{\B^{-3\eps}_p} 
+ K \int_0^t \frac{e^{-\uc (t-u)}}{(t-u)^{\sigma}} \Ll(\|w(u)\|_{L^p} + K \Rr) \, \d u .
\end{equation}
Furthermore, if $\beta > \eps$, then for
\begin{equation*}  
\si' := \frac 1 2 + \eps \quad \text{ and } \quad \uc' := c-1-\Ll[K\Gamma \Ll( 1-\si' \Rr)\Rr] ^{\frac 1 {1-\si'}},
\end{equation*}
we have
for every $s \le t \in [0,T)$, 
\begin{equation}
\label{e:apriori-v2}
\|\de_{st} v\|_{L^p} \ls (c+ K) |t-s|^{ \frac{\be - \eps} 2 } \, \|v(s)\|_{\B^{\be}_p} 
+ K \int_s^t \frac{e^{- \uc' (t-u)}}{(t-u)^{\si'}} \Ll(\|w(u)\|_{L^{p}} + K \Rr) \, \d u .
\end{equation}
In all estimates, the implicit constants depend on $\eps$, $p$ $q$ and ${\be}$, but \emph{neither} on $c \ge 1$ \emph{nor} on $K \in [1,\infty)$.
\end{thm}
\begin{rem}
\label{r:apriori-v1-Lp}
In view of the proof below and of Remarks~\ref{r:Besov-vs-Lp} and \ref{r:Lpbound}, we also have
$$
\|v(t)\|_{L^q} \lesssim  \frac{e^{-\uc t}}{t^{2\eps + \frac 3 2 \Ll( \frac 1 {p'} - \frac 1 q \Rr) }}  \|v_0\|_{\B^{-3\eps}_{p'}} 
+ K \int_0^t \frac{e^{-\uc (t-s)}}{(t-s)^{\frac 1 2 + \eps + \frac 3 2 \Ll( \frac 1 p - \frac 1 q \Rr)}} \Ll(\|w(s)\|_{L^p} + K \Rr) \, \d s  .
$$
\end{rem}
The proof of Theorem~\ref{t:apriori-v} relies on the following Gronwall-type lemma.

\begin{lem}[Gronwall-type lemma]
\label{l:gronwall}
Let $\sigma, \sigma' \in (0,1)$, $c\in \R$, $K_0 \in (0,\infty)$ and $k_1(s) = e^{-cs}{s^{-\sigma'}} \1_{s > 0}$, 
 $k_2(s) = K_0 e^{-cs}{s^{-\sigma}} \1_{s > 0}$.
Assume that $f,g,h : \R_+ \to \R_+$ are locally bounded measurable functions such that for every $t \ge 0$,
$$
f(t) \le g(t) + \int_0^t k_1(t-s) h(s)  \, \d s + \int_0^t k_2(t-s) f(s) ds.
$$
Then for every $t \ge 0$,
\begin{equation}
\label{e:gronwall}
f(t)  \le g(t) +  \int_0^t \Ll(\ov k_2(t-s) g(s) + \ov k_1(t-s) h(s)\Rr) \, \d s,
\end{equation}
where
\begin{align*}
\ov k_1(s) &=\frac{e^{-cs}}{s^{\sigma'}} \ 
\sum_{n = 0}^{+\infty}  \frac{\Ll(K_0\Gamma(1-\sigma) \Rr)^{n}\Gamma(1-\sigma')}{\Gamma[n(1-\sigma)+(1-\sigma')]} \, s^{n(1-\sigma)}, \\
\ov k_2(s) &=\frac{e^{-cs}}{s^{\sigma}} \ 
\sum_{n = 0}^{+\infty}  \frac{\Ll(K_0 \Gamma(1-\sigma) \Rr)^{n+1}}{\Gamma[(n+1)(1-\sigma)]} \, s^{n(1-\sigma)}.
\end{align*}
Moreover, 
\begin{equation}
\label{e:gron-est}
\frac{1}{s} \log \Ll( \sum_{n = 0}^{+\infty}\frac{\Ll( \Gamma(1-\sigma) \Rr)^{n+1}}{\Gamma[(n+1)(1-\sigma)]} \, 
s^{n(1-\sigma)} \Rr) \xrightarrow[s \to +\infty]{} \Ll( \Gamma(1-\sigma)\Rr)^{\frac 1 {1-\sigma}}.
\end{equation}
\end{lem}
\begin{rem}
We did not include a multiplicative constant $K_0$ in the definition of the convolution kernel $k_1$, contrary to the definition of $k_2$. This is because any multiplicative constant on $k_1$ can be incorporated into the definition of $h$, while such a manipulation is not possible with $k_2$. 
\end{rem}
\begin{rem}
\label{r.obvious.kbounds}
By writing
\begin{align*}
\ov k_1(s) &=\Gamma(1-\sigma')\frac{e^{-cs}}{s^{\sigma'}} \ 
\sum_{n = 0}^{+\infty} \frac{\Ll(\Gamma(1-\sigma) \Rr)^{n}}{\Gamma[n(1-\sigma)+(1-\sigma')]} \, \Ll(K_0^{\frac 1 {1-\si}}s\Rr)^{n(1-\sigma)}, \\
\ov k_2(s) &=K_0\frac{e^{-cs}}{s^{\sigma}} \ 
\sum_{n = 0}^{+\infty}  \frac{\Ll(\Gamma(1-\sigma) \Rr)^{n+1}}{\Gamma[(n+1)(1-\sigma)]} \, \Ll(K_0^{\frac 1 {1-\si}}s\Rr)^{n(1-\sigma)},
\end{align*}
we deduce from \eqref{e:gron-est} that
\begin{align*}
\ov k_1(s) &\ls  s^{-\si'} \exp \Ll( - \Ll( c - 1 - \Ll[K_0\Gamma(1-\sigma) \Rr]^{\frac 1 {1-\si}}\Rr)s \Rr) , \\
\ov k_2(s) &\ls K_0 s^{-\si} \exp \Ll( - \Ll( c - 1 - \Ll[K_0\Gamma(1-\sigma) \Rr]^{\frac 1 {1-\si}}\Rr)s \Rr),
\end{align*}
with an implicit constant independent of $s$, $K_0$, and $c$.
\end{rem}
\begin{proof}[Proof of Lemma~\ref{l:gronwall}]
Note that by iterating the hypothesis once,
\begin{align*}
f(t) & \le g(t) + \int_0^t k_1(t-t_1) h(t_1) \d t_1 \\
& \quad + \int_0^t k_2(t - t_1) \Ll( g(t_1)  + \int_0^{t_1} k_1(t_1-t_2) h(t_2) \, \d t_2 \Rr)  \d t_1 \\
&\quad + \int_{0 \le t_2 \le  t_1 \le t} k_2(t-t_1) k_2(t_1-t_2) f(t_2) \, \d t_2 \, \d t_1.
\end{align*}
We introduce some notation that will allow to iterate further. For every integer $n \ge 0$, we let
\begin{align*}
\ov k^{(n)}_1(t_0,t_{n+1}) =& \int_{t_0 \le \cdots \le t_{n+1}} \Ll( \prod_{k=1}^n k_2(t_{k+1}-t_{k}) \Rr) k_1(t_{1}-t_{0}) \, \d t_{1} \cdots \d t_{n} \\
\ov k^{(n)}_2(t_0,t_{n+1}) =& \int_{t_0 \le \cdots \le t_{n+1}} \Ll( \prod_{k=0}^n k_2(t_{k+1}-t_{k}) \Rr) \, \d t_{1} \cdots \d t_{n}
\end{align*}
(with $\ov k_i^{(0)}(s,t) = k_i(t-s)$). By induction,
\begin{align}
\notag
f(t) & \le  g(t) +  \Ll(\sum_{n = 0}^{N-1} \int_0^t \Ll( \ov k_2^{(n)}(s,t) g(s) + \ov k_1^{(n)}(s,t)h(s) \Rr) \, \d s\Rr) \\
\label{e:induct}
& \quad + \int_0^t \Ll(\ov k_1^{(N)}(s,t) h(s) + \ov k_2^{(N)}(s,t)  f(s) \Rr) \, \d s.
\end{align}
The kernel $\ov k_2^{(n)}$ satisfies
$$
\ov k_2^{(n)}(t_0,t_{n+1}) = K_0^{n+1} e^{-c(t_{n+1}-t_0)} \int_{t_0 \le \cdots \le t_{n+1}} (t_{n+1}-t_{n})^{-\sigma} 
\cdots (t_{1}-t_{0})^{-\sigma} \, \d t_{1} \cdots \d t_{n}.
$$
A change of variables enables us to rewrite this integral  as
\begin{multline*}
\int_{s_1 + \cdots + s_{n} \le t_{n+1}-t_0} s_1^{-\sigma} \cdots s_{n}^{-\sigma} (t_{n+1} - t_0 -s_1-
\cdots -s_{n})^{-\sigma} \, \d s_1 \cdots \d s_{n} \\
= (t_{n+1}-t_0)^{n(1-\sigma)-\sigma} \int_{s_1 + \cdots + s_{n} \le 1} s_1^{-\sigma} 
\cdots s_{n}^{-\sigma} (1-s_1-\cdots -s_{n})^{-\sigma} \, \d s_1 \cdots \d s_{n}
\end{multline*}
(the condition $s_i > 0$ is kept implicit).
The latter integral is the (multivariate) beta function evaluated at $(1-\sigma, \ldots, 1-\sigma)$, and is equal to 
$$
\frac{\Ll[ \Gamma(1-\sigma) \Rr]^{n+1}}{\Gamma[(n+1)(1-\sigma)]}.
$$
To sum up, we have shown that
$$
\ov k_2^{(n)}(s,t) = K_0^{n+1} e^{-c(t-s)} (t-s)^{n(1-\sigma)-\sigma} \ \frac{\Ll[ \Gamma(1-\sigma) \Rr]^{n+1}}{\Gamma[(n+1)(1-\sigma)]}.
$$
In the same way, it follows that 
$$
\ov k_1^{(n)}(s,t) = K_0^n e^{-c(t-s)} (t-s)^{n(1-\sigma)-\sigma'} \ \frac{\Ll[ \Gamma(1-\sigma) \Rr]^{n} \Gamma(1-\sigma')}{\Gamma[n(1-\sigma) + (1-\sigma')]}.
$$

This proves that the remainder term in \eqref{e:induct} tends to $0$ as $N$ tends to infinity, and yields \eqref{e:gronwall}.
In order to check \eqref{e:gron-est}, 
%
%
we use the fact that for $x \ge 1$,
$$
\sum_{n = 0}^{+\infty} \frac{x^{n(1-\sigma)}}{\Gamma[(n(1-\sigma)+1]} \le \sum_{n = 0}^{+\infty} \frac{x^{\lfloor n(1-\sigma) \rfloor + 1}}{\lfloor n(1-\sigma) \rfloor !} 
\leq \Big\lfloor\frac{1}{1-\sigma}+1\Big\rfloor   x e^{x}.
$$
Since $\Gamma[(n+1)(1-\sigma)] = [(n+1)(1-\sigma)]^{-1} \, \Gamma[(n+1)(1-\sigma)+1]$, this gives the upper bound for \eqref{e:gron-est}. Since we will never use the matching lower bound, 
we simply mention that it follows by evaluating the contribution of the summand indexed by $n$ such that 
$n(1-\sigma) \simeq s [\Gamma(1-\sigma)]^{1/(1-\sigma)}$.
\end{proof}

\begin{proof}[Proof of Theorem~\ref{t:apriori-v}]
By Propositions~\ref{p:smooth-besov} and \ref{p:embed}, the first term in the \rhs\ of \eqref{e:mild-v} is estimated by 
\begin{equation}
\label{e.apriori-v1-bis}
\|e^{t (\Dd - c)} v_0\|_{\B_q^{\be}} \lesssim 
\frac{
		e^{-ct}
	}
	{
		t^{ \frac{\beta + 3 \eps}{2}+ \frac 3 2 \Ll( \frac 1 {p'} - \frac 1 q \Rr)  }
	} \, 
\|v_0\|_{\B_{p'}^{-3\eps}}.
\end{equation}
For the second term in the \rhs\ of \eqref{e:mild-v}, recall the definition of 
$F$ in \eqref{e:defF}. 
%
Here we want to allow for $v$ of negative regularity (but no worse than $-3\eps$), so we decompose $F$ into
\begin{equation*}  
F = (w- \<30> ) \pl \<2> + v \pl \<2>,
\end{equation*}
and by Proposition~\ref{p:mult} and Remark~\ref{r:Besov-vs-Lp},  
\begin{align*}
\|[(w- \<30> ) \pl \<2>](s)\|_{\B_p^{-1-\eps}} & 
\lesssim K \|(w-\<30>)(s)\|_{L^p} \lesssim  K \Ll( \|w(s)\|_{L^p}  + K \Rr),\\
 \|[v \pl \<2>](s)\|_{\B_q^{-1-4\eps}}   & \lesssim K \|v(s)\|_{\B_q^{-3\eps}} .
\end{align*}
Hence, by Propositions~\ref{p:smooth-besov} and \ref{p:embed},
\begin{multline*}
\|v(t)\|_{\B^{\be}_q} 
\lesssim 
\frac{
		e^{-ct}
	}
	{
		t^{ \frac{\beta + 3 \eps}{2}+ \frac 3 2 \Ll( \frac 1 {p'} - \frac 1 q \Rr)  }
	} \, 
\|v_0\|_{\B_{p'}^{-3\eps}} \\
+ K\int_0^t \frac{e^{-c(t-s)}}{(t-s)^{\sigma'}}  \Ll(  \|w(s)\|_{L^p}  + K \Rr)  \, \d s 
+ K\int_0^t \frac{e^{-c(t-s)}}{(t-s)^{\sigma}}  \|v(s)\|_{\B^{-3\eps}_q}   \, \d s  ,
\end{multline*}
where 
\begin{equation*}  
\sigma := \frac{\be+1+4\eps}{2}
\end{equation*}
and
\begin{equation*}  
 \sigma' := \frac{\beta + 1 + \eps}{2}+ \frac 3 2 \Ll( \frac 1 p - \frac 1 q \Rr) < 1.
\end{equation*}
The assumption of $\beta \geq -3 \eps$ ensures that $\| v(t) \|_{\B^{-3\eps}_q} \lesssim \| v(t) \|_{\B^{\beta}_q}$, while the assumption of $\beta <  1 - 4 \eps$ ensures that $\sigma < 1$. Lemma~\ref{l:gronwall} thus yields that, for $\uc$ as in \eqref{e:def:unc},
\begin{multline*}  
\|v(t)\|_{\B^{\be}_q} \ls \frac{e^{-ct}}{t^{\frac{\beta+3\eps}{2}+ \frac 3 2 \Ll( \frac 1 {p'} - \frac 1 q \Rr)}} \|v_0\|_{\B_{p'}^{-3\eps}}  + \int_0^t \frac{e^{-\uc (t-s)}}{(t-s)^{\sigma}} \frac{e^{-cs}}{s^{\frac{\beta+3\eps}{2}+ \frac 3 2 \Ll( \frac 1 {p'} - \frac 1 q \Rr)}} \|v_0\|_{\B_{p'}^{-3\eps}} \, \d s \\
+  K \int_0^t \frac{e^{-\uc (t-s)}}{(t-s)^{\sigma'}} \Ll(  \|w(s)\|_{L^p}  + K \Rr) \, \d s;
\end{multline*}
see also Remark~\ref{r.obvious.kbounds}.
Noting that
\begin{equation*}  
\int_0^t \frac{e^{- \uc (t-s)}}{(t-s)^{\sigma}} \frac{e^{-c s}}{s^{\frac{\beta+3\eps}{2}+ \frac 3 2 \Ll( \frac 1 {p'} - \frac 1 q \Rr)}} \|v_0\|_{\B_{p'}^{-3\eps}} \, \d s \ls \frac{e^{- \uc t}}{t^{\sigma -1+ \frac{\be+3\eps}{2}+ \frac 3 2 \Ll( \frac 1 {p'} - \frac 1 q \Rr)}} \|v_0\|_{\B_{p'}^{-3\eps}} ,
\end{equation*}
and that $\sigma < 1$, we obtain \eqref{e:apriori-v1}.

\smallskip

In order to derive \eqref{e:apriori-v1.bis}, we repeat the reasoning above with minor modification. Indeed, the argument above shows that
\begin{multline*}  
\|v(t) - e^{t(\De - c)} v_0\|_{\B^{\be}_q} \\
\ls K\int_0^t \frac{e^{-c(t-s)}}{(t-s)^{\sigma'}}  \Ll(  \|w(s)\|_{L^p}  + K \Rr)  \, \d s 
+ K\int_0^t \frac{e^{-c(t-s)}}{(t-s)^{\sigma}}  \|v(s)\|_{\B^{-3\eps}_q}   \, \d s  ,
\end{multline*}
The last integral is bounded by a constant times
\begin{equation*}  
K\int_0^t \frac{e^{-c(t-s)}}{(t-s)^{\sigma}}  \Ll(\|v(s) - e^{s(\De - c)} v_0\|_{\B^{\be}_q} + \|e^{s(\De - c)}v_0\|_{\B^{-3\eps}_q} \Rr) \, \d s   ,
\end{equation*}
and by Propositions~\ref{p:smooth-besov} and \ref{p:embed},
\begin{multline*}  
\int_0^t \frac{1}{(t-s)^{\sigma}}\|e^{s(\De - c)}v_0\|_{\B^{-3\eps}_q}  \, \d s   \\
\ls \int_0^t \frac{1}{(t-s)^{\sigma+ \frac 3 2 \Ll( \frac 1 {p'} - \frac 1 q \Rr) }}\|v_0\|_{\B^{-3\eps}_{p'}}  \, \d s   \ls t^{1-\sigma - \frac 32 \Ll( \frac 1 {p'} - \frac 1 q \Rr) } \|v_0\|_{\B^{-3\eps}_{p'}}.
\end{multline*}
Inequality \eqref{e:apriori-v1.bis} then follows by another application of Lemma~\ref{l:gronwall}.

\smallskip

We now turn to \eqref{e:apriori-v2}. By homogeneity in time of the equation, 
it suffices to show \eqref{e:apriori-v2} for $s = 0$. 
By Remark~\ref{r:Besov-vs-Lp}, we have $\|\cdot\|_{L^p} \ls \|\cdot\|_{\B^{\eps}_p}$, so
$$
\|v(t) - e^{t(\Dd-c)} v_0\|_{L^p} \ls \int_0^t \frac{e^{-c(t-s)}}{(t-s)^{\frac 1 2 + \eps}} \, \|F(v+w,s)\|_{\B^{-1-\eps}_p} \, \d s.
$$
By Proposition~\ref{p:smooth-besov}, Remarks~\ref{r:Besov-vs-Lp} and \ref{r.smooth.withc} and the assumption of $c \ge 1$, we have
$$
\|(1-e^{t(\Dd-c)})v_0\|_{L^p} \ls  c t^{\frac{\be - \eps} 2 } \, \|v_0\|_{\B^{\be}_p}.
$$
Hence,
\begin{equation}
\label{e:step-vtv0}
\|v(t) - v_0\|_{L^p} \ls c t^{\frac{\be - \eps} 2 } \, \|v_0\|_{\B^{\be}_p} + \int_0^t \frac{e^{-c(t-s)}}{(t-s)^{\frac 1 2 + \eps}} \|F(v+w,s)\|_{\B^{-1-\eps}_p} \, \d s.
\end{equation}
For this bound  we do not have to deal with $v$ of negative regularity, so we simply bound
\begin{equation*}  
\|F(v+w,s)\|_{\B^{-1-\eps}_p} \ls K \Ll( \|v(s)-v_0\|_{L^p} + \|v_0\|_{L^p} +  \|w(s)\|_{L^p} + K\Rr).
\end{equation*}
Combining this estimate with \eqref{e:step-vtv0} and using that
\begin{equation*}  
\int_0^t \frac{e^{-c(t-s)}}{(t-s)^{\frac 1 2 + \eps}} \|v_0\|_{L^p} \, \d s \ls t^{\frac {1-2\eps}{2}} \, \|v_0\|_{\B^{\be}_p} ,
\end{equation*}
we obtain
\begin{multline*}  
\|v(t) - v_0\|_{L^p} \ls  (c+ K) t^{\frac {\be-\eps}{2}} \, \|v_0\|_{\B^{\be}_p} \\
+ K \int_0^t \frac{e^{-c(t-s)}}{(t-s)^{\frac 1 2 + \eps}} \Ll(  \|v(s)-v_0\|_{L^p} +  \|w(s)\|_{L^p} + K \Rr) \, \d s.
\end{multline*}
We then apply Lemma~\ref{l:gronwall} and conclude as above. 
\end{proof}

We conclude this section by fixing an important convention. We started the section by explaining that the possibility to choose $c$ sufficiently large is only really useful in Lemma~\ref{lem:cubic2} to control the cubic non-linearity. While this is indeed the case if we aim for \emph{any} bound on the solution, irrespectively of its dependency on the constant $K$, here we are aiming for more: we want to make sure that the bound obtained in the end depends \emph{polynomially} on $K$. In view of the definition of $\uc$ in \eqref{e:def:unc} and of the way it enters the estimates \eqref{e:apriori-v1}-\eqref{e:apriori-v2}, we risk encountering terms that are super-exponential in $K$ if $c$ is chosen of order $1$. This observation already suggests to fix $c$ sufficiently large in terms of $K$, to ensure that $\uc \ge 0$ and restore a polynomial dependence on $K$ in the bounds \eqref{e:apriori-v1}-\eqref{e:apriori-v2}. 

\smallskip

How large $c$ needs to be chosen depends on the exponent $\sigma$, which itself depends on the choices we will make of the parameters $p,p',q$ and $\beta$ appearing in Theorem~\ref{t:apriori-v}. In the overarching structure of the argument for our main result, we will fix an integrability exponent $p \in [1,\infty)$ sufficiently large, and then $\eps > 0$ sufficiently small in terms of $p$. We will then apply Theorem~\ref{t:apriori-v} a number of times, but always with $\beta \le \frac 1 2 + 2 \eps$, and with every integrability exponent appearing there bounded from below by the exponent $p$ fixed sufficiently large. Thus every appeal to Theorem~\ref{t:apriori-v} will produce an exponent $\sigma$ satisfying, as per \eqref{e.def.sigma},
\begin{equation*}  
\sigma \le \frac 3 4 + \frac {3\eps} 2 + \frac{3}{2p}.
\end{equation*}
In view of this, we fix from now on the following

\noindent \textbf{Important convention.} Throughout the rest of the paper, we impose
\begin{equation}  
\label{e.p.large.eps.small}
p \ge 24 \quad \text{and} \quad \eps \le 10^{-3}.
\end{equation}
In this way, every time we appeal to Theorem~\ref{t:apriori-v}, we will do so with a choice of parameters ensuring the inequality
\begin{equation*}  
\sigma \le \frac 7 8 .
\end{equation*}
In such instances, we can always replace the parameter-dependent value of $\uc$ by the lower bound
\begin{equation*}  
\uc \ge c - 1 - \Ll[ K \Gamma \Ll( \tfrac 1 8 \Rr)  \Rr] ^8.
\end{equation*}
Since $\Gamma \Ll( \tfrac 1 8 \Rr) \le 7.6$, we may use the more explicit lower bound 
\begin{equation}  
\label{e.lowerbound.unc}
\uc \ge c - (8K)^8.
\end{equation}
For convenience, \textbf{we redefine $\uc$ to be}
\begin{equation}  
\label{e.redef.uc}
\uc := c - (8K)^8,
\end{equation}
\textbf{and assume throughout that}
\begin{equation}
\label{e.lowerbound.c}
\uc \ge 0, \quad \text{that is,} \quad c \ge (8K)^8.
\end{equation}
With this new convention, the estimates \eqref{e:apriori-v1}-\eqref{e:apriori-v2} are valid provided that we make sure that $\sigma \le \frac 7 8$, which will be the case every time we actually appeal to Theorem~\ref{t:apriori-v} thanks to \eqref{e.p.large.eps.small}. A more stringent lower bound on~$\uc$ will appear later in Theorem~\ref{t:apriori-w}, by the requirements of Lemma~\ref{lem:cubic2}.
For convenience, we also impose that 
\begin{equation}
\label{e.p.integer}
\mbox{$p$ is an even integer.}
\end{equation}

%
%
%
%
%
%
\section{A priori estimate on~\texorpdfstring{$\de_{st} w$}{w(t) -w(s)}}
\label{s:apriori-dw}

As was already apparent in Section~\ref{s:loc}, one difficulty in the analysis of the behaviour of solutions to \eqref{e:eqvwc} comes from the presence of the first commutator term $\com_1$ in \eqref{e:def:com}. Indeed, assessing the (finiteness and) spatial regularity of this term requires information on the \emph{time} regularity of $v$, $w$ and $\<30>$. Adequate information on the time regularity of $v$
 was obtained in Theorem~\ref{t:apriori-v}, while the time regularity of $\<30>$ is given. The purpose of this section is to derive a bound on $\|\de_{st} w\|_{L^p}$ in terms of various norms of $w$. (Recall that $\de_{st} w = w(t) - w(s)$.)

\begin{thm}[A priori estimate on $\de_{st} w$]
\label{t:apriori-dw}
Let $p \ge 24$ and $\eps>0$ be sufficiently small. For every $s \le t \in [0,T)$,  we have
\begin{multline}
\label{e:apriori-dw}
\|\de_{st} w\|_{L^p} \ls  c K^7 (t-s)^{\frac 1 8}  \Bigg[
1 + \|v_0\|_{\B^{-3\eps}_{p}}^{3}  +  \|w(s)\|_{\B^{1 + 4 \eps}_p}  \\
+   \Ll( \int_0^t \| w(u)    \|_{\B_p^{1+4\eps}}^{p} \, \d u \Rr)^{\frac{1}{p}} 
+  \Ll(\int_0^t \|w(u)\|_{L^{3p}}^{3p} \, \d u \Rr)^{\frac{1}{p}}
\Bigg],
\end{multline}
where the implicit constant depends on $p$ and $\eps$, but neither on $K$ nor on $c$ satisfying~\eqref{e.lowerbound.c}.
\end{thm}

\smallskip 

We introduce
$$
\de_{st}' w := w(t) - e^{(t-s)\Dd} \, w(s),
$$
so that 
$$
\de_{st}' w = \int_s^t e^{(t-u)\Dd} \, [G(v,w)  + cv](u) \, \d u.
$$
The core of the proof of Theorem~\ref{t:apriori-dw} focuses on the estimation of the $L^p$ norm of~$\de_{st}' w$. We then derive an estimate of $\|\de_{st} w\|_{L^p}$ at the last step, which makes the term $\|w(s)\|_{\B^{1+4\eps}_p}$ appear. We could replace this term by the weaker quantity $\|w(s)\|_{\B^{\frac 1 4 + \eps}_p}$, but this does not facilitate subsequent arguments.

\smallskip

Recall the definition of $G$ in \eqref{e:defG}; see also \eqref{e:def:com}. There are several terms in $G$ which require special attention. The cubic term $-(v+w)^3$  has the highest degree. In the proof of Theorem~\ref{t:apriori-dw}, we cannot make use of the ``good'' sign of this term, but only treat it as a ``bad'' term. This makes the cubic non-linearities in \eqref{e:apriori-dw} appear. The estimation of $\com_1(v,w)$ involves $\|\de'_{st} w\|_{L^p}$ itself; we will derive an estimate of the form
$$
\|\de'_{st}w\|_{L^p} \ls (t-s)^{\frac 1 8} \Ll[\Ll(\sup_{u'\le u \le t} \frac{\|\de'_{u'u} w\|_{L^p}}{|u-u'|^{\frac 1 8}} \Rr)^{1/2} (\,  \cdots \, ) + \, \cdots \, \Rr],
$$
where $\cdots$ are quantities that do not involve $\de'_{st}w$. An explicit estimate on $\|\de'_{st}w\|_{L^p}$ follows, since we know from \eqref{e.smooth.consequence} and \eqref{e.elem.continuity.delta'} below that for every $t < T$, 
\begin{equation*}  
\sup_{u'\le u \le t} \frac{\|\de'_{u'u} w\|_{L^p}}{|u-u'|^{\frac 1 8}} < \infty.
\end{equation*}
The term involving $w \pe \<2>$ is the only term which requires to control derivatives of $w$ of order higher than one. This is the reason for the appearance of the term $\| w    \|_{\B_p^{1+4\eps}}$ on the right-hand side of \eqref{e:apriori-dw}. The term $a_2(v+w)^2$ also requires attention, since it involves controlling the spatial regularity of non-linear quantities of $v$ and $w$ (recall that $a_2$ is a distribution with regularity exponent $-\frac12 -\eps$). 
We summarize this decomposition as
\begin{align}
\notag
\de_{st}' w & = -\int_s^t e^{(t-u) \Dd} \, (v+w)^3(u) \, \d u \\
\notag
& \qquad -3 \int_s^t e^{(t-u) \Dd} \, [\com_1(v,w) \pe \<2>](u) \, \d u \\
\notag
& \qquad -3 \int_s^t e^{(t-u) \Dd} \, [w \pe \<2>](u) \, \d u \\
\notag
&\qquad + \int_s^t e^{(t-u) \Dd} \, [a_2(v+w)^2](u) \, \d u \\
\label{dest-4}
& \qquad + \int_s^t e^{(t-u) \Dd}\, [\ \ldots\ ](u) \, \d u,
\end{align}
where $[\ \ldots\ ]$ stands for the easier terms left out.
We provide bounds on the terms listed in \eqref{dest-4} in the following lemmas. Although we do not repeat it each time, the implicit constants in these lemmas depend neither on $K$ nor on $c$ satisfying~\eqref{e.lowerbound.c}.

%

\begin{lem}\label{lem:cubic}
Let $p \ge 24$ and $\eps>0$ be sufficiently small. For every $s \le t \in [0,T)$, 
\begin{multline}
\label{e:estim-cube}
\Ll\| \int_s^t e^{(t-u)\Dd} (v+w)^3(u) \, \d u \Rr\|_{L^p} \\
 \ls (t-s)^{\frac{p-1}{p}} \, \Ll[\|v_0\|_{\B^{-3\eps}_{p}}^{3} + K^{6}+ K^{3} \Ll(\int_0^t \|w(u)\|_{L^{3p}}^{3p} \, \d u \Rr)^{\frac{1}{p}}\Rr],
\end{multline}
where the implicit constant depends on $p$ and $\eps$.
\end{lem}
\begin{proof}
We start with the simple estimate
\begin{align*}
& \Ll\| \int_s^t e^{(t-u)\Dd} (v+w)^3(u) \, \d u \Rr\|_{L^p} \\
& \qquad \ls \int_s^t  \|(v+w)^3(u)\|_{L^p} \, \d u \\
& \qquad \ls (t-s)^{\frac{p-1}{p}}\, \Ll(\int_0^t  \|(v+w)(u)\|_{L^{3p}}^{3p} \, \d u \Rr)^{\frac{1}{p}} .
\end{align*}
By Theorem~\ref{t:apriori-v} and Remark~\ref{r:Besov-vs-Lp}, we have
$$
\|v(u)\|_{L^{3p}} \ls   {u^{-\Ll(2\eps + \frac 1 p\Rr)}} \|v_0\|_{\B^{-3\eps}_{p}}+ K \int_0^u \frac{1}{(u-s)^{\sigma}} \Ll(\|w(s)\|_{L^{3p}} + K \Rr) \, \d s  ,
$$
for  $\sigma = \frac1 2 + \eps < 1$.
We can focus on bounding
$$
\int_0^t \Ll( \int_0^u \frac{1}{(u-s)^{\sigma}} \Ll(\|w(s)\|_{L^{3p}} + K \Rr) \, \d s  \Rr)^{3p} \, \d u.
$$
By Jensen's inequality, the quantity above is bounded by a constant times
$$
 \int_0^t  \int_0^u \frac{1}{(u-s)^{\sigma}} \Ll(\|w(s)\|^{3p}_{L^{3p}} + K^{3p} \Rr) \, \d s   \, \d u 
 \ls K^{3p} + \int_0^t \|w(s)\|_{L^{3p}}^{3p} \, \d s.
$$
Summarizing, we obtain \eqref{e:estim-cube}.
\end{proof}

\begin{lem}[Estimating $\com_1$]
\label{l:estim-com1} 
Let $p \ge 24$ and $\eps>0$ be sufficiently small. For every $t \in [0, T)$,
\begin{align}
\label{e:estim-com1}
 \|\msf{com}_1(v,w)(t) - e^{t\Delta} v_0\|_{\B_p^{1 +2\eps}} 
& \ls  K^3 + K (K+c) t^{-{4\eps}} \, \|v_0\|_{\B^{-3\eps}_p}   \\
\notag
& \quad + \int_0^t \frac{K^2}{(t-s)^{\frac 3 4 +\eps}}  \|w(s)\|_{\B^{\frac 1 2 + 2\eps}_p} \, \d s\\
\notag
& \quad + \int_0^t \frac{K}{(t-s)^{1+2\eps}} \|\de_{st} w\|_{L^p}  \, \d s,
\end{align}
where the implicit constant depends on $p$ and $\eps$.
\end{lem}
\begin{rem}  
Keeping the left side of \eqref{e:estim-com1} in this form, as opposed to directly using the estimate
\begin{equation*}  
\|e^{t\Delta} v_0\|_{\B^{1+2\eps}_p} \ls t^{-\frac {1 + 7\eps}{2}} \|v_0\|_{\B^{-3\eps}_p},
\end{equation*}
will turn out to be useful for the proof of Lemmas~\ref{le:W2} and~\ref{l.test.dst} below. 
\end{rem}
\begin{proof}[Proof of Lemma~\ref{l:estim-com1}]
By Proposition~\ref{p:com1}, 
\begin{align*}
 \|\msf{com}_1(v,w)(t) - e^{t\Delta}v_0\|_{\B^{1+2\eps}_p} 
& \ls  K^2
 + \int_0^t \frac{K}{(t-s)^{\frac 3 4  + \eps}} \Ll(\|v(s)\|_{\B^{\frac 1 2 + 2 \eps}_p} + \|w(s)\|_{\B^{\frac 1 2 + 2 \eps}_p} \Rr) \, \d s \\
& \quad + \int_0^t \frac{K}{(t-s)^{1+2\eps}} \|\de_{st} (v+w)\|_{L^p}  \, \d s,
\end{align*}
We now use the estimates of $\|v(s)\|_{\B^{\frac 1 2 + 2 \eps}_p}$ and $\|\de_{st} v\|_{L^p}$ provided by Theorem~\ref{t:apriori-v}. 
We start by estimating 
\begin{equation}
\label{e:youllbeseenagain}
\int_0^t \frac{1}{(t-s)^{\frac 3 4 + \eps}} \|v(s)\|_{\B^{\frac 1 2 + 2 \eps}_p} \, \d s
\end{equation}
using \eqref{e:apriori-v1}, which takes the form of a sum of two terms. The first term is
$$
\int_0^t \frac{1}{(t-s)^{\frac 3 4 + \eps}} \ \frac{1}{s^{\frac 1 4 + \frac{5\eps}{2}}} \|v_0\|_{\B^{-3\eps}_p}  \, \d s \ls t^{-\frac{7\eps}{2}} \|v_0\|_{\B^{-3\eps}_p}.
$$
Note that the estimate holds uniformly over $c$, by the assumption of \eqref{e.lowerbound.c}. 
Similarly, the second term of the upper bound for \eqref{e:youllbeseenagain} is bounded by
\begin{multline*}
\int_0^t \frac{1}{(t-s)^{\frac 3 4 + \eps}} \int_0^s \frac{K}{(s-u)^{\frac 3 4 + \frac {3\eps}{2}}} \, (K+\|w(u)\|_{L^p}) \, \d u \, \d s
\\ \le K\int_0^t (K+\|w(u)\|_{L^p}) \int_u^t \frac{1}{(t-s)^{\frac 3 4 + \eps} (s-u)^{\frac 3 4 + \frac {3\eps}{2}}} \,  \, \d s \, \d u,
\end{multline*}
and the last integral is bounded by a constant times $(t-u)^{-\frac 1 2 - \frac{5\eps}{2}}$. Since for $\eps > 0$ sufficiently small, $\frac 1 2 + \frac {5\eps} 2 \le \frac 3 4 + \eps$, and $\|\, \cdot \, \|_{L^p} \ls \| \, \cdot \, \|_{\B^{\frac 1 2  + 2\eps}_p}$, this term is bounded by the \rhs\ of \eqref{e:estim-com1}.

As for the term with $\|\de_{st} v\|_{L^p}$, we have
\begin{multline*}
\int_0^t \frac{1}{(t-s)^{1+2\eps}} \|\de_{st} v\|_{L^p}  \, \d s 
\ls \int_0^t \frac{1}{(t-s)^{1+2\eps}} \Bigg((K+c)|t-s|^{\frac 1 4 + \frac \eps 2} \|v(s)\|_{\B^{\frac 1 2  + 2\eps}_p} \\
+ \int_0^s \frac{K}{(s-u)^{\frac 1 2 + \eps}} \, (K+\|w(u)\|_{L^p}) \, \d u \, \d s\Bigg).
\end{multline*}
The first term is \eqref{e:youllbeseenagain} again, up to an extra exponent $\eps/2$, while by the same reasoning as above, the double integral is bounded by
$$
\int_0^t \frac{K}{(t-u)^{\frac 1 2 + 3\eps}} \, (K+\|w(u)\|_{L^p}) \, \d u,
$$
and this completes the proof.
\end{proof}

\begin{lem} 
\label{l.tn}
Let $p \ge 24$ and $\eps>0$ be sufficiently small. For every $s \le t \in [0,T)$, 
\begin{align}
\label{e:step2}
& \Ll\| \int_s^t e^{(t-u)\Dd} [\com_1(v,w) \pe \<2>](u) \, \d u \Rr\|_{L^p} \\
& \ls K(t-s)^{\frac 1 8} \,\Ll[K^2 + (K+c)\|v_0\|_{\B^{-3\eps}_p} + K \Ll(\int_0^t \|w(u)\|_{\B^{\frac 1 2  + 2\eps}_p}^p \, \d u \Rr)^{\frac 1 p}  \Rr] \notag \\
& \qquad +  K(t-s)^{1 - \frac{1}{6p}} \, \tn w\tn_{p,t}^{\frac 1 2} \, \Ll( \int_0^t \|w(u)\|_{L^p}^{3p}  \, \d u \Rr)^{\frac{1}{6p}} , \notag
\end{align}
where $\tn w \tn_{p,t}$ is defined by 
\begin{equation}
\label{e:def:D}
\tn w \tn_{p,t} := \sup_{u' \le u \le t} \frac{\| \delta'_{u' u} w \|_{L^p}}{|u-u'|^{\frac 1 8}} \ .
\end{equation}
The implicit constant in \eqref{e:step2} depends on $p$ and $\eps$.
\end{lem}
\begin{proof}
We start the proof by using Proposition~\ref{p:mult}:
$$
\Ll\| \int_s^t e^{(t-u)\Dd} [\com_1(v,w) \pe \<2>](u) \, \d u \Rr\|_{L^p} \ls \int_s^t \|\com_1(v,w)(u)\|_{\B^{1+2\eps}_p} \, \d u.
$$
The initial condition is easily dealt with:
\begin{equation*}  
\int_s^t \|e^{u \Delta} v_0\|_{\B^{1+2\eps}_p} \, \d u \ls \int_s^t u^{-\frac {1 + 7\eps}{2}} \|v_0\|_{\B^{-3\eps}_p} \, \d u \ls (t-s)^{\frac 1 4} \|v_0\|_{\B^{-3\eps}_p}.
\end{equation*}
We now recall that by Lemma~\ref{l:estim-com1}, 
\begin{align}
\label{e:decomp-com1}
 \|\msf{com}_1(v,w)(u) - e^{u \Delta} v_0\|_{\B_p^{1 +2\eps}} 
 & \ls   K^3 +K(K+c) u^{-{4\eps}} \, \|v_0\|_{\B^{-3\eps}_p} \notag \\
& \quad + \int_0^u \frac{K^2}{(u-u')^{\frac 3 4 +\eps}}  \|w(u')\|_{\B^{{\frac 1 2  + 2\eps}}_p} \, \d u' \notag \\
& \quad + \int_0^u \frac{K}{(u-u')^{1+2\eps}} \|\de_{u'u} w\|_{L^p}  \, \d u' .
\end{align}
The contribution of the first line above to the integral 
\begin{equation*}  
\int_s^t  \|\msf{com}_1(v,w)(u) - e^{u \Delta} v_0\|_{\B_p^{1 +2\eps}}  \, \d u
\end{equation*}
is easily estimated.
As for the integral on the second line of \eqref{e:decomp-com1}, since $p \ge \frac 8 7$ and $\frac 3 4 + \eps < 1$ for $\eps$ sufficiently small, we can apply H\"older's and Jensen's inequalities to get
\begin{align}
& \int_s^t  \int_0^u \frac{1}{(u-u')^{\frac 3 4 +\eps}}  \|w(u')\|_{\B^{{\frac 1 2  + 2\eps}}_p} \, \d u' \, \d u \notag \\
& \qquad \ls (t-s)^{\frac 1 8} \Ll( \int_s^t \int_0^u \frac{1}{(u-u')^{\frac 3 4 + \eps}} \|w(u')\|_{\B^{\frac 1 2  + 2\eps}_p}^p\, \d u' \, \d u  \Rr)^{\frac 1 p} \notag \\
& \qquad \ls (t-s)^{\frac 1 8} \Ll(\int_0^t \|w(u)\|_{\B^{\frac 1 2  + 2\eps}_p}^p \, \d u \Rr)^{\frac 1 p}.
\label{e:error-beta} 
\end{align}
We now analyse the more subtle term coming from \eqref{e:decomp-com1}: 
\begin{equation}
\label{e:still-de}
\int_s^t  \int_0^u \frac{1}{(u-u')^{1+\eps}} \|\de_{u'u} w\|_{L^p}  \, \d u'\, \d u.
\end{equation}
To begin with, we replace $\de_{u'u} w$ by $\de'_{u'u} w$. The difference is estimated by Proposition~\ref{p:smooth-besov}: 
\begin{align*}
\Ll|\|\de_{u'u} w\|_{L^p} - \|\de'_{u'u} w\|_{L^p} \Rr| & \le \|(1-e^{-(u-u')\Dd}) w(u')\|_{L^p} \\
& \ls (u-u')^{\frac 1 8} \|w(u')\|_{\B^{\frac 1 2  + 2\eps}_p}.
\end{align*}
Hence, the difference between \eqref{e:still-de} and the same expression with $\de_{u'u}$ replaced by $\de'_{u'u}$ is bounded by
\begin{align*}
& \int_s^t \int_0^u \frac{1}{(u-u')^{\frac 7 8 +\eps}} \|w(u')\|_{\B^{\frac 1 2  + 2\eps}_p}  \, \d u'\, \d u \\
& \qquad \ls (t-s)^{\frac 1 8} \Ll( \int_s^t  \int_0^u \frac{1}{(u-u')^{\frac 7 8 +\eps}} \, \|w(u')\|_{\B^{\frac 1 2  + 2\eps}_p}^p \, \d u' \d u \Rr)^{\frac 1 p} \\
& \qquad \ls (t-s)^{\frac 1 8} \Ll( \int_0^t \|w(u)\|_{\B^{\frac 1 2  + 2\eps}_p}^p \d u \Rr)^{\frac 1 p} ,
\end{align*}
where we used H\"older's and Jensen's inequalities and $p \ge \frac 8 7$.
Note that this is the same error term as in \eqref{e:error-beta}. Moreover, by Remark~\ref{r:Lpbound},
$$
\|\de'_{u'u} w\|_{L^p} \ls \|\de'_{u'u} w\|_{L^p}^{1/2}\, \Ll(\|w(u)\|_{L^p}^{1/2} + \|w(u')\|_{L^p}^{1/2} \Rr).
$$
Hence, the double integral in \eqref{e:still-de} with $\de_{u'u}$ replaced by $\de'_{u'u}$ is bounded by
\begin{equation}
\label{e.double-int.dw.w.w}
 \tn w\tn_{p,t}^{1/2} \int_s^t  \int_0^u \frac{1}{(u-u')^{\frac{15}{16}+\eps}} (\|w(u)\|_{L^p}^{1/2} + \|w(u')\|_{L^p}^{1/2}) \, \d u' \, \d u .
\end{equation}
We have 
\begin{align*}
& \int_s^t  \int_0^u \frac{1}{(u-u')^{\frac{15}{16}+\eps }} \|w(u)\|_{L^p}^{1/2} \, \d u' \, \d u  \\
& \qquad \ls \int_s^t  \|w(u)\|_{L^p}^{1/2} \, \d u \\
& \qquad \ls (t-s)^{1- \frac{1}{6p}} \Ll(\int_0^t 
\|w(u)\|_{L^p}^{3p}\, \d u \Rr)^{\frac{1}{6p}},
\end{align*}
as well as
\begin{align*}
& \int_s^t  \int_0^u\frac{1}{(u-u')^{\frac{15}{16}+\eps}} \|w(u')\|_{L^p}^{1/2} \, \d u' \, \d u \\
& \qquad \ls (t-s)^{1-\frac{1}{6p}} \Ll( \int_s^t \int_0^u \frac{1}{(u-u')^{\frac{15}{16}+\eps}} \|w(u')\|_{L^p}^{3p} \, \d u' \, \d u \Rr)^{\frac{1}{6p}}  \\
& \qquad \ls  (t-s)^{1 - \frac{1}{6p}} \Ll( \int_0^t 
\|w(u)\|_{L^p}^{3p}  \, \d u \Rr)^{\frac{1}{6p}} .
\end{align*}
Summarizing, we obtain \eqref{e:step2}.
\end{proof}

The following lemma is the only place where we need to measure a derivative of index higher than $1$ of $w$.
\begin{lem}
Let $p \ge 24$ and $\eps>0$ be sufficiently small. For every $s \le t \in [0,T)$,
\begin{equation}
\label{e:step3}
\Ll\|  \int_s^t e^{(t-u) \Dd} \, [w \pe \<2>](u) \, \d u \Rr\|_{L^p}  \\
 \ls  K  (t-s)^{\frac{p-1}{p}} \,  \Ll( \int_0^t \| w(u)    \|_{\B_p^{1+2\eps}}^{p } \, \d u \Rr)^{\frac{1}{p}} \;,    
\end{equation}
where the implicit constant depends on $p$ and $\eps$.
\end{lem}
\begin{proof}
The estimate \eqref{e:step3} follows easily by writing
\begin{align}
&\Ll\|  \int_s^t e^{(t-u) \Dd} \, [w \pe \<2>](u) \, \d u \Rr\|_{L^p}  \notag\\
&\ls \int_s^t     \| w \pe \<2> \|_{L^p } (u) \; du 
\ls K  \int_s^t  \| w (u)  \|_{\B_p^{1+2\eps} }  \; du  \notag\\
\notag
&\ls  K  (t-s)^{\frac{p-1}{p} }  \Ll( \int_0^t \| w (u)   \|_{\B_p^{1+2\eps}}^{p } du \Rr)^{\frac{1}{p}} \;. \qedhere
\end{align}
\end{proof}

For the next lemma, we recall that $a_2$ is the coefficient in front of the quadratic term in  $P$  which was defined in \eqref{e:defP}, and that  $a_2$ is a distribution with spatial  
regularity $-\frac12 - \eps$ controlled uniformly in time. 
\begin{lem}\label{le:vw-The_First_Time}
Let $p \ge 24$ and $\eps>0$ be sufficiently small. For every $s \le t \in [0,T)$,
\begin{multline}
\label{e:step4}
\Ll\|  \int_s^t e^{(t-u) \Dd} \, [a_2(v+w)^2](u) \, \d u \Rr\|_{L^p}  
\ls  K (t-s)^{\frac{1}{8} }  \;\\
  \times\Ll[ K^6 + \|v_0\|_{\B^{-3\eps}_{p}}^3 + K^3 \Ll( \int_0^t  \|w(u)\|_{L^{3p}}^{3p} \, \d u\Rr)^\frac 1 p + \Ll(\int_0^t \|w(u)\|_{\B^{1+4\eps}_p}^p  \, \d u \Rr)^\frac 1 p  \Rr]  ,
\end{multline}
where the implicit constant depends on $p$ and $\eps$.
\end{lem}
%
%
\begin{proof}
We start by bounding the term which is of highest order in $w$, using Remark~\ref{r:Besov-vs-Lp} and Propositions~\ref{p:smooth-besov} and \ref{p:mult}: 
\begin{multline}
\label{s4-1}
\Ll\|  \int_s^t e^{(t-u) \Dd} \, [a_2 w^2] (u) \, \d u \Rr\|_{L^p}  \\
 \ls   \int_s^t \frac{1}{(t-u)^{\frac{1}{4}+\eps }} \, \|a_2 w^2\|_{\B^{-\frac{1}{2}-\eps}_p}  (u) \, \d u 
\ls \int_s^t \frac{K}{(t-u)^{\frac{1}{4}+\eps }} \, \|w^2(u)\|_{\B^{\frac{1}{2}+ 2\eps}_p}   \, \d u  \;.
\end{multline}
By Proposition~\ref{p:mult} (specifically, \eqref{e.sharp.mult}), we have
\begin{equation*}  
\|w^2(u)\|_{\B^{\frac{1}{2}+ 2\eps}_p} \ls \|w(u)\|_{L^{3p}} \, \|w(u)\|_{\B^{\frac 1 2 + 2 \eps}_{\frac{3p}{2}}}.
\end{equation*}
Moreover, by Proposition~\ref{p:interpol} and Remark~\ref{r:Besov-vs-Lp}, we have
\begin{equation}  
\label{e.bound.for.w2.0}
\|w(u)\|_{\B^{\frac 1 2 + 2 \eps}_{\frac{3p}{2}}} \ls \|w(u)\|_{\B^{1+4\eps}_{p}}^\frac 1 2  \, \|w(u)\|_{L^{3p}}^\frac 1 2,
\end{equation}
so that, by Young's inequality,
\begin{equation}
\label{e.bound.for.w2}
\|w^2(u)\|_{\B^{\frac{1}{2}+ 2\eps}_p} \ls \|w(u)\|_{\B^{1+4\eps}_p}^\frac 1 2 \, \|w(u)\|_{L^{3p}}^\frac 3 2 \ls \|w(u)\|_{\B^{1+4\eps}_p} +  \|w(u)\|_{L^{3p}}^3.
\end{equation}
We deduce from this and H\"older's inequality that
\begin{multline*}  
\int_s^t \frac{1}{(t-u)^{\frac{1}{4}+\eps }} \, \|w^2(u)\|_{\B^{\frac{1}{2}+ 2\eps}_p}   \, \d u  \ls \Ll( \int_s^t \frac{1}{(t-u)^{\Ll(\frac{1}{4}+\eps \Rr)\frac{p}{p-1}}} \, \d u \Rr)^{\frac {p-1}{p}} \\
\times \Ll( \int_0^t \Ll( \|w(u)\|_{\B^{1+4\eps}_p}^p +  \|w(u)\|_{L^{3p}}^{3p} \Rr) \, \d u \Rr)^\frac 1 p. 
\end{multline*}
For $p > \frac 8 3$ and $\eps > 0$ sufficiently small, we have
\begin{equation*}  
\Ll(\frac{1}{4}+\eps \Rr)\frac{p}{p-1} < 1 \quad \text{and} \quad \frac {p-1}{p} - \frac 1 4 \ge \frac 1 8,
\end{equation*}
so that we obtain
\begin{equation*}  
\int_s^t \frac{1}{(t-u)^{\frac{1}{4}+\eps }} \, \|w^2(u)\|_{\B^{\frac{1}{2}+ 2\eps}_p}   \, \d u  \\
\ls (t-s)^{\frac 1 8} \Ll( \int_0^t \Ll( \|w(u)\|_{\B^{1+4\eps}_p}^p +  \|w(u)\|_{L^{3p}}^{3p} \Rr) \, \d u \Rr)^\frac 1 p. 
\end{equation*}

\smallskip

We now turn to the term involving $a_2 v^2$. Arguing as in \eqref{s4-1} and then using Proposition~\ref{p:mult}, we get
\begin{align}
\notag
&\Ll\| \int_s^t e^{(t-u) \Dd}   [a_2 v^2 ](u) \, \d u \Rr\|_{L^p} \\
\label{e.a2v2}
& \qquad \ls   \int_s^t \frac{K}{(t-u)^{\frac{1}{4}+\eps }} \, \|v^2(u)\|_{\B^{\frac{1}{2}+2\eps}_p}   \, \d u 
\ls  \int_s^t \frac{K}{(t-u)^{\frac{1}{4}+\eps }} \, \|v(u)\|^2_{\B^{\frac{1}{2}+2\eps}_{2p}}  \, \d u  \;.
\end{align}
Recall that by Theorem~\ref{t:apriori-v},
\begin{equation}
\label{e.recall-t31}
\|v(u)\|_{\B^{\frac{1}{2}+2\eps}_{2p}}  \ls  u^{-\Ll(\frac{1}{4} + \frac{5\eps}{2} + \frac{3}{4p}\Rr)}\|v_0\|_{\B^{-3\eps}_{p}}  + K \int_0^u \frac{1}{(u-u')^{\frac{3}{4} + \frac{3\eps}{2}}} (K + \| w(u') \|_{L^{2p}}) \, \d u'.
\end{equation}
The term containing the initial condition contributes 
\begin{equation}\label{v-init-in-the-middle-of-lemma}
\int_s^t \frac{1}{(t-u)^{\frac{1}{4}+\eps }}u^{-\Ll(\frac{1}{2} + {5\eps} + \frac{3}{2p}\Rr)}\|v_0\|_{\B^{-3\eps}_{p}}^2 \, \d u \ls (t-s)^{\frac 1 4 - 6\eps - \frac {3}{2p}} \ \|v_0\|_{\B^{-3\eps}_{p}}^2 .
\end{equation}
The contribution of the second term in \eqref{e.recall-t31} to the integral on the \rhs\ of \eqref{e.a2v2} can be rewritten as
$$
\int_s^t f(t-u) \Ll( \int_0^u g(u-u') h(u') \, \d u' \Rr)^2 \, \d u  \;,
$$
for 
\begin{align}\label{e:fgh}
f(u) = \frac{1}{u^{\frac{1}{4}+\eps }},    \qquad g(u) = \frac{1}{u^{\frac{3}{4} + \frac{3\eps}{2}}}, \qquad h(u) = K + \| w(u) \|_{L^{2p}}\;.
\end{align}
Therefore, using H\"older's inequality in the first and Young's inequality in the second step we get 
\begin{align*}
\notag
&\int_s^t f(t-u) \Ll( \int_0^u g(u-u') h(u') \, \d u' \Rr)^2 \, \d u  \\
&\qquad   \ls \Ll( \int_s^t  f(u)^{q_1}  \d u \Rr)^{\frac{1}{q_1}} \Ll(\int_0^t  \Ll| \int_0^u  g(u-u')  h(u') \d u' \Rr|^{2 q_1'} \,  \d u\Rr)^{\frac{1}{q_1'}}  \\
& \qquad \ls \Ll( \int_s^t  f(u)^{q_1}  \d u \Rr)^{\frac{1}{q_1}} \Ll(\int_0^t   g (u)^{q_2} \d u  \Rr)^{\frac{2}{q_2}} \Ll(\int_0^t   h (u)^{q_3} \d u  \Rr)^{\frac{2}{q_3}}  \;,
\end{align*}
where $q_1'$ is the adjoint exponent of $q_1$ and $q_2, q_3 \in (1,\infty)$ satisfy $\frac 1 {q_2} + \frac 1 {q_3} = 1 + \frac{1}{2 q_1'}$. We also impose $q_1$ and $q_2$ to be sufficiently small that the corresponding integrals are finite. That is, we impose
\begin{align}\label{e:q-conditions}
\frac{3}{2} = \frac{1}{2q_1} + \frac{1}{q_2} + \frac{1}{q_3} \;, \qquad    q_1 <  \frac{4}{1 + 4 \eps} \qquad \text{and} \qquad q_2 < \frac{1}{\frac{3}{4} + \frac{3\eps}{2}} \;.
\end{align}
Choosing $q_3 = 3 p$, and $q_1 = \frac{2}{1+ 2\eps}$ (which implies that the second condition in \eqref{e:q-conditions} is satisfied)  one sees that the 
the $q_2$ determined by the first condition in \eqref{e:q-conditions} satisfies $q_2 \leq \frac{12}{11}$ for any $p>1$, which implies in turn that for $\eps>0$ small enough the third 
condition holds.
Therefore, using $\| w\|_{L^{2p}} \ls \| w \|_{L^{3p}}$ we can summarise 
\begin{align}
\notag
\lefteqn{\int_s^t f(t-u) \Ll( \int_0^u g(u-u') h(u') \, \d u' \Rr)^2 \, \d u}
\qquad &   \\
\notag
& \ls  (t-s)^{\frac14}  \Ll[ \int_0^t \Ll( K +  \| w(u) \|_{L^{3p}}\Rr)^{3p} \d u \Rr]^{\frac{2}{3p}} \\
\label{e:vv-and-a-lot-of-Holder}
& \ls (t-s)^{\frac14}  \Ll[ K^2 +  \Ll(\int_0^t   \| w(u) \|_{L^{3p}}^{3p} \d u \Rr)^{\frac 1 p}\Rr].
\end{align}
\smallskip

We now analyse the term involving the product $vw$. 
As before, we write
\begin{align}
\label{e:vw1111}
&\Ll\| \int_s^t e^{(t-u) \Dd}   [a_2 v w ](u) \, \d u \Rr\|_{L^p} 
\ls  \int_s^t \frac{K}{(t-u)^{\frac{1}{4}+\eps}} \, \|v w (u)\|_{\B^{\frac{1}{2}+2\eps}_{p}}  \, \d u .
\end{align}
and then use Proposition~\ref{p:mult} to bound
\begin{equation*}  
\|v w (u)\|_{\B^{\frac{1}{2}+2\eps}_{p}} \ls \|v(u)\|_{L^{3p}} \, \|w(u)\|_{\B^{\frac 1 2 + 2 \eps}_{\frac {3p} 2}} + \|w(u)\|_{L^{2p}} \|v(u)\|_{\B^{\frac 12 + 2\eps}_{2p}}.
\end{equation*}
The second term is easily taken care of, since the inequality
\begin{equation*}  
\|w(u)\|_{L^{2p}} \|v(u)\|_{\B^{\frac 12 + 2\eps}_{2p}} \ls \|w(u)\|_{L^{2p}}^2 +  \|v(u)\|_{\B^{\frac 12 + 2\eps}_{2p}}^2
\end{equation*}
reduces the analysis to the sum of
\begin{equation*}  
\int_s^t  \frac{1}{(t-u)^{\frac{1}{4}+\eps}} \, \|w(u)\|_{L^{2p}}^2 \, \d u \ls (t-s)^{\frac 1 8} \Ll( \int_0^t \|w(u)\|_{L^{3p}}^{3p} \, \d u \Rr)^{\frac 2 {3p}},
\end{equation*}
and
\begin{equation*}  
\int_s^t  \frac{1}{(t-u)^{\frac{1}{4}+\eps}} \, \|v(u)\|_{\B^{\frac 12 + 2\eps}_{2p}}^2 \, \d u, 
\end{equation*}
which was already analysed, see \eqref{e.a2v2}. There remains to bound
\begin{equation*}  
\int_s^t \frac{1}{(t-u)^{\frac{1}{4}+\eps}} \, \|v(u)\|_{L^{3p}} \, \|w(u)\|_{\B^{\frac 1 2 + 2 \eps}_{\frac {3p} 2}} \, \d u.
\end{equation*}
By \eqref{e.bound.for.w2.0}, we have
\begin{align*}  
\|v(u)\|_{L^{3p}} \, \|w(u)\|_{\B^{\frac 1 2 + 2 \eps}_{\frac {3p} 2}} & 
\ls \|v(u)\|_{L^{3p}} \|w(u)\|_{\B^{1+4\eps}_{p}}^\frac 1 2  \, \|w(u)\|_{L^{3p}}^\frac 1 2 \\
& \ls \|w(u)\|_{\B^{1+4\eps}_p} +  \|w(u)\|_{L^{3p}}^3 + \|v(u)\|_{L^{3p}}^3.
\end{align*}
The contribution of the first two terms was already analysed, see \eqref{e.bound.for.w2} and following. The contribution of the last term is only simpler to analyse than the quantity on the right-hand side of \eqref{e.a2v2}. Indeed, appealing again to Theorem~\ref{t:apriori-v}, the initial condition appearing there poses no difficulty, and there remains to bound
\begin{align*}  
 \int_s^t \frac{1}{(t-u)^{\frac 1 4 + \eps}} \Ll( \int_0^u \frac{1}{(u-u')^{\frac 1 2 + \eps}}(K + \|w(u)\|_{L^{3p}})  \, \d u'  \Rr) ^3 \, \d u,
\end{align*}
which we rewrite as
\begin{equation*}  
\int_s^t  f(t-u) \Ll( \int_0^u g'(u-u') h(u') \, \d u'\Rr)^3 \, \d u,
\end{equation*}
with $f$ and $h$ as in \eqref{e:fgh}, and $g'(u) = u^{-\frac 1 2 -\eps}$. We then note that, by H\"older's and Young's inequalities, this quantity is bounded by
\begin{equation*}  
\Ll( \int_s^t  f(u)^{q_1}  \d u \Rr)^{\frac{1}{q_1}} \Ll(\int_0^t   g'(u)^{q_2} \d u  \Rr)^{\frac{3}{q_2}} \Ll(\int_0^t   h (u)^{q_3} \d u  \Rr)^{\frac{3}{q_3}} ,
\end{equation*}
provided that 
\begin{equation*}  
\frac 4 3 = \frac 1 {3q_1} + \frac 1 {q_2} + \frac 1 {q_3}, \qquad q_1 < \frac {4}{1 +4\eps}, \qquad \text{and} \qquad q_2 < \frac 2 {1 + 2\eps}.
\end{equation*}
As before, we choose $q_3 = 3p$ and $q_1 = \frac 2 {1+2\eps}$, and then the equality in the first condition above implies $q_2 \le \frac 6 5$, which in particular satisfies the last condition if $\eps > 0$ is sufficiently small. We therefore obtain that the integral above is bounded by
\begin{equation*}  
(t-s)^\frac 1 4   \Ll[ \int_0^t \Ll( K +  \| w(u) \|_{L^{3p}}\Rr)^{3p} \d u \Rr]^{\frac{1}{p}},
\end{equation*}
and this completes the estimation of this term.
\end{proof}

\smallskip

We now bound the terms which were not made explicit in \eqref{dest-4}. 
\begin{lem}\label{le:dots}
Let $p \ge 24$ and $\eps>0$ be sufficiently small. For every $s \le t \in [0,T)$, 
\begin{multline}
\label{e:dot-Lemma}
 \int_s^t e^{(t-u) \Dd}\, [\ \ldots\ ](u) \, \d u  \lesssim   K(K+c)(t-s)^{\frac 1 8}   \\
 \times  \Ll[K^2+ \|v_0\|_{\B^{-3\eps}_p} +   \Ll( \int_0^t  \,\|   w(u)\|_{\B^{1 + 2 \eps}_p}^p  \d u \Rr)^{\frac{1}{p}} +K\Ll(   \int_0^t   \|w(u)\|_{L^p}^{3p}  \d u  \Rr)^{\frac{1}{3p}} \Rr].
\end{multline}
where the dots $\ldots$ represent all the terms left out in \eqref{dest-4} (spelled out explicitly in \eqref{dot-terms} below). The implicit constant depends on $p$ and $\eps$.
\end{lem}
\begin{proof}
We need to estimate 
\begin{align}
\label{dot-terms}
 \int_s^t e^{(t-u) \Dd}\, \Ll[ -3 \msf{com}_2(v+w)   - 3(v+w-\<30>) \pg \<2> + a_0 +a_1(v+w) + c v\Rr](u) \, \d u,
\end{align}
and we proceed by bounding these terms one by one.
\medskip

 To begin with, we show that
\begin{multline}\label{Rest-lemma1}
\Ll\|\int_s^t e^{-(t-u) \Dd} \, \com_2(v+w)(u) \, \d u\Rr\|_{L^p} \\
\ls K^2 (t-s)^{\frac{p-1}{p}} \, \Ll[\|v_0\|_{\B^{-3\eps}_p} + K^2 + K\Big( \int_0^t \|w(u)\|_{\B^{1+2 \eps}_p}^p \, \d u \Big)^{\frac{1}{p}}\Rr].
\end{multline}
Indeed, by Remark~\ref{r:Lpbound} and Proposition~\ref{p:comm1}, the \lhs\ above is bounded by 
\begin{align*}
\int_s^t &  \|\com_2(v+w)(u)\|_{L^p}\, \d u \\
&\ls K^2 \int_s^t \Ll(K +  \|(v+w)(u)\|_{\B^{3\eps}_p}\Rr) \, \d u\\
& \lesssim K^2  (t-s)^{\frac{p-1}{p}}  \Ll[ K + \Big( \int_0^t  \|v(u) \|_{\B^{3\eps}_p}^p \d u \Big)^{\frac{1}{p}}+ \Big( \int_0^t \|w(u)\|_{\B^{3\eps}_p}^p \, \d u \Big)^{\frac{1}{p}}\Rr] .
\end{align*}
 For the integral involving $v$, we apply Theorem~\ref{t:apriori-v} as before to obtain
 \begin{align*}
 \int_0^t  \|v(u) \|_{\B^{3\eps}_p}^p \d u &\lesssim \|v_0\|_{\B^{-3\eps}_p}^p +  \int_0^t \Big( K\int_0^u \frac{1}{(u-u')^{\frac{1+4\eps}{2}} }  \Ll( K + \| w(u') \|_{L^p} \Rr) \d u' \Big)^{p} \d u \\
 & \lesssim \|v_0\|_{\B^{-3\eps}_p}^p +  K^p \int_0^t \Ll(  K^p + \| w(u') \|_{L^p}^p \Rr)   \d u' \;,
 \end{align*}
where we have first  used Jensen's inequality to move the $p$-th power inside the $\d u'$- integral, and then carried out the $\d u$ integral. So \eqref{Rest-lemma1} follows.

\medskip
We now show that
\begin{multline}
\label{Rest-lemma2}
\Ll\|\int_s^t  e^{(t-u) \Dd}\, \Ll(  (v+w) \pg \<2>\Rr)(u)  \d u \Rr\|_{L^p}\\
\lesssim  K  (t-s)^{\frac 1 8}   
 	\Ll[ \|v_0\|_{\B^{-3\eps}_p} +  \Big( \int_0^t  \,\|   w(u)\|_{\B^{1 + 2 \eps}_p}^p  \d u \Big)^{\frac{1}{p}} +K^2 + K\Big(   \int_0^t  \|w(u)\|_{L^p}^{3p} \d u  \Big)^{\frac{1}{3p}} \Rr].
\end{multline}
Indeed, on the one hand, by Proposition~\ref{p:mult},
\begin{align}
\notag
\Ll\|   \int_s^t e^{(t-u) \Dd}\, ( w \pg \<2>) (u) \d u \Rr\|_{L^p} & \lesssim  \int_s^t  \Ll\| ( w \pg \<2>)(u)  \Rr\|_{L^p} \d u \\
\notag
 &\lesssim K  \int_s^t \|   w(u) \|_{\B^{1 + 2 \eps}_p}  \d u\\
\notag
&  \lesssim  K (t-s)^{\frac{p-1}{p}} \Ll( \int_0^t  \,\|   w(u)\|_{\B^{1 + 2 \eps}_p}^p  \d u \Rr)^{\frac1p}.
\end{align}
On the other hand, 
\begin{align*}
\Ll\|   \int_s^t e^{(t-u) \Dd}\, ( v \pg \<2>)(u)  \d u \Rr\|_{L^p} 
&\lesssim    \int_s^t \frac{1}{(t-u)^{\frac14 + \eps}}\, \| ( v \pg \<2>)(u)\|_{\B_p^{ -\frac12}}  \d u \\
\lesssim   K  \int_s^t \frac{1}{(t-u)^{\frac14 + \eps}}\, \|  v (u)\|_{\B_p^{ \frac12+\eps}}  \d u.
\end{align*}
We use Theorem~\ref{t:apriori-v} again to estimate $\|v(u)\|_{\B_p^{\frac 1 2 + \eps}}$. The contribution of the initial condition is 
\begin{equation}  
\label{e.whowouldworry}
 \int_s^t \frac{1}{(t-u)^{\frac14 + \eps}}\, \frac{1}{u^{\frac 1 4 +2\eps}} \|  v_0\|_{\B_p^{-3\eps}}  \d u \ls (t-s)^{\frac 1 2 - 3 \eps} \|  v_0\|_{\B_p^{-3\eps}}.
\end{equation}
The contribution of the other term from Theorem~\ref{t:apriori-v} takes the form
\begin{align*}
 &      \int_s^t \frac{1}{(t-u)^{\frac14 + \eps}}\,    \Big( K  \int_0^u \frac{1}{(u-u')^{\frac34 + \eps}} \Ll(K + \|w(u')\|_{L^p} \Rr) \, \d u' \Big) \d u  \\
 & \lesssim  K  \Big(  \int_s^t \frac{1}{(t-u)^{\Ll(\frac14 + \eps\Rr) \frac{3p}{3p-1}}}\,  \d u \Big)^{\frac{3p-1}{3p}}   \\
 & \qquad \qquad \qquad \times \Big(   \int_0^t  \Big( \int_0^u\frac{1}{(u-u')^{\frac34 + \eps}}   \Ll(K + \|w(u')\|_{L^p} \Rr) \d u'  \Big)^{3p} \d u \Big)^{\frac{1}{3p}}  \\
 & \lesssim    (t-s)^{\frac{3p-1}{3p} -\frac{1}{4} - \eps} \,   \Ll(   \int_0^t  \Ll(K + \|w(u)\|_{L^p}^{3p} \Rr) \d u  \Rr)^{\frac{1}{3p}} \;.
\end{align*}
As before, we have used Jensen's inequality to move the power $3p$ inside the $\d u' $ integral. 
Therefore,  \eqref{Rest-lemma2} follows, since for $p \ge 1$, we have $\frac{3p-1}{3p} -\frac{1}{4} > \frac 1 8$.
\medskip

We also have 
\begin{align*}
\Ll\|  \int_s^t e^{(t-u) \Dd}  \;  (\<30> \pg \<2> + a_0)(u) \; \d u \Rr\|_{L^p} &\lesssim  \int_s^t \frac{1}{(t-u)^{\frac{1}{4} +2\eps}}  \| (\<30> \pg \<2> + a_0)(u)  \|_{\mathcal{C}^{ -\frac{1}{2} - 2 \eps}} \d u \\
&\lesssim K (t-s)^{\frac34 - 2\eps}\;,
\end{align*}
which is bounded by the right-hand side of \eqref{e:dot-Lemma}. 

\medskip
Finally, we write
\begin{align*}
&\Ll\| \int_s^t e^{(t-u) \Dd}\, [ a_1(v+w) + c v ](u)\; \d u \Rr\|_{L^p} \\
&\lesssim  \int_s^t \frac{ 1}{(t-u)^{\frac14 + \eps}}\Big( (K+c) \| v(u) \|_{\B^{\frac12  + 2\eps}_p}+ K \| w(u)\|_{\B^{\frac12  + 2\eps}_p} \Big)  \; \d u. 
\end{align*}
For the term involving $v$, we call again Theorem~\ref{t:apriori-v} to write
\begin{multline*}
 \int_s^t \frac{1}{(t-u)^{\frac14 + \eps}}  \| v(u) \|_{\B^{\frac12  + 2\eps}_p}  \d u \ls  \int_s^t \frac{1}{(t-u)^{\frac14 + \eps}} \frac 1 {u^{\frac 1 4 + \frac {5 \eps} 2}}  \| v_0 \|_{\B^{-3\eps}_p}  \d u \\
  +  \int_s^t \frac{K}{(t-u)^{\frac14 + \eps}} \int_0^u \frac{1}{(u-u')^{\frac{3}{4} + \frac{3\eps}{2}}} \Ll(K + \| w(u') \|_{L^{p}} \Rr) \d u'    \d u .
\end{multline*}
The first term is bounded as in \eqref{e.whowouldworry}. Using H\"older's inequality, we bound the second term by
\begin{multline*}
K (t-s)^{\frac{3p-1}{3p} - \frac 1 4 - \eps} \Ll[ \int_0^t \Big( \int_0^u \frac{1}{(u-u')^{\frac{3}{4} + \frac{3\eps}{2}}} \Ll(K + \| w(u') \|_{L^{p}} \Rr) \d u'  \Big)^{3p} \d u \Rr]^\frac 1 {3p} \\
\lesssim K (t-s)^{\frac{3p-1}{3p} - \frac 1 4 - \eps}  \Ll[\int_0^t \Ll(K + \| w(u') \|_{L^{p}}\Rr)^{3p} \d u'\Rr]^{\frac 1 {3p}}.
\end{multline*}
%
%
For the integral involving $w$, we write 
\begin{align*}
  \int_s^t& \frac{ 1}{(t-u)^{\frac14 + \eps}}  \| w(u)\|_{\B^{\frac12  + 2\eps}_p}   \; \d u  
  \\ 
  &\ls  (t-s)^{1-\frac{2}{3p} -\frac{1}{4} -\eps}  
  		 \Big( \int_0^t  \| w(u)\|_{\B^{\frac12  + 2\eps}_p}^\frac{3p}{2}  \, \d u \Big)^{\frac{2}{3p}}\\
  & \ls (t-s)^{1-\frac{1}{q} -\frac{1}{4} -\eps}   
  		 \Big( \int_0^t  \| w(u)\|_{\B^{1  + 4\eps}_p}^{^\frac{3p}{4} }  \| w(u)\|_{L^p} ^{^\frac{3p}{4}  } \,  \d u  \Big)^{\frac{2}{3p}} \\
& \ls (t-s)^{1-\frac{2}{3p} -\frac{1}{4} -\eps} 
		  \Big( \int_0^t  \| w(u)\|_{\B^{1  + 4\eps}_p}^{ p} \,  \d u  \Big)^{ \frac{1}{2p}}
			 \Big( \int_0^t  \| w(u)\|_{L^p} ^{  3p } \,  \d u  \Big)^{ \frac{1}{6p}} \\
& \ls  (t-s)^{1-\frac{2}{3p} -\frac{1}{4} -\eps} 
		 \Ll[ \Big( \int_0^t  \| w(u)\|_{\B^{1  + 2\eps}_p}^{ p} \,  \d u \Big)^{  \frac{1}{p}} 
		 	+  \Big( \int_0^t  \| w(u)\|_{L^p} ^{  3p }  \,  \d u\Big)^{  \frac{1}{3p}} \Rr]  \ ,
\end{align*}
where in the second step, we have made use of the interpolation bound provided by Proposition~\ref{p:interpol} and of Remark~\ref{r:Besov-vs-Lp}. Note that $1-\frac{2}{3p} -\frac{1}{4} >\frac 1 8$ for $p \ge \frac 8 7$,
so this term is bounded by the right-hand side of  \eqref{e:dot-Lemma} as well provided that $\eps > 0$ is sufficiently small.
\end{proof}

%

\begin{proof}[Proof of Theorem~\ref{t:apriori-dw}]
Combining the bounds we have derived in Lemmas \ref{lem:cubic}--\ref{le:dots}, and simplifying their dependence on $K$ and $c$, we obtain 
\begin{align*}
& \|\de_{st}' w\|_{L^p}  \\ 
& \ls
K^6 (t-s)^{\frac{p-1}{p}} \,  \Ll[1 + \|v_0\|_{\B^{-3\eps}_{p}}^{3} +  \Ll(\int_0^t \|w(u)\|_{L^{3p}}^{3p} \, \d u \Rr)^{\frac{1}{p}}\Rr] \\
& \quad +  cK^3 (t-s)^{\frac18} \, \Ll[1 + \|v_0\|_{\B^{-3\eps}_p} + \Ll(\int_0^t \|w(u)\|_{\B^{\frac 1 2  + 2\eps}_p}^p \, \d u \Rr)^{\frac 1 p}\Rr] \\
& \quad +   K(t-s)^{1 - \frac{1}{6p}} \, \tn w\tn_{p,t}^{\frac 1 2} \, \Ll( \int_0^t \|w(u)\|_{L^p}^{3p}  \, \d u \Rr)^{\frac{1}{6p}} \\
& \quad +  K (t-s)^{\frac{p-1}{p}} \,  \Ll( \int_0^t \| w(u)    \|_{\B_p^{1+2\eps}}^{p} \, \d u \Rr)^{\frac{1}{p}}\\ 
& \quad +  K^7 (t-s)^{\frac{1}{8} }  
   \Bigg[ 1 + \|v_0\|_{\B^{-3\eps}_{p}}^3 +  \Ll( \int_0^t  \|w(u)\|_{L^{3p}}^{3p} \, \d u\Rr)^\frac 1 p \\
   & \hspace{7cm} + \Ll(\int_0^t \|w(u)\|_{\B^{1+4\eps}_p}^p  \, \d u \Rr)^\frac 1 p  \Bigg] \\
& \quad +   c K^4(t-s)^{\frac 1 8}    \Bigg[1+ \|v_0\|_{\B^{-3\eps}_p} +   \Ll( \int_0^t  \,\|   w(u)\|_{\B^{1 + 2 \eps}_p}^p  \d u \Rr)^{\frac{1}{p}}\\
& \hspace{7cm} +\Ll(   \int_0^t   \|w(u)\|_{L^p}^{3p}  \d u  \Rr)^{\frac{1}{3p}} \Bigg],
\end{align*}
where we recall that $\tn w\tn_{p,t}$ is defined in \eqref{e:def:D}, and that this quantity is finite by \eqref{e.smooth.consequence}. Using that $p \ge \frac 8 7$, the comparisons $\|\, \cdot \, \|_{\B^{\frac 1 2  + 2\eps}_p} \ls \|\, \cdot \, \|_{\B^{1+2\eps}_{p}} \ls \|\, \cdot \, \|_{\B^{1+4\eps}_{p}}$
and that 
$$
x \le a + \sqrt{bx} \implies x \ls a + b,
$$
we obtain
\begin{multline*}
\frac{\|\de_{st}' w\|_{L^p}}{(t-s)^{\frac 1 8}}  \ls c K^7 \Bigg[
1 + \|v_0\|_{\B^{-3\eps}_{p}}^{3}  +   \Ll( \int_0^t \| w(u)    \|_{\B_p^{1+4\eps}}^{p} du \Rr)^{\frac{1}{p}} \\
+  \Ll(\int_0^t \|w(u)\|_{L^{3p}}^{3p} \, \d u \Rr)^{\frac{1}{p}}
\Bigg].
\end{multline*}
To conclude, we observe that 
by Proposition~\ref{p:smooth-besov} and Remark~\ref{r:Besov-vs-Lp}, we have
\begin{equation}
\label{e.elem.continuity.delta'}
\big|\|\de'_{st} w\|_{L^p} - \|\de_{st} w\|_{L^p}\big| \le \|(1-e^{(t-s)\Dd})w(s)\|_{L^p} \ls (t-s)^{\frac 1 8} \|w(s)\|_{\B^{\frac 1 4+\eps}_p},
\end{equation}
and then use the crude bound $\|\cdot\|_{\B^{\frac 1 4 + \eps}_p} \ls \|\cdot\|_{\B^{1 + 4\eps}_p}$.
\end{proof}

%
%
%
%
%
%
\section{Higher regularity for \texorpdfstring{$w$}{w}}
\label{s:Gronwall-w}

In this section, we use the regularizing properties of the heat semigroup once more to estimate $w$ in a norm with an exponent of regularity larger than $1$. Such an information is necessary to control the behaviour of the term $w \pe \<2>$.

\begin{thm}\label{t:Gronwall-w}
Let $p \ge 24$, $\eps>0$ be sufficiently small, and $0< \gamma< \frac{4}{3}$. For every $t \in [0, T)$, we have
\begin{align}
\notag
\| w(t) \|_{\B_p^{\gamma}} 
&
\lesssim 
\| e^{t \Delta} w_0 \|_{\B^{\gamma}_{p}} \\
\notag
&+ 
cK^9  
\Ll[
	1+ 
	\Ll( 
		\int_0^t 
			\| w(s) \|_{L^{3p}}^{3p}  
		\, \d s 
	\Rr)^{\frac{1}{p}} + 
	\Ll(
		\int_0^t 
			\| w(s) \|_{\B^{1+4\eps}_p}^p 
		\, \d s
	\Rr)^{\frac{1}{ p}} + 
	\| v_0 \|_{\B_{2p}^{-3\eps}}^3  
\Rr]
\\
\label{e:prop-Gronwall}
&+ c K^3 t^{\frac{1}{2}(1-\gamma) - 3\eps -\frac{1}{3p}} 
\Ll[ 
	1+ 
	\| v_0 \|_{\B^{-3\eps}_{2p}}^2 +  
	\Ll( 
		\int_0^t 
			\| w(s) \|_{L^{3p}}^{3p} 
		\, \d s 
	\Rr)^{\frac{2}{3p}} 
\Rr],
\end{align}
where the implicit constant depends on $p$, $\eps$ and $\gamma$, but neither on $K$ nor on $c$ satisfying~\eqref{e.lowerbound.c}.
\end{thm}
Recall that we are mainly intersted in the case $\ga > 1$, in which case the power of~$t$ appearing in \eqref{e:prop-Gronwall} is negative. However, we will use this theorem in the form of the following corollary, in which diverging powers of $t$ no longer appear. Note that this requires a rather fine control on the excess of the exponent in the diverging power of $t$ in \eqref{e:prop-Gronwall}, and particular attention needs to be paid to this aspect in the proof of Lemma~\ref{le:uqt} below.
\begin{cor}
\label{cor:high-reg-w}
Let $p \ge 24$ and $\eps>0$ be sufficiently small. There exists an exponent~$\kappa < \infty$ depending on $p$ and $\eps$ such that for every $t \in [0,T)$,
\begin{multline*}
\Ll(
	\int_0^t 
		\|w(s) \|_{\B^{1+7\eps}_p}^p 
	\, \d s 
\Rr)^{\frac{1}{p}}
\\
\lesssim 
%
	\Ll(\int_0^t \|e^{s\Delta} w_0 \|_{\B^{1+7\eps}_p}^p \, \d s \Rr)^\frac 1 p +  
(c K)^\kappa 
\Ll[
	1+
	\Ll( 
		\int_0^t 
			\| w(s) \|_{L^{3p}}^{3p}  
		\, \d s 
	\Rr)^{\frac{1}{p}}  +	 
	\| v_0 \|_{\B_{2p}^{-3\eps}}^3  
\Rr],
\end{multline*}
where the implicit constant depends on $p$ and $\eps$, but neither on $K$ nor on $c$ satisfying~\eqref{e.lowerbound.c}.
\end{cor}
\begin{proof}[Proof of Corollary~\ref{cor:high-reg-w}]
We use Theorem~\ref{t:Gronwall-w} in the simplified version
\begin{multline*}
\| w(s) \|_{\B_p^{1+7\eps}} 
\lesssim 
\| e^{s \Delta} w_0 \|_{\B^{1+7\eps}_{p}} 
\\
+ 
c K^9  s^{-\lambda} 
\Big[	1+\Big(
		\int_0^t 
			\| w(r) \|_{\B^{1+4\eps}_p}^p 
		\d r
	\Big)^{\frac{1}{ p}} + 
	\Big( 
		\int_0^t 
			\| w(s) \|_{L^{3p}}^{3p}  
		\d s 
	\Big)^{\frac{1}{p}} +  
	\| v_0 \|_{\B_{2p}^{-3\eps}}^3  
\Big],
\end{multline*}
where $-\lambda := \frac{1}{2}(1-(1+7\eps)) - 3\eps -\frac{1}{3p} = -\frac{13\eps}{2} -\frac{1}{3p}$.
By the definition of $\lambda$ and the choice of $\eps>0$ small enough ($\eps < \frac{4}{39 p}$ is sufficient), we see that 
\begin{align*}
\Big( \int_0^t s^{-\lambda p} \, \d s \Big)^{\frac{1}{p}} \lesssim 1,
\end{align*}
and thus deduce that 
\begin{multline*}
\Big(
	\int_0^t 
		\|w(s) \|_{\B^{1+7\eps}_p}^p 
	\d s 
\Big)^{\frac{1}{p}}
\lesssim 
	\Ll(\int_0^t \|e^{s\Delta} w_0 \|_{\B^{1+7\eps}_p}^p \, \d s \Rr)^\frac 1 p  
\\
+ c K^9 	
\Big[
	1+
	\Big(
	\int_0^t 
		\| w(s) \|_{\B^{1+4\eps}_p}^p 
	\d s
	\Big)^{\frac{1}{ p}} + 
	\Big( 
		\int_0^t 
			\| w(s) \|_{L^{3p}}^{3p}  
		\d s 
	\Big)^{\frac{1}{p}} +  
	\| v_0 \|_{\B_{2p}^{-3\eps}}^3  
\Big].
\end{multline*}
It only remains to remove the term involving $\| w(s) \|_{\B^{1+4\eps}_p}^p $ on the right-hand side. By Proposition~\ref{p:interpol}, Remark~\ref{r:Besov-vs-Lp} and Young's inequality, we have the interpolation bound
\begin{equation}  
\label{e.interpol.4eps}
\|w\|_{\B^{1+4\eps}_p} \le \|w\|_{\B^{1+7\eps}_p}^{\frac{1+4\eps}{1+7\eps}} \,\|w\|_{L^p}^{\frac{3\eps}{1+7\eps}} \ls \|w\|_{\B^{1+5\eps}_p}^{\frac{1+4\eps}{1+6\eps}} + \|w\|_{L^{3p}}^{3},
\end{equation}
and thus
\begin{align*}
	 \Big(\int_0^t \| w(s) \|_{\B^{1+4\eps}_p}^p  \, \d s\Big)^\frac 1 p
	\lesssim 
	\Big( \int_0^t \| w(s) \|_{\B^{1+5\eps}_p}^p \, \d s\Big)^{\frac 1 p{\frac{1+4\eps}{1+6\eps}}}
	+ \Big(\int_0^t  \| w(s) \|_{L^{3p}}^{3p} \, \d s\Big)^{\frac{1}{p}}.
\end{align*} 
Since ${\frac{1+4\eps}{1+6\eps}} < 1$, an application of Young's inequality then yields the result.
\end{proof}

We now proceed to the proof of Theorem~\ref{t:Gronwall-w}.
We use again the decomposition \eqref{dest-4} (setting $s = 0$ there), and proceed to bound the terms one by one in the following lemmas. 
Although we do not repeat it each time, the implicit constants in these lemmas depend neither on $K$ nor on $c$ satisfying~\eqref{e.lowerbound.c}.

\begin{lem}\label{le:W1}
Let $p \ge 24$, $\eps>0$ be sufficiently small, and let $0<\gamma <\frac32$. For every $ t \in [0,T)$, we have
\begin{equation*}
\Ll\|\int_0^t  e^{\Delta(t-\tau)} \, (w+v)^3(\tau) \,  \d \tau \Rr\|_{\B^\gamma_p}   
\lesssim K^3 
\Big(  \int_0^t  \| w(r) \|_{L^{3p}}^{3p} \d r\Big)^{\frac{1}{p}}   +\| v_0 \|_{\B^{-3 \eps}_p}^3 + K^6  ,
\end{equation*}
where the implicit constant depends on $p$, $\eps$ and $\gamma$.
\end{lem}

\begin{proof}
We start by observing that by Proposition~\ref{p:smooth-besov}, we have for any $\tau < T$ 
\begin{align*}
 \| e^{\Delta(t-\tau)} \, (w+v)^3(\tau)  \|_{\B^\gamma_p} \lesssim \frac{1}{(t-\tau)^{\frac{\gamma}{2}}} \big(\| w(\tau) \|_{L^{3p}}^3 + \| v(\tau) \|_{L^{3p}}^3\big).
\end{align*}
We proceed by bounding the integrals over the expressions involving $w$ and $v$ one by one. For $w$, we use H\"older's inequality in the form
\begin{align*}
\int_0^t  \frac{1}{(t-\tau)^{\frac{\gamma}{2}}} \| w(\tau) \|_{L^{3p}}^3 \, \d \tau \leq \Big( \int_0^t  \frac{1}{(t-\tau)^{\frac{\gamma p'}{2}}} \d \tau \Big)^{\frac{1}{p'}} \Big( \int_0^t \| w(\tau) \|_{L^{3p}}^{3p}\d \tau\Big)^{\frac{1}{p}},
\end{align*}
where $\frac{1}{p} + \frac{1}{p'}=1$. The first integral on the right-hand side is finite if and only if $\frac{\gamma p'}{2}<1$, which amounts to $p > \frac{2}{2-\gamma}$. This condition is clearly satisfied, since $p \ge 4$.

\smallskip

For the integral involving $v$, we use H\"older's inequality again, but this time in the form 
\begin{align*}
\int_0^t  \frac{1}{(t-\tau)^{\frac{\gamma}{2}}} \| v(\tau) \|_{L^{3p}}^3 \d \tau \leq \Big( \int_0^t  \frac{1}{(t-\tau)^{\frac{\gamma q}{2}}} d \tau \Big)^{\frac{1}{q}} \Big( \int_0^t \| v(\tau) \|_{L^{3p}}^{3 \frac{p}{6}} \d \tau \Big)^{\frac{6}{p}},
\end{align*}
where $\frac{1}{q} + \frac{6}{p} = 1$. This time the condition for the first integral on the right-hand side to be finite reads $\frac{\gamma q}{2} <1$, or equivalently $p> \frac{12}{2-\gamma}$. Using again that $\gamma < \frac32$, we see that this condition is satisfied for $p \geq 24$.
For the second integral on the right-hand side, we use Theorem~\ref{t:apriori-v} (in the form given by Remark~\ref{r:apriori-v1-Lp}) to get
\begin{align*}
 \Big( \int_0^t \| v(\tau) \|_{L^{3p}}^{ \frac{p}{2}}\d \tau\Big)^{\frac{6}{p}} 
 &\lesssim 
 \Big(\int_0^t \big(\tau^{-2 \eps - \frac{1}{p}}\big)^{\frac{p}{2}} \d \tau \Big)^{\frac{6}{p}} \| v_0\|_{\B^{-3 \eps}_p}^3 
 \\
 & \qquad +  
 \Big(
 	\int_0^t 
		K^{\frac{p}{2}} \Big(
			\int_0^\tau 
				\frac{1}{(\tau-r)^{\frac12+\eps}}( \| w(r) \|_{L^{3p}} +K )
			\d r 
		\Big)^{\frac{p}{2}}  
	\d \tau 
\Big)^{\frac{6}{p}}.
\end{align*}
The first integral on the right-hand side is finite as soon as $\eps < \frac{1}{2p}$. We estimate the second expression using Jensen's inequality:
\begin{align*}  
&  \Big(
 	\int_0^t 
		K^{\frac{p}{2}} \Big(
			\int_0^\tau 
				\frac{1}{(\tau-r)^{\frac12+\eps}}( \| w(r) \|_{L^{3p}} +K )
			\d r 
		\Big)^{\frac{p}{2}}  
	\d \tau 
\Big)^{\frac{6}{p}}
\\
& \qquad \ls  
K^3 \Big(
 	\int_0^t 
		\int_0^\tau 		
				\frac{1}{(\tau-r)^{\frac12+\eps}}( \| w(r) \|_{L^{3p}} +K )^\frac p 2
			\,\d r  \,
	\d \tau 
\Big)^{\frac{6}{p}}\\
& \qquad \ls 
K^3 \Big(
 	\int_0^t 
		( \| w(s) \|_{L^{3p}} +K )^{3p}
			\, \d s
\Big)^{\frac{1}{p}}.
\end{align*}
The desired estimate thus follows.
\end{proof}
\begin{lem}\label{le:W2} 
Let $p \ge 24$, $\eps>0$ be sufficiently small, and $0<\gamma < \frac{3}{2}$. For every $ t \in [0,T)$, we have
\begin{multline}
\label{e:lem5:3-aim}
\Ll\|\int_0^t   e^{\Delta (t-\tau)}[\com_1(v,w) \pe \<2>](\tau)  \, \d \tau   \Rr\|_{\B_p^\gamma}
\lesssim 
K t^{\frac{1}{2}(1-\gamma)-\frac{5\eps}{2} } \| v_0 \|_{\B^{-3\eps}_p} \\
+  c K^9 
\Big[
	1 + \| v_0 \|_{\B^{-3\eps}_{p}}^3+
	\Big(\int_0^t \| w(s)\|_{L^{3p}}^{3p} \d s \Big)^{\frac{1}{p}} + 
	\Big( \int_0^t \| w(s) \|_{\B^{1+4\eps}_p}^p \d s \Big)^{\frac{1}{p}} 
\Big],
\end{multline}
where the implicit constant depends on $p$, $\eps$ and $\gamma$.
\end{lem}
%
%
%
\begin{proof}
As before, we start by observing that by Proposition~\ref{p:smooth-besov} and Remark~\ref{r:Besov-vs-Lp},
\begin{align*}
\int_0^t  \| e^{\Delta (t-\tau)}\com_1(v,w) \pe \<2>(\tau) \|_{\B_p^\gamma} \, \d \tau
\lesssim 
  \int_0^t \frac{1}{(t-\tau)^{\frac{\gamma}{2}}} \, \| \com_1(v,w) \pe \<2> \|_{L^p}(\tau)  \, \d \tau  .
 %
\end{align*}
In order to bound this integral, we first observe that
$$
\| \com_1(v,w) \pe \<2> \|_{L^p} \lesssim K \| \com_1(v,w)  \|_{\B_p^{1+2\eps}},
$$
then 
use the decomposition $ \com_1(v,w) (\tau)  =  (\com_1(v,w) (\tau) - e^{\tau\Delta}v_0 ) + e^{\tau\Delta}v_0$ and recall that according to Lemma~\ref{l:estim-com1}, we have
\begin{align}
\notag
&\| \com_1(v,w) (\tau) - e^{\tau\Delta}v_0 \|_{\B^{1+2\eps}_p} \\
\notag
&\lesssim 
K^3 + K(K+c) \tau^{-4\eps}\| v_0\|_{\B_p^{3\eps}} 
+ \int_0^\tau \frac{K^2}{(\tau-s)^{\frac 3 4 + \eps}} \| w(s) \|_{\B^{\frac 1 2 + 2 \eps}_p} \, \d s \\
\notag
&+ \int_0^\tau \frac{K}{(\tau-s)^{1+2\eps}} \| \delta_{s\tau} w \|_{L^p} \, \d s\\
\label{e:lem5.3-1}
&=: (I_1 + I_2 + I_3 + I_4)(\tau).
\end{align}
We proceed by bounding these terms one by one, starting with the integral involving $e^{\tau\Delta}v_0$. We get
\begin{align*}
\int_0^t \frac{1}{(t - \tau )^{\frac{\gamma}{2}}} \| e^{\Delta \tau} v_0 \|_{\B_p^{1+2\eps}} \d \tau 
& \lesssim \| v_0 \|_{\B_p^{-3\eps}} \int_0^t \frac{1}{(t-\tau)^{\frac{\gamma}{2}}} \frac{1}{\tau^{\frac{1+5\eps}{2}}} \d \tau\\
& \lesssim t^{1- \frac12(\gamma +1+5\eps)} \| v_0 \|_{\B^{-3\eps}_p},
\end{align*}
thus resulting in the first term on the right-hand side of \eqref{e:lem5:3-aim}. It is worth observing here  that as we are mostly interested in $\gamma >1$, the resulting exponent of $t$ 
is negative. However, as both exponents $\frac{\gamma}{2}$ and $\frac{1+5\eps}{2}$ individually are strictly less than $1$, this does not affect the finiteness of the integral.

\smallskip

For the integrals involving each of the terms listed on the right-hand side of \eqref{e:lem5.3-1}, we write
\begin{align*}
&\int_0^t \frac{1}{(t-\tau)^{\frac{\gamma}{2}}}  I_j (\tau) \d \tau 
\lesssim 
\Big( \int_0^t \frac{1}{(t-\tau)^{\frac{p' \gamma}{2}}} \d \tau \Big)^{\frac{1}{p'}} 
  \Big(\int_0^t I_j^p(\tau) \d \tau \Big)^{\frac{1}{p}},
\end{align*}
where $\frac{1}{p} + \frac{1}{p'} = 1$ and $j =1,2,3,4$. The first integral on the right-hand side is finite by our conditions on $p$ and $\gamma$, and it thus remains to bound the temporal $L^p$ norms of the $I_j$. For the first two terms, we get 
\begin{align*}
\Big( \int_0^t I_1^p(\tau) \d \tau \Big)^{\frac{1}{p}} &\le K^3 t^{\frac{1}{p}},\\
\Big( \int_0^t I_2^p(\tau) \d \tau \Big)^{\frac{1}{p}} &\lesssim K(K+c) \| v_0\|_{\B^{-3\eps}_p}.
\end{align*}
In the second identity we have used that $4 \eps < \frac{1}{p}$. For the integral involving $I_3$, we write
\begin{align*}
\Big(\int_0^t I_3^p(\tau) \d \tau \Big)^{\frac{1}{p}}
&= \Big( 
	\int_0^t 
		\Big( 
			\int_0^\tau
				\frac{K^2}{(\tau-s)^{\frac 3 4 + \eps}} \| w(s) \|_{\B^{\frac 1 2 + 2 \eps}_p} 
			\d s
		\Big)^{p}
	\d \tau
\Big)^{\frac{1}{p}}\\
&\lesssim  
\Big(\int_0^t \frac{K^2}{(t-s)^{\frac 3 4 + \eps}} \d s \Big)
\Big( \int_0^t \| w(s) \|_{\B^{\frac 1 2 + 2 \eps}_p}^p \d s \Big)^{\frac{1}{p}},
\end{align*}
and this term is controlled by the last expression on the right-hand side of \eqref{e:lem5:3-aim}. 
The last term $I_4$ requires to invoke Theorem~\ref{t:apriori-dw}, which in our context states that 
\[
\| \delta_{s\tau} w\|_{L^p} 
\lesssim 
cK^7 (\tau-s)^{\frac{1}{8}}\big[ T(t) +\| w(s) \|_{\B_p^{1+4\eps}}\big],
\] 
where 
\[
T(t) = 
1 + \| v_0\|_{\B^{-3\eps}_{p}}^3+
\Big( \int_0^t \|w(s) \|_{\B^{1+4\eps}_p}^p \d s \Big)^{\frac{1}{p}} +
\Big( \int_0^t \| w(s) \|_{L^{3p}}^{3p} \d s \Big)^{\frac{1}{p}} .
\]
Using this bound, we get
\begin{align*}
\Big(\int_0^t I_4^p(\tau) \d \tau \Big)^{\frac{1}{p}} 
&\lesssim c K^8
\Big(
\int_0^t 
	\Big( 
		\int_0^{\tau} 
			\frac{1}{(\tau-s)^{\frac{7}{8} +2\eps  }} 
			\big[ T(t) + \| w(s) \|_{\B_p^{1+4\eps}} \big]
		\d s 
	\Big)^{p}
\d \tau 
\Big)^{\frac{1}{p}}\\
&\lesssim 
cK^8 
\Big( \int_0^t \frac{1}{(t-s)^{\frac{7}{8} +2\eps}} \d s \Big)
\Big[ T(t) + \Big( \int_0^t \| w(s) \|^p_{\B_p^{1+4\eps}} \d s \Big)^{\frac{1}{p}}\Big],
\end{align*}
which concludes the proof of Lemma~\ref{le:W2}.
\end{proof}

\begin{lem}
Let $p \ge 24$, $\eps>0$ be sufficiently small, and $0<\gamma<\frac32$. For every $t \in [0,T)$,
\begin{align*}
\Ll\|\int_0^t  e^{\Delta(t-\tau)}   [w \pe \<2>]  (\tau)    \, \d \tau \Rr\|_{\B_p^\gamma}  \lesssim  K  \Big( \int_0^t \| w \|_{\B^{1+2\eps}_p}^{p} d \tau \Big)^{\frac{1}{p}},
\end{align*}
where the implicit constant depends on $p$, $\eps$ and $\gamma$.
\end{lem}
\begin{proof}
We use Proposition~\ref{p:smooth-besov} to write
\begin{align*}
 \int_0^t \| e^{\Delta(t-\tau)}   w \pe \<2>  (\tau)  \|_{\B^p_\gamma}   \, \d \tau
& \lesssim   \int_0^t \frac{1}{(t-\tau)^{\frac{\gamma}{2}}} \, \| w \pe \<2>\|_{L^p} (\tau) \, \d \tau \\
&  \lesssim  \Ll( \int_0^t \frac{1}{ (t-\tau)^{  \frac{p' \gamma}{2}   }      }  \, \d \tau   \Rr)^{\frac{1}{p'}}
 \Big( \int_0^t  \| w \pe \<2>\|_{L^p}^p (\tau)  \, \d \tau\Big)^{\frac{1}{p}}  ,
\end{align*}
where $\frac{1}{p} + \frac{1}{p'}  =1$. As already seen, the first integral is finite under our assumptions on $p$ and $\gamma$.
In order to treat the second term on the right-hand side, we use the multiplicative inequality in Proposition~\ref{p:mult} to get that 
for every fixed $\tau$,
\begin{align*}
 \int_0^t \| w \pe \<2>\|_{L^p}^p (\tau)  \d \tau \lesssim K \int_0^t \| w(\tau) \|_{\B^{1+2\eps}_p}^p \d \tau,
\end{align*}
so that the conclusion follows.
\end{proof}

\begin{lem}\label{le:uqt}
Let $p \ge 24$, $\eps>0$ be sufficiently small, and $0< \gamma < \frac4 3$.
For every $ t \in [0,T)$, we have
\begin{align}
\notag
& \Ll\|\int_0^t  e^{\Delta(t-\tau)}  a_2(v+w)^2(\tau)(\tau) \, \d \tau \Rr\|_{\B^{\gamma}_p}  \\
& \qquad \lesssim  K^5 \Big[1+ \Big( \int_0^t \| w(s) \|_{L^{3p}}^{3p}  \d s \Big)^{\frac{1}{p}} + \Big(\int_0^t \| w(s) \|_{\B^{1+4\eps}_p}^p \d s \Big)^{\frac{1}{ p}} + \| v_0 \|_{\B_{2p}^{-3\eps}}^3  \Big]
\\
\label{uik}
&\qquad \qquad + K^3 t^{\frac{1}{2}(1-\gamma) - 3\eps -\frac{1}{3p}} 
\Big[ 
	1+ 
	\| v_0 \|_{\B^{-3\eps}_{2p}}^2 +  
	\Big( 
		\int_0^t 
			\| w(s) \|_{L^{3p}}^{3p} 
		\d s 
	\Big)^{\frac{2}{3p}} 
\Big],
\end{align}
where the implicit constant depends on $p$, $\eps$ and $\gamma$.
\end{lem}
%
%
%
\begin{proof}
We start by writing
\begin{multline}
\label{e:le55-0}
\int_0^t \| e^{\Delta(t-\tau)}  a_2(v+w)^2(\tau)\|_{\B^{\gamma}_p} \, \d \tau \\
\lesssim K \int_0^t \frac{1}{(t-\tau)^{\frac12(\gamma+\frac12 +\eps)  }} \| (v+w)^2(\tau) \|_{\B^{\frac12 +2 \eps}_p} \d \tau,
\end{multline}
where we have made use of $\| a_2 (v+w)^2 \|_{\B_p^{-\frac12 -\eps}} \lesssim K \| (v+v)^2 \|_{\B^{\frac12+2\eps}_p}$. We now apply Proposition~\ref{p:mult} in the form
\begin{align}\label{e:le55-1}
\| (v+w)^2 \|_{\B^{\frac12+2\eps}_p} \lesssim \| v+w \|_{L^{3p}} \| v+w \|_{\B_{\frac{3p}{2}}^{\frac12+2\eps}} ,
\end{align}
as well as the bounds \eqref{e:apriori-v1} and Remark~\ref{r:apriori-v1-Lp} which yield 
\begin{align}
\notag
\| v(t)\|_{L^{3p}} &\lesssim \frac{1}{t^{2\eps + \frac 1 {4p}}} \| v_0 \|_{\B^{-3\eps}_{2p}} + 
K \int_0^t \frac{1}{(t-s)^{\frac12 +\eps}} \big(\| w(s) \|_{L^{3p}} +K \big) \d s \\
\label{e:le55-2A}
&=: A_0(t) + A_1(t),\\
\notag
\| v(t)\|_{\B^{\frac12+2\eps}_{\frac{3p}{2}}} &\lesssim \frac{1}{t^{\frac14+\frac{5\eps}{2}}} \|v_0 \|_{\B^{-3\eps}_{2p}} +
K \int_0^t \frac{1}{(t-s)^{\frac34+3\eps}} \big(\| w(s) \|_{L^{3p}} +K \big) \d s\\
\label{e:le55-2B}
& =: B_0(t) +B_1(t),
\end{align}
so that the expression in \eqref{e:le55-1} can be rewritten as
\begin{align}
\notag
&\| (v+w)^2 \|_{\B^{\frac12+2\eps}_p} \\
\notag
& \lesssim 
\big( A_0 + A_1 + \| w \|_{L^{3p}}\big) \big( B_0+B_1 + \| w \|_{\B^{\frac12+2\eps}_{\frac{3p}{2}}} \big)\\
\notag
&= \big( A_1 + \| w \|_{L^{3p}} \big) \big( B_1 + \| w \|_{\B^{\frac12+2\eps}_{\frac{3p}{2}}} \big) 
+ A_0 \big( B_1 + \| w \|_{\B^{\frac12+2\eps}_{\frac{3p}{2}}} \big) \\
\label{le:55-2C}
&\qquad +  \big( A_1 + \| w \|_{L^{3p}} \big)B_0 + A_0 B_0.
\end{align}
We now plug these bounds into the right-hand side of \eqref{e:le55-0} and treat the resulting terms one by one, using the shorthand notation 
\begin{equation}
\label{e.def.gamma'}
\gamma' := \frac12\Ll(\gamma+\frac12 +\eps\Rr).
\end{equation}
We first get
\begin{align}
\notag
&\int_0^t \frac{1}{(t-\tau)^{\gamma'  }}\big( A_1(\tau) + \| w(\tau) \|_{L^{3p}} \big) 
\big( B_1(\tau) + \| w(\tau) \|_{\B^{\frac12+2\eps}_{\frac{3p}{2}}} \big)  \d \tau\\
\notag
&\lesssim \Big( \int_0^t \frac{1}{(t-\tau)^{q \gamma'  }} \d \tau\Big)^{\frac{1}{q}}
\Big[
\Big( \int_0^t A_1(\tau)^{3p} \d \tau \Big)^{\frac{1}{3p}} + \Big(\int_0^t \| w(\tau) \|_{L^{3p}}^{3p} \d \tau \Big)^{\frac{1}{3p}}
\Big]\\
\label{e:le55-3}
&\times
\Big[
\Big( \int_0^t B_1(\tau)^{p} \d \tau \Big)^{\frac{1}{p}} + \Big(\int_0^t \| w(\tau) \|_{\B^{\frac12+2\eps}_{\frac{3p}{2}}}^{p} \d \tau \Big)^{\frac{1}{p}}
\Big],
\end{align}
where we have set $q:= \frac{3p}{3p-4}$, so that $1 = \frac{1}{q} + \frac{1}{3p} + \frac{1}{p}$. The first integral on the \rhs\ above is finite if and only if
$$
p>\frac{8}{3}\Ll(\frac32-\gamma-\eps\Rr)^{-1}.
$$
This condition is implied by our assumptions of $\gamma < \frac 4 3$ and $p \ge 16$, provided that $\eps > 0$ is sufficiently small.
 Applying Young's inequality to the definitions \eqref{e:le55-2A} and \eqref{e:le55-2B}, we get
\begin{align}
\label{e:le55-4a}
\Big( \int_0^t A_1(\tau)^{3p} \d \tau \Big)^{\frac{1}{3p}} &\lesssim K t^{\frac12-\eps}
\Big[ \Big( \int_0^t \| w(s) \|_{L^{3p}}^{3p}  \d s \Big)^{\frac{1}{3p}}+K \Big],\\
\label{e:le55-4}
\Big( \int_0^t B_1(\tau)^{p} \d \tau \Big)^{\frac{1}{p}} &
\lesssim K t^{\frac14-3\eps} 
\Big[ \Big( \int_0^t \| w(s) \|_{L^{3p}}^{p}  \d s \Big)^{\frac{1}{p}}+K \Big].
\end{align}
To control the  last term on the right-hand side of \eqref{e:le55-3}, we
first make use of Proposition~\ref{p:interpol}, in the form of the interpolation bound
\begin{equation}
\label{e.interp.wwww}
\| w \|_{\B^{\frac12+2\eps}_{\frac{3p}2}} \lesssim \| w \|_{L^{3p}}^{\frac12} \| w \|^{\frac12}_{\B^{1+4\eps}_p},
\end{equation}
and then of
H\"older's inequality to get 
\begin{align*}
\Big(\int_0^t \| w(\tau) \|_{\B^{\frac12+2\eps}_{\frac{3p}2}}^{p} \d \tau \Big)^{\frac{1}{p}}
&\lesssim \Big(\int_0^t \| w(\tau) \|_{\B^{1+4\eps}_p}^{\frac{3 p}{5}} \d \tau \Big)^{\frac12 \frac{5}{3 p}} \Big( \int_0^t \|w(\tau) \|_{L^{3p}}^{3p} \d \tau 
\Big)^{\frac12 \frac{1}{3p}}\\
&\lesssim \Big(\int_0^t \| w(\tau) \|_{\B^{1+4\eps}_p}^p \d \tau \Big)^{\frac{1}{2 p}} \Big( \int_0^t \|w(\tau) \|_{L^{3p}}^{3p} \d \tau 
\Big)^{\frac12 \frac{1}{3p}}.
\end{align*}
We also observe that by Young's inequality,
\begin{multline*}  
\Big( \int_0^t \|w(\tau) \|_{L^{3p}}^{3p} \d \tau 
\Big)^{\frac32 \frac{1}{3p}} \, \Big(\int_0^t \| w(\tau) \|_{\B^{1+4\eps}_p}^p \d \tau \Big)^{\frac{1}{2 p}} \\
\ls  \Big( \int_0^t \|w(\tau) \|_{L^{3p}}^{3p} \d \tau 
\Big)^{\frac{1}{p}} + \Big(\int_0^t \| w(\tau) \|_{\B^{1+4\eps}_p}^p \d \tau \Big)^{\frac{1}{ p}} 
\end{multline*}
Combining these calculations with \eqref{e:le55-3}, we obtain
\begin{align}
\notag
&\int_0^t \frac{1}{(t-\tau)^{\gamma'  }}\big( A_1(\tau) + \| w(\tau) \|_{L^{3p}} \big) 
\big( B_1(\tau) + \| w (\tau)\|_{\B^{\frac12+2\eps}_p} \big)  \d \tau\\
\notag
&\lesssim 
K^4\Big[1+ \Big( \int_0^t \| w(s) \|_{L^{3p}}^{3p}  \d s \Big)^{\frac{1}{p}} + \Big(\int_0^t \| w(s) \|_{\B^{1+4\eps}_p}^p \d s \Big)^{\frac{1}{ p}}  \Big].
\end{align}
It remains to bound the terms involving $A_0$ and $B_0$ (i.e. the initial datum $v_0$) on 
the right-hand side of  \eqref{le:55-2C}.
We write
\begin{align*}
&\int_0^t 
	\frac{1}{(t-\tau)^{\gamma'}} 
	A_0(\tau) 
	\big(   B_1(\tau) + \|w(\tau) \|_{\B_{\frac{3p}{2}}^{\frac12+2\eps}}   \big) 
\d \tau\\
&\lesssim 
\| v_0 \|_{\B_{2p}^{-3\eps}} 
\int_0^t 
	\frac{1}{(t-\tau)^{\gamma'}} 
	\frac{1}{\tau^{2\eps + \frac 1 {4p}}} 
	\big( B_1(\tau) + \|w(\tau) \|_{\B_p^{1+4\eps}}^{\frac12} \|w(\tau) \|_{L^{3p}}^{\frac12}\big)
\d \tau \\
&\lesssim 
\| v_0 \|_{\B^{-3\eps}_{2p}} 
	\Big(
		\int_0^t  
			\Big(
				\frac{1}{(t-\tau)^{\gamma'}} 
				\frac{1}{\tau^{2\eps + \frac 1 {4p}}}
			\Big)^{p'}
		\d \tau 
	\Big)^{\frac{1}{p'}}\\
&\qquad \times
\Big[ 
	\Big(
		\int_0^t  
			B_1^p(\tau) 
		\d \tau 
	\Big)^{\frac{1}{p}} 
	+ 
	\Big(
		\int_0^t 
			\|w(\tau) \|_{\B_p^{1+4\eps}}^{p} 
		\d \tau 
	\Big)^{\frac{1}{2p}}
	\Big( 
		\int_0^t 
			\|w(\tau) \|_{L^{3p}}^{p} 
		\d \tau
	\Big)^{\frac{1}{2p}}
\Big],
\end{align*}
where $\frac{1}{p}+ \frac{1}{p'}=1$, and where we used once more the 
interpolation bound \eqref{e.interp.wwww}. The first integral is bounded by $t^{1-\frac 1 p - \gamma'-2\eps - \frac 1 {4p}}$. Noting from the definition of $\gamma'$ in \eqref{e.def.gamma'} that $\gamma' < \frac {11}{12} + \eps$, we see that this exponent satisfies
\begin{equation*}  
1 - \frac{5}{4p} - \gamma'-2\eps \ge 0.
\end{equation*}
Using \eqref{e:le55-4} and Young's inequality, we thus conclude that
\begin{multline*}  
\int_0^t 
	\frac{1}{(t-\tau)^{\gamma'}} 
	A_0(\tau) 
	\big(   B_1(\tau) + \|w(\tau) \|_{\B_{\frac{3p}{2}}^{\frac12+2\eps}}   \big) 
\d \tau \\
\ls 
  K^3
\Big[ 
\| v_0 \|_{\B^{-3\eps}_{2p}}^3 
+ \Big( 
	\int_0^t \|w(\tau) \|_{L^{3p}}^{3p}
	\d \tau
\Big)^{\frac{1}{p}}
+ 
\Big(
	\int_0^t \|w(\tau) \|_{\B_p^{1+4\eps}}^{p} \d \tau 
\Big)^{\frac{1}{p}} 
\Big].
\end{multline*}
Similarly, using the definition \eqref{e:le55-2B} of $B_0$, then H\"older's inequality, and then \eqref{e:le55-4a}, we get
\begin{align*}
&\int_0^t \frac{1}{(t-\tau)^{\gamma'}}  \big( A_1(\tau) + \| w(\tau) \|_{L^{3p}} \big)B_0(\tau) \d \tau\\
&\lesssim 
\| v_0 \|_{\B^{-3\eps}_{2p}} 
\int_0^t 
	\frac{1}{(t-\tau)^{\gamma'}} 
	\frac{1}{\tau^{\frac14+\frac{5\eps}{2}}}  
	\big( 
		A_1(\tau) + \| w(\tau) \|_{L^{3p}} 
	\big) 
\d \tau \\
&\lesssim 
\| v_0 \|_{\B^{-3\eps}_{2p}}  
	\Big(
		\int_0^t 
			\Big(\frac{1}{(t-\tau)^{\gamma'}} 
			\frac{1}{\tau^{\frac14+\frac{5\eps}{2}}}\Big)^{p'}
		\d \tau 
	\Big)^{\frac{1}{p'}} \\
&\qquad \times 	
\Big[ 
	\Big(
		\int_0^t 
			A_1^p(\tau) 
		\d \tau 
	\Big)^{\frac{1}{p}} 
+ 
	\Big( 
		\int_0^t 
			\| w(\tau) \|_{L^{3p}}^p 
		d \tau 
	\Big)^{\frac{1}{p}} 
\Big]\\
&\lesssim 
\| v_0 \|_{\B^{-3\eps}_{2p}} 
t^{1-\frac{1}{p} - \gamma' - \frac{1}{4} - \frac{5\eps}{2}} 
\Big[ 
	K t^{\frac12-\eps} 
	\Big( 
		\Big( 
			\int_0^t 
				\| w(s) \|_{L^{3p}}^{p}  
			\d s 
		\Big)^{\frac{1}{p}}
		+K 
	\Big)
\\
& \qquad \qquad  \qquad \qquad 
	+
	t^{\frac{2}{3p}}
	\Big( 
		\int_0^t 
			\| w(\tau) \|_{L^{3p}}^{3p} 
		d \tau 
	\Big)^{\frac{1}{3p}} 
\Big]
\\
&\lesssim 
K^2 t^{1-\frac{1}{3p} - \gamma' - \frac{1}{4} - \frac{5\eps}{2}} 
\Big[ 
	1+ 
	\| v_0 \|_{\B^{-3\eps}_{2p}}^2 +  
	\Big( 
		\int_0^t 
			\| w(\tau) \|_{L^{3p}}^{3p} 
		\d \tau 
	\Big)^{\frac{2}{3p}} 
\Big].
\end{align*}
Recalling the definition of $\gamma'$ in \eqref{e.def.gamma'}, we see that the exponent in the power of $t$ above can be rewritten as
\begin{equation}
\label{e.border.exp}
1-\frac{1}{3p} - \gamma' - \frac{1}{4} - \frac{5\eps}{2} = \frac{1}{2}(1-\gamma) - 3\eps -\frac{1}{3p}.
\end{equation}
Finally, for the last term we write, recalling the definitions 
\eqref{e:le55-2A} and \eqref{e:le55-2B} of $A_0$ and $B_0$,
\begin{align*}
&\int_0^t 
	\frac{1}{(t-\tau)^{\gamma'}} 
	A_0(\tau) B_0(\tau) 
\d\tau\\
&\lesssim 
\|v_0\|_{\B^{-3\eps}_{2p}}^2 
\int_0^t 
	\frac{1}{(t-\tau)^{\gamma'}}  
	\frac{1}{\tau^{\frac14+\frac{9\eps}{2}+\frac 1 {4p}}} 
\d \tau
\lesssim 
\|v_0\|_{\B^{-3\eps}_{2p}}^2 
t^{1 - \frac{1}{4p} - \gamma' - \frac14 - \frac{9\eps}{2}}.
\end{align*}
For $\eps > 0$ sufficiently small, the exponent in the power of $t$ above is smaller than that appearing in \eqref{e.border.exp}, so the proof is complete.
\end{proof}
%
%
%

%
\begin{lem}\label{le:W5}
Let $p \ge 24$, $\eps>0$ be sufficiently small, and $0 < \gamma < \frac 4 3$.
For every $ t \in [0,T)$, we have
\begin{multline*}
\Ll\|\int_0^t    e^{\Delta(t-\tau)}   [ \ \ldots \ ] (\tau) \, \d \tau \Rr\|_{\B^{\gamma}_p}
\lesssim 
cK^3 t^{\frac12(1-\gamma)-3\eps} \|v_0\|_{\B^{-3\eps}_p} \\
+
c K^5
\Big( 
	1 + 
	\Big(
		\int_0^t 
			\| w(s)\|_{\B^{1+4\eps}_p}^p
		\d s	
	\Big)^{\frac{1}{p}}
\Big)
\end{multline*}
where the dots $\ldots$ represent the terms spelled out explicitly in \eqref{e.recall.dots} below, and where the implicit constant depends on $p$, $\eps$ and $\gamma$.
\end{lem}
%
%
%
\begin{proof}
Recall that as in the previous section, the dots $\ldots$ represent the terms
\begin{equation}
\label{e.recall.dots}
\ldots =  -3 \msf{com}_2(v+w)   - 3(v+w-\<30>) \pg \<2> + a_0 +a_1(v+w) + c v.
\end{equation}
Using the definition~\eqref{e:def:com2} of $\msf{com}_2$ and  the bound provided in Proposition~\ref{p:comm1}, one can check that
\begin{align*}
 \| \ \ldots\ \|_{\B^{-\frac12-\eps}_p}(\tau) 
& \lesssim 
c K^3 
\Big(
	1+
	\| v(\tau) \|_{\B^{\frac{1}{2}+2\eps}_p} + 
	\| w(\tau) \|_{\B^{\frac{1}{2}+2\eps}_p}
\Big).
\end{align*}
This yields
\begin{multline*}
\int_0^t   
	\| 
	e^{\Delta(t-\tau)}   [\ \ldots\ ]
	\|_{\B^{\gamma}_p}(\tau) \,
	\d \tau\\
\lesssim 
c K^3 
	\int_0^t  
		\frac{1}{(t-\tau)^{\frac{2\gamma+1+2\eps}{4}  }} 
		\Big(
			1+
			\| v(\tau) \|_{\B^{\frac12+2\eps}_p} + 
			\| w(\tau) \|_{\B^{\frac12+2\eps}_p}
		\Big) 
	\d \tau . 
\end{multline*}
Similarly to the previous lemma, we use Theorem~\ref{t:apriori-v} to bound
\begin{align*}
\| v(\tau)\|_{\B^{\frac12+2\eps}_p} &\lesssim \frac{1}{\tau^{\frac14+\frac{5\eps}{2}}} \|v_0 \|_{\B^{-3\eps}_p} +
K \int_0^\tau \frac{1}{(\tau-s)^{\frac34+3\eps}} \big(\| w(s) \|_{L^p} +K \big) \d s\\
\notag
& =: B_0(\tau) +B_1(\tau),
\end{align*}
as well as
\begin{align*}
\Big( \int_0^t B_1(\tau)^{p} \d \tau \Big)^{\frac{1}{p}} &\lesssim Kt^{\frac14-{3\eps}}
\Big[ \Big( \int_0^t \| w(s) \|_{L^p}^{p}  \d s \Big)^{\frac{1}{p}}+K \Big].
\end{align*}
We then get
\begin{align*}
\int_0^t  
	\frac{1}{(t-\tau)^{\frac{2\gamma+1+2\eps}{4}  }}  
	B_0(\tau) 
\d \tau 
\lesssim 
\| v_0 \|_{\B^{-3\eps}_p} 
t^{\frac12(1-\gamma)-3\eps },
\end{align*}
and using H\"older's inequality for $\frac{1}{p'}+\frac{1}{p}=1$ yields
%
\begin{align*}
&\int_0^t  
	\frac{1}{(t-\tau)^{\frac{2\gamma+1+4\eps}{4}  }} 
	\Big(
		1+B_1(\tau) + 
		\| w(\tau) \|_{\B^{\frac12+2\eps}_p}
	\Big) 
\d \tau \\
 &\lesssim 
 \Big(
 	\int_0^t  
		 \Big(
			\frac{1}{(t-\tau)^{\frac{2\gamma+1+4\eps}{4}  }}
		\Big)^{p' } 
	\d \tau 
\Big)^{\frac{1}{p'}} 
\Big( 
	1 + 
	\int_0^t 
		\| w(\tau)\|_{\B^{\frac12+2\eps}_p}^p
	\d \tau
\Big)^{\frac{1}{p}}\\
&\lesssim 
K^2 t^{\frac34 - \frac{\gamma}{2} -\frac{1}{p} - \eps }
\Big( 
	1 + 
	\Big(
		\int_0^t 
			\| w(\tau)\|_{\B^{\frac12+2\eps}_p}^p
		\d \tau	
	\Big)^{\frac{1}{p}}
\Big),
\end{align*}
so that the desired bound follows from the embedding $\| \cdot \|_{\B^{\frac12+2\eps}_p} \lesssim \| \cdot \|_{\B^{1+4\eps}_p}$.
\end{proof}

\begin{proof}[Proof of Theorem~\ref{t:Gronwall-w}] The result is a straightforward consequence of the decomposition in \eqref{dest-4} (with $s = 0$) and of the results of Lemmas~\ref{le:W1} to \ref{le:W5}.
\end{proof} 
%

%
%
%
%
%
%
\section{Leveraging on the cubic non-linearity}
\label{s:testing-w}
In this section, we test the equation for $w$ against suitable powers of $w$. This allows us to benefit from the ``good'' sign of the term $-w^3$ in the definition of $G$. In the course of the argument, we will use Section~\ref{s.apriori-v} to dispense with the terms involving $v$, and effectively reduce the analysis of the system~\eqref{e:eqvwc} to that of a single equation on $w$; and Section~\ref{s:apriori-dw} to control the time regularity of $w$ and handle the commutator term~$\com_1$. We postpone the incorporation of the results of Section~\ref{s:Gronwall-w} to the next section. Recall that the relationship between $\uc$ and $c$ is fixed by \eqref{e.redef.uc}.

\begin{thm}[A priori estimate on $w$]
\label{t:apriori-w}
Let $p \ge 24$ and $\eps>0$ be sufficiently small. There exist $c_0, \kappa < \infty$ depending only on $p$ such that if 
\begin{equation}
\label{e.cond.c}
\uc \ge c_0 K^{30p} - (8K)^8, \quad \text{that is}, \quad c \ge c_0 K^{30p},
\end{equation}
then for every $t \in [0,T)$, we have
\begin{multline*}  
 \|w(t)\|_{L^{3p-2}}^{3p-2}  + \int_0^t  \|w(s)\|_{L^{3p}}^{3p} \, \d s \\
 \ls \|w_0\|_{L^{3p-2}}^{3p-2} +  (cK)^{\kappa} \Ll[ 1 + \| v_0  \|_{\B^{-3\eps}_{2p}}^{3p} + \int_0^t   \|w(s)\|_{\B^{1 + 6 \eps}_p}^p   \, \d s \Rr]  ,
\end{multline*}
where the implicit constant depends only on $p$ and $\eps$.
\end{thm}

In order to isolate the ``good term'' $-w^3$, we let $\td G$ be such that
\begin{equation}
\label{e.decomp.G}
G(v,w)  = -w^3 + \td G(v,w).
\end{equation}
\begin{prop}[Testing against $w^{3p-3}$]
\label{p:testing}
Let $p \ge 24$, which we recall is assumed to be an even integer, see \eqref{e.p.integer}. For every $t \in [0,T)$, we have
\begin{multline}
\label{e:testing}
\frac 1 {3p-2} \Ll( \|w(t)\|_{L^{3p-2}}^{3p-2} - \|w_0\|_{L^{3p-2}}^{3p-2} \Rr) 
+ (3p-3) \int_0^t  \| |\nabla w|^2 w^{3p-4}(s)\|_{L^1} \, \d s \\
  + \int_0^t \|w(s)\|_{L^{3p}}^{3p} \, \d s  = \int_0^t \la \td G(v,w) + cv, w^{3p-3} \ra (s) \, \d s.
\end{multline}
\end{prop}
\begin{proof}
By classical arguments (see e.g.\ \cite[Proposition~6.7]{JCH}), $w$ is a weak solution of \eqref{e:eqvwc}, in the sense that for every $\phi \in C^\infty_\per$, 
$$
\la w(t), \phi \ra - \la w(s),\phi \ra = \int_s^t [ -\la \nabla w(u), \nabla \phi \ra + \la [G(v,w) + c v](u) , \phi \ra ] \, \d u.
$$
We proceed as in the proof of \cite[Proposition~6.8]{JCH}. We split the interval $[0,t]$ into a subdivision $0 = t_0 \le \cdots \le t_n = t$, apply the identity above with $s = t_i$, $t = t_{i+1}$ and $\phi = w^{3p-3}(t_i)$, take the sum over $i$, and study the convergence of the result as the subdivision gets finer and finer. In order to obtain the result, we need to show that in this limit,
\begin{equation}
\label{e.test.w3p-2}
\sum_{i = 0}^{n-1}  \la w(t_{i+1}) - w(t_i), w^{3p-3}(t_i) \ra \longrightarrow \frac{1}{3p-2} \Ll( \|w(t)\|_{L^{3p-2}}^{3p-2} - \|w_0\|_{L^{3p-2}}^{3p-2} \Rr),
\end{equation}
\begin{equation}
\label{e.test.nablaw}
\sum_{i = 0}^{n-1} \int_{t_i}^{t_{i+1}} \la \nabla w(s), w^{3p-4}\nabla w(t_i) \ra \, \d s \longrightarrow \int_0^t  \| |\nabla w|^2 w^{3p-4}(s)\|_{L^1} \, \d s,
\end{equation}
and
\begin{equation}
\label{e.test.w3}
\sum_{i = 0}^{n-1} \int_{t_i}^{t_{i+1}} \la [G(v,w)+cv](s), w(t_i) \ra \, \d s \longrightarrow \int_0^t \la G(v,w) + cv, w^{3p-3} \ra (s) \, \d s.
\end{equation}
Indeed, the conclusion is then immediate from the decomposition of $G$ in \eqref{e.decomp.G}. We decompose the sum on the right-hand side of \eqref{e.test.w3p-2} into
\begin{multline*}  
\sum_{i = 0}^{n-1} \Ll( \|w(t_{i+1})^{3p-3}\|_{L^1} - \|w(t_i)^{3p-3}\|_{L^1} \Rr) \\
=  \sum_{i = 0}^{n-1} \la w(t_{i+1}) - w(t_i), w^{3p-3}(t_{i+1}) + w^{3p-4}(t_{i+1}) w(t_i) + \cdots + w^{3p-3}(t_i) \ra.
\end{multline*}
Each of the terms in the sum on the right side above is treated similarly. For notational simplicity, we only discuss the term $w^{3p-3}(t_{i+1})$. The difference between its contribution and the left-hand side of \eqref{e.test.w3p-2} is
\begin{equation*}  
\sum_{i = 0}^{n-1} \la w(t_{i+1}) - w(t_i), (w(t_{i+1}) - w(t_i))(w^{3p-4}(t_{i+1}) + \cdots + w^{3p-4}(t_i)) \ra. 
\end{equation*}
This difference tends to $0$ as the subdivision gets finer and finer, since by \eqref{e.smooth.consequence}, we have $w \in C^{\frac 1 2 + \eps} \Ll( [0,T),L^\infty \Rr)$. This completes the proof of \eqref{e.test.w3p-2}. The convergence in \eqref{e.test.nablaw} is a consequence of the fact that, by Theorem~\ref{thm:local-theory} and Proposition~\ref{p:derivatives}, 
\begin{equation*}  
w \in C([0,T), L^\infty) \quad \text{and} \quad \nabla w \in C([0,T),L^\infty).
\end{equation*}
Finally, we obtain the convergence in \eqref{e.test.w3} using Lemma~\ref{lem:loc-theory-lemma1} and the fact that $w \in C([0,T),\B^{1}_\infty)$.
\end{proof}

Similarly to \eqref{dest-4}, we now rewrite the right-hand side of \eqref{e:testing} as
\begin{align}
\notag
\int_0^t \la \td G(v,w) + cv, w^{3p-3} \ra (s) \, \d s 
\notag
& = -\int_0^t \la \,  3w^2 v + 3w v^2 + v^3 \,,  w^{3p-3} \ra (s)\; \d s \\
\notag
& \qquad -3 \int_0^t \la \, \com_1(v,w) \pe \<2> , \; w^{3p-3} \ra (s)\, \d s \\
\notag
& \qquad -3 \int_0^t \la \, w \pe \<2> , \; w^{3p-3} \ra (s)\, \d s \\
\notag
&\qquad + \int_0^t \la \, a_2(v+w)^2, \; w^{3p-3} \ra (s)\, \d s \\
\label{nlest-5}
& \qquad + \int_0^t \la\,  \ldots\,, \; w^{3p-3} \ra (s)\, \d s.
\end{align}
We now proceed to estimate each of these terms. The first term has a cubic homogeneity. We need to control it with the contribution of the ``good term'' $-w^3$. This crucially relies on our ability to choose $c$ sufficiently large.
\begin{lem} 
\label{lem:cubic2}
Let $p \ge 24$ and $\eps>0$ be sufficiently small. There exists a constant $c_1 < \infty$ depending only on $p$ such that for every $\delta \in (0,1]$, if
\begin{equation*}  
\uc \ge c_1 \de^{-(5+15p)}  K^{30p},
\end{equation*}
then for every $t \in [0,T)$, we have
\begin{align*}\label{nl1-0}
\int_0^t \la  v^3 +3v^2 w + 3 v w^2, w^{3p-3} \ra(s) \, \d s \le \delta \Ll[  1 +\|v_0\|_{\B^{-3\eps}_{2p}}^{3p} + \int_0^t \|w(s) \|_{L^{3p}}^{3p} \, \d s \Rr].
\end{align*}
 \end{lem}

\begin{proof}
We start with the bound
\begin{align}
\notag
&\int_0^t  \la  v^3 +3v^2 w + 3 v w^2, w^{3p-3} \ra(s) \, \d s  \\
	\qquad&\le \int_0^t \Ll( \| v^3 w^{3p-3}\|_{L^1}  + 3\|v^2 w^{3p-2}\|_{L^1} + 3\|v w^{3p-1}\|_{L^1}  \Rr)(s) \, \d s \\
\label{e.cubic21}
	\qquad&\leq \delta \int_0^t \| w(s)\|_{L^{3p}}^{3p}\, \d s  + \frac{7}{\delta^{3p}} \int_0^t \| v(s) \|_{L^{3p}}^{3p} \; \d s \;,
\end{align}
which follows from H\"older's and Young's inequalities. It is therefore sufficient to bound the space-time $L^{3p}$-norm 
of $v$. By Theorem~\ref{t:apriori-v} (or more precisely Remark~\ref{r:apriori-v1-Lp}),
 we have 
\begin{equation}
\label{e:control-hom}
\|v(s)\|_{L^{3p}}\lesssim 
\frac{ e^{-\uc s}}{s^{2\eps + \frac 1 {4p}}} \|v_0\|_{\B^{-3\eps}_{2p}} + 
K \int_0^s \frac{e^{-\uc (s-u)}}{(s-u)^{\frac 1 2 + \eps}} \Ll(K + \|w(u)\|_{L^{3p}} \Rr) \, \d u .
\end{equation}
By Jensen's inequality, we have uniformly over $\un c \ge 1$ and $s \geq 0$,
\begin{equation*}
\Ll(\int_0^s \frac{e^{-\uc (s-u)}}{(s-u)^{\frac 1 2 + \eps}} \Ll(K + \|w(u)\|_{L^{3p}} \Rr) \, \d u \Rr)^{3p} \lesssim \int_0^s \frac{e^{-\uc (s-u)}}{(s-u)^{\frac 1 2 + \eps}} \Ll(K^{3p} + \|w(u)\|^{3p}_{L^{3p}} \Rr) \, \d u .
\end{equation*}
We deduce that
\begin{align}
\notag
\lefteqn{
\int_0^t \| v(s) \|_{L^{3p}}^{3p} \;  \d s }
\qquad & \\
\notag
 	&\ls \int_0^t  \frac{ e^{-3p\uc s}}{s^{\frac 3 4 + 6p\eps}} \|v_0\|_{\B^{-3\eps}_{2p}}^{3p} \, \d s 
 + K^{3p}  \int_0^t \int_0^s  \frac{e^{-\uc (s-u)}}{(s-u)^{\frac 1 2 + \eps}} \Ll(K^{3p} + \|w(u)\|^{3p}_{L^{3p}}  \Rr)\, \d u \, \d s \\
 \label{e.cubic22}
 	&\ls I(\uc) \Ll[  \|v_0\|_{\B^{-3\eps}_{2p}}^{3p}   +
	 K^{6p} + K^{3p} \int_0^t  \|w(s)\|^{3p}_{L^{3p}} \, \d  s\Rr] \;,
\end{align}
where $I(\uc) =\int_0^\infty \frac{e^{-3p\uc s}}{s^{\frac 3 4 + 6p\eps}} \d s \vee \int_0^\infty \frac{e^{-\uc s}}{s^{\frac 1 2 + \eps}} \d s$. For $\eps > 0$ sufficiently small, this quantity is finite, and moreover,
\begin{equation*}  
\int_0^\infty \frac{e^{-3p\uc s}}{s^{\frac 3 4 + 6p\eps}} \d s = \uc^{- \Ll( \frac 1 4 - 6p\eps \Rr) } \int_0^\infty \frac{e^{-3ps}}{s^{\frac 3 4 + 6p\eps}} \d s ,
\end{equation*}
and
\begin{equation*}
\int_0^\infty \frac{e^{-\uc s}}{s^{\frac 1 2 + \eps}} \d s = \uc^{- \Ll( \frac 1 2 - \eps \Rr) } \int_0^\infty \frac{e^{- s}}{s^{\frac 1 2 + \eps}} \d s.
\end{equation*}
Fixing $\eps > 0$ sufficiently small in terms of $p$, we can therefore enforce that $I(\uc) \ls \uc^{- \frac 1 5}$. Combining this with \eqref{e.cubic21} and \eqref{e.cubic22} completes the proof.
\end{proof}

We now use the a priori estimate on $\de_{st} w$ derived in Section~\ref{s:apriori-dw} to estimate the contribution of the first commutator term.

\begin{lem}
\label{l.test.dst}
Let $p \ge 24$ and $\eps>0$ be sufficiently small. There exists an exponent $\kappa > 0$ depending only on $p$ such that for every $\delta\in (0,1]$ and $t \in [0,T)$, we have
\begin{align}
\label{ls-3-bound}
&  \Ll| \int_0^t \la \, \com_1(v,w) \pe \<2> , \; w^{3p-3} \ra(s)  \, \d s  \Rr|
\\
\notag
&  \qquad \ls  \Ll(\de^{-1} c K\Rr)^\kappa \Ll[ 1 + \| v_0  \|_{\B^{-3\eps}_{p}}^{3p} + \int_0^t   \|w(s)\|_{\B^{1 + 4 \eps}_p}^p   \, \d s \Rr] 
 \\
 \notag
& \qquad \qquad + \de \Ll[ \sup_{s \le t} \|w(s)\|_{L^{3p-2}}^{3p-2}  + \int_0^t \| w(s) \|^{3p}_{L^{3p }} \d s \Rr] ,
\end{align}
where the implicit constant 
depends only on $p$ and $\eps$. 
\end{lem}

\begin{proof}
We start by applying H\"older's inequality and then Proposition~\ref{p:mult} to get (dropping the time variable in the notation)
\begin{align}
\notag
\lefteqn{
| \la \, [\com_1(v,w) \pe \<2>] , \; w^{3p-3} \ra  |
} \qquad  & \\
\notag
& \leq 
\| e^{\cdot \Delta} v_0 \pe \<2> \|_{L^{3p-2}}   \| w\|_{L^{3p-2}}^{3p-3}   + \| (\com_1(v,w) - e^{\cdot \Delta} v_0) \pe \<2> \|_{L^p}   \| w^{3p-3} \|_{L^{\frac{p}{p-1}}}  \\
& \lesssim K
\| e^{\cdot \Delta} v_0  \|_{\B^{1+2\eps}_{3p-2}}  \| w\|_{L^{3p-2}}^{3p-3} + K \| \com_1(v,w) - e^{\cdot \Delta} v_0  \|_{\B_p^{1+2\eps}}\;   \| w \|_{L^{3p }}^{3p-3} \;.
\label{e.whereiss}
\end{align}
We integrate in time the first term, use Propositions~\ref{p:smooth-besov} and \ref{p:embed}, Jensen's and Young inequalities to get
\begin{align}  
\notag
\int_0^t \| e^{s \Delta} v_0  \|_{\B^{1+2\eps}_{3p-2}}  \| w(s)\|_{L^{3p-2}}^{3p-3} \, \d s & \ls \int_0^t s^{-\Ll(\frac{1+5\eps}{2} + \frac 3 2 \Ll( \frac 1 p - \frac 1 {3p-2} \Rr) \Rr)} \| v_0  \|_{\B^{-3\eps}_{p}}   \| w(s)\|_{L^{3p-2}}^{3p-3} \, \d s \\
\label{e.ahh.the.initial.condition}
& \ls  
\| v_0  \|_{\B^{-3\eps}_{p}} \, \sup_{s \le t} \|w(s)\|_{L^{3p-2}}^{3p-3} \\
\notag
& \ls\|v_0\|_{\B^{-3\eps}_{p}}^{3p} +  \, \sup_{s \le t} \|w(s)\|_{L^{3p-2}}^{\frac{3p(3p-3)}{3p-1} }.
\end{align}
Moreover, since 
\begin{equation}  
\label{e.stupid.3p.inequality}
\frac{3p(3p-3)}{3p-1} < 3p-2,
\end{equation}
a second application of Young's inequality yields that, for some exponent $\kappa > 0$ depending only on $p$ and every $\de \in (0,1]$,
\begin{equation*}  
K \sup_{s \le t} \|w(s)\|_{L^{3p-2}}^{\frac{3p(3p-3)}{3p-1} } \le \de \sup_{s \le t} \|w(s)\|_{L^{3p-2}}^{3p-2} + 
+ (\de^{-1} K)^\kappa.
\end{equation*}

Integrating the second term in \eqref{e.whereiss} and applying H\"older's inequality, we get
\begin{multline}
\label{e.test.dst.1}
\Ll| \int_0^t  \| \com_1(v,w)(s) - e^{s \Delta} v_0  \|_{\B_p^{1+2\eps}}\;   \| w(s) \|_{L^{3p }}^{3p-3} \; \d s\Rr|  \\
\le \Ll( \int_0^t \| \com_1(v,w)(s) - e^{s \Delta} v_0 \|_{\B_p^{1+2\eps}}^p  \d s \Rr)^{\frac{1}{p}}\; \Ll( \int_0^t \| w(s) \|^{3p}_{L^{3p }} \d s \Rr)^{\frac{p-1}{p}}.
\end{multline}
We now focus on bounding the first integral on the \rhs\ above.  
According to Lemma~\ref{l:estim-com1}, for any fixed $s$, we have the bound 
\begin{multline*}
 \|\msf{com}_1(v,w)(s) - e^{s\Delta} v_0\|_{\B_p^{1 +2\eps}} 
\ls   K^3 + K(K+c)s^{-{{4\eps}}} \, \|v_0\|_{\B^{-3\eps}_p} \\ 
+ \int_0^s \frac{K^2}{(s-u)^{\frac 3 4 +\eps}}  \|w(u)\|_{\B^{\frac 1 2 + 2 \eps}_p} \, \d u   + \int_0^s \frac{K}{(s-u)^{1+2\eps}} \|\de_{us} w\|_{L^p}  \, \d u.
\end{multline*}
The contribution of $\|v_0\|_{\B^{-3\eps}_p}$ is easily taken care of. We calculate the $L^p$ norm in time of the first integral, using Jensen's inequality and the bound $\|\, \cdot \, \|_{\B^{\frac 1 2 + 2\eps}_p} \lesssim \|\, \cdot \, \|_{\B^{1 + 4 \eps}_p}$:
\begin{align*}
\int_0^t \Ll(   \int_0^s \frac{1}{(s-u)^{\frac12 +3\eps}}  \|w(u)\|_{\B^{1 - 2 \eps}_p} \d u\Rr)^p \, \d s \lesssim \int_0^t   \|w(s)\|_{\B^{1 + 4 \eps}_p}^p   \, \d s .
\end{align*}
For the remaining integral, we first write for any $\delta>0$,
\begin{align*}
 \int_0^{(s-\delta)\vee 0} \frac{1}{(s-u)^{1+2\eps}} \|\de_{us} w\|_{L^p}  \, \d u \lesssim \frac{1}{\delta^{1+2 \eps}} \int_{0}^{s}   \Ll( \| w(u)\|_{L^p} + \| w(s)\|_{L^p} \Rr) \d u ,
\end{align*}
which implies that 
\begin{align*}
\int_0^t \Ll(   \int_0^{(s-\delta)\vee 0} \frac{1}{(s-u)^{1+2\eps}} \|\de_{us} w\|_{L^p}  \, \d u \Rr)^p \d s \lesssim \frac{1}{\delta^{(1+2 \eps)p}} \int_0^t  \| w(s)\|_{L^p}^p  \d s.
\end{align*}
We then use Theorem~\ref{t:apriori-dw} to get
\begin{multline*}
\int_0^t   \Ll(   \int_{{(s-\delta) \vee 0}}^s \frac{K}{(s-u)^{1+2\eps}} \|\de_{us} w\|_{L^p}  \, \d u \Rr)^p \d s \\
\lesssim   \int_0^t   \Ll(   \int_{(s-\delta) \vee 0}^s \frac{cK^8}{(s-u)^{\frac{7}{8}+2\eps}}  \Ll[ \td N(t) + \|w(u)\|_{\B^{1+4\eps}_p} \Rr]  \, \d u \Rr)^p \d s,
\end{multline*}
where we have set
\begin{equation}
\label{e:DefN}
\td N(t) := 1 + \|v_0\|_{\B^{-3\eps}_{3p}}^{3}   \\
+   \Ll( \int_0^t \| w(u)    \|_{\B_p^{1+4\eps}}^{p} \ \d u \Rr)^{\frac{1}{p}} 
+  \Ll(\int_0^t \|w(u)\|_{L^{3p}}^{3p} \, \d u \Rr)^{\frac{1}{p}}.
\end{equation}
Note that $\td N(t)$ does not depend on the variables of integration, and that
$$
 \int_0^t   \Ll(   \int_{(s-\delta) \vee 0}^s \frac{1}{(s-u)^{\frac{7}{8}+2\eps}}   \, \d u \Rr)^p \d s \ls \delta^{p \Ll( \frac 1 8 - 2\eps \Rr)  } t .
$$
Finally, by Jensen's inequality,
\begin{align*}
& \int_0^t   \Ll(   \int_{(s-\delta) \vee 0}^s \frac{1}{(s-u)^{\frac{7}{8}+2\eps}}  \|w(u)\|_{\B^{1+4\eps}_p} \, \d u \Rr)^p \d s \\
 & \quad 
\ls \int_0^t     \int_{(s-\delta) \vee 0}^s \frac{1}{(s-u)^{\frac{7}{8}+2\eps}}  \|w(u)\|^p_{\B^{1+4\eps}_p} \, \d u  \, \d s  \ls  \int_0^t \|w(u)\|^p_{\B^{1+4\eps}_p} \, \d u .
\end{align*}
Summarizing, we have bounded the left side of \eqref{e.test.dst.1} by
\begin{multline*}  
c K^{8} \Ll( \int_0^t \| w(s) \|^{3p}_{L^{3p }} \d s \Rr)^{\frac{p-1}{p}} \Bigg[1 + \|v_0\|_{\B^{-3\eps}_p}^3 + \Ll(\int_0^t   \|w(s)\|_{\B^{1 + 4 \eps}_p}^p   \, \d s  \Rr)^\frac 1 p \\
 + \frac 1 {\de^{1+2\eps}} \Ll( \int_0^t  \| w(s)\|_{L^p}^p  \, \d s\Rr)^{\frac 1 p}+  \de^{\frac 1 8 - 2\eps} \Ll(\int_0^t \|w(s)\|_{L^{3p}}^{3p} \, \d s \Rr)^{\frac{1}{p}}
\Bigg].
\end{multline*}
Applying Young's inequality on each term (save the last one) then yields \eqref{ls-3-bound}.
\end{proof}

We now turn to the term involving $w \pe \<2>$, which can only be controlled by a norm of $w$ with regularity index above $1$.

\begin{lem}
Let $p \ge 24$ and $\eps>0$ be sufficiently small. For every $t \in [0,T)$,
\begin{align*}
\Ll| \int_0^t  \la \, w \pe \<2> , \; w^{3p-3} \ra(s)  \; \d s\Rr| 
\ls K \Ll( \int_0^t \| w(s) \|_{L^{3p}}^{3p} \d s \Rr)^{\frac{p-1}{p}} 
\Ll(\int_0^t \| w(s)\|_{\B^{1+2\eps}_p}^p \d s \Rr)^{\frac{1}{p}},
\end{align*}
where the implicit constant depends only $p$ and $\eps$.
\end{lem}
\begin{proof}
This bound follows directly from H\"older's inequality and the bound 
\begin{equation*}
 \| w \pe \<2>(s) \|_{L^p } \lesssim  K \| w(s)\|_{\B^{1+2\eps}_p} \;. \qedhere
\end{equation*}
\end{proof}

The quadratic non-linearity is rather delicate to handle.

\begin{lem}
Let $p \ge 24$ and $\eps>0$ be sufficiently small. There exists an exponent $\kappa > 0$ depending only on $p$ such that for every $\delta \in (0,1]$ and $t \in [0,T)$, we have
\begin{multline}
\label{nl-2:Bound0}
\Ll| \int_0^t \la \, a_2(v+w)^2, \; w^{3p-3} \ra(s)  \, \d s \Rr| \ls (\de^{-1}K)^{\kappa}\Ll[1+ \int_0^t \|w(s)\|_{\B^{1+4\eps}_p}^p \, \d s\Rr]\\ + \delta \Ll[ \|v_0\|_{\B^{-3\eps}_{2p}}^{3p}  
+ \int_0^t   \|  w(s)\|_{L^{3p}}^{3p} \, \d s + \sup_{s \le t} \|w(s)\|_{L^{3p-2}}^{3p-2}  \Rr] ,
\end{multline}
where the implicit constant depends only $p$ and $\eps$.
\end{lem}
\begin{proof}
Throughout the proof, the exponent $\kappa > 0$ may vary from one occurence to another, but is only allowed to depend on $p$. 
We decompose the proof into three steps, treating the contributions of $w^2, v^2$ and $vw$ successively. 

\smallskip

\emph{Step 1.} We first treat the term of highest homogeneity in $w$. Recall that $a_2$ is uniformly bounded in $\B^{-\frac 1 2 - \eps}_\infty$. We write, using Propositions~\ref{p:dual} (dropping the time 
variable in the notation),
\begin{align}
\notag
\la  a_2 w^2, \; w^{3p-3} \ra 
	&= \la a_2 , w^{3p-1} \ra \lesssim K \| w^{3p-1}\|_{\B_1^{\frac12+2\eps}}.
\end{align}
Moreover, by Corollary~\ref{c:mult},
\begin{equation*}  
\|w^{3p-1}\|_{\B_1^{\frac12+2\eps}} \ls \|w^{3p-2}\|_{L^{\frac{3p}{3p-2}}} \, \|w\|_{\B^{\frac 1 2 + 2\eps}_{\frac {3p}{2}}}.
\end{equation*}
By Proposition~\ref{p:interpol} and Remark~\ref{r:Besov-vs-Lp}, 
\begin{equation}  
\label{e.w.3p2}
\|w\|_{\B^{\frac 1 2 + 2\eps}_{\frac {3p}{2}}} \ls \|w\|_{\B^{1 + 4\eps}_p}^\frac 1 2 \, \|w\|_{L^{3p}}^\frac 1 2 ,
\end{equation}
Moreover, by Young's inequality,
\begin{equation*}  
\|w\|_{L^{3p}}^{3p-2 + \frac 1 2 } \, \|w\|_{\B^{1 + 4\eps}_{p}}^\frac 1 2 \le  \delta \|w\|_{L^{3p}}^{3p} + \delta^{-\kappa} \|w\|_{\B^{1 + 4\eps}_{p}}^{p},
\end{equation*}
so that integrating over time completes the estimate of this term.

\smallskip

\emph{Step 2.}
We now turn to the contribution of the term involving $v^2$. As above, our starting point is the observation that
\begin{equation*}  
\la a_2 v^2, w^{3p-3} \ra \ls K \|v^2 w^{3p-3}\|_{\B^{\frac 1 2 + 2\eps}_1},
\end{equation*}
and by Proposition~\ref{p:mult}, 
\begin{equation} 
\label{e.bound.v2w3p}
\|v^2 w^{3p-3}\|_{\B^{\frac 1 2 + 2\eps}_1} \ls \|v^2\|_{\B^{\frac 1 2 + 2\eps}_{\infty}} \, \|w^{3p-3}\|_{L^{\frac{3p-2}{3p-3}}} + \|v^2\|_{L^{\frac{3p}{2}}} \, \|w^{3p-3}\|_{\B^{\frac 1 2 +2 \eps}_{\frac{3p}{3p-2}}}.
\end{equation}
For the first term on the right side above, we expect $\|v^2\|_{\B^{\frac 1 2+2\eps}_\infty}$ to have almost $L^4$ integrability in time. We may choose to bound $\|w\|_{L^{3p-2}}$ by $\|w\|_{L^{3p}}$; such a bound is interesting since the term involving $\|w\|_{L^{3p}}$ on the right side of \eqref{nl-2:Bound0} appears with a different homogeneity than the term involving $\|w\|_{L^{3p-2}}$. However, the term involving $\|w\|_{L^{3p}}$ only appears integrated in time, which is problematic for controlling the small-time divergence of $\|v^2\|_{\B^{\frac 1 2+2\eps}_\infty}$. We will therefore use an interpolation of these bounds, such as
\begin{equation}  
\label{e.interpol.choice}
\|w\|_{L^{3p-2}}^{3p-3} \ls \|w\|_{L^{3p-2}}^{\frac{3p-6}{2}} \, \|w\|_{L^{3p}}^{\frac {3p}{2}}.
\end{equation}
This choice of exponents yields, by H\"older's and Young's inequalities,
\begin{align}  
\notag
\lefteqn{
\int_0^t \|v^2(s)\|_{\B^{\frac 1 2 + 2\eps}_{\infty}} \, \|w^{3p-3}(s)\|_{L^{\frac{3p-2}{3p-3}}} \, \d s 
} \qquad & \\
\notag
&  \ls 
\Ll(\int_0^t \|v^2(s)\|_{\B^{\frac 1 2 + 2\eps}_{\infty}}^{2} \, \d s\Rr)^{\frac 1 2}
\Ll(\sup_{s \le t} \|w(s)\|_{L^{3p-2}}^{\frac{3p-6}{2}}\Rr)
\Ll(\int_0^t  \, \|w(s)\|_{L^{3p}}^{3p} \, \d s \Rr)^{\frac 1 2} \\
\notag
&  \ls 
\de^{-\kappa} + \de \Bigg[\Ll(\int_0^t \|v^2(s)\|_{\B^{\frac 1 2 + 2\eps}_{\infty}}^{2} \, \d s\Rr)^{\frac{3p}{4}} + 
\Ll(\sup_{s \le t} \|w(s)\|_{L^{3p-2}}^{3p-2}\Rr)\\
\label{e.three.v2}
& \qquad \qquad \qquad + 
\int_0^t  \, \|w(s)\|_{L^{3p}}^{3p} \, \d s\Bigg] ,
\end{align}
since
\begin{equation*}  
\frac 2 {3p} + \frac{3p-6}{6p-4} + \frac 1 2 = 1+\frac 2 {3p} - \frac{4}{6p-4}  < 1.
\end{equation*}
There remains to bound the first integral on the right side of \eqref{e.three.v2}, that is,
\begin{equation}  
\label{e.v22}
\Ll(\int_0^t  \|v^2(s)\|_{\B^{\frac 1 2 + 2\eps}_{\infty}}^{2} \, \d s\Rr)^{\frac {3p}{4}}.
\end{equation}
By Proposition~\ref{p:mult}, we have
\begin{equation*}  
\|v^2\|_{\B^{\frac 1 2+2\eps}_\infty} \ls \|v\|_{L^\infty} \, \|v\|_{\B^{\frac 1 2+2\eps}_\infty},
\end{equation*}
and Theorem~\ref{t:apriori-v} and Remark~\ref{r:apriori-v1-Lp} allow to bound each of these two terms, that is,
\begin{multline}  
\label{e.heres.some.v}
\|v(s)\|_{\B^{\frac 1 2 + 2\eps}_{\infty}} \ls s^{- \Ll( \frac 1 4 + \frac{7\eps}{2} + \frac 3 {2p}  \Rr) }\|v_0\|_{\B^{-3\eps}_{p}}  
\\
+K \int_0^s (s-u)^{- \Ll(\frac{3}{4} + \frac{3\eps}{2} + \frac 1 {2p}  \Rr)} (K + \| w(u) \|_{L^{3p}}) \, \d u,
\end{multline}
\begin{multline}  
\label{e.heres.some.v.L}
\|v(s)\|_{L^{\infty}} \ls s^{- \Ll( 2\eps + \frac 3 {2p}  \Rr) }\|v_0\|_{\B^{-3\eps}_{p}}  
\\
+K \int_0^s (s-u)^{- \Ll(\frac{1}{2} + \eps + \frac 1 {2p}  \Rr)} (K + \| w(u) \|_{L^{3p}}) \, \d u.
\end{multline}
The cross-term involving $\|v_0\|_{\B^{-3\eps}_{p}}$ only contributes
\begin{equation*}  
\int_0^t s^{- 2 \Ll( \frac 1 4 + \frac{11 \eps}{2} + \frac 3 p \Rr) } \|v_0\|_{\B^{-3\eps}_{3p}}^4 \, \d s \ls \|v_0\|_{\B^{-3\eps}_{3p}}^4,
\end{equation*}
provided that $p > 12$ and $\eps > 0$ is sufficiently small. The contribution of the cross-term involving the integrals in \eqref{e.heres.some.v}--\eqref{e.heres.some.v.L} can be bounded by
\begin{align*}  
\lefteqn{
K^4 \int_0^t \Ll(\int_0^s (s-u)^{- \Ll(\frac{3}{4} + \frac{3\eps}{2} + \frac 1 {2p}  \Rr)} (K + \| w(u) \|_{L^{3p}}) \, \d u  \Rr)^4 \, \d s
} \qquad & \\
& \ls K^4 \int_0^t \int_0^s (s-u)^{- \Ll(\frac{3}{4} + \frac{3\eps}{2} + \frac 1 {2p}  \Rr)} (K + \| w(u) \|_{L^{3p}})^4 \, \d u\, \d s \\
& \ls K^4 \int_0^t (K + \| w(u) \|_{L^{3p}})^4 \, \d u,
\end{align*}
by Jensen's inequality and Fubini's theorem. By H\"older's inequality, the two mixed terms involving $\|v_0\|_{\B^{-3\eps}_{p}}$ and an integral from \eqref{e.heres.some.v}--\eqref{e.heres.some.v.L} are both bounded by
\begin{align*}  
\lefteqn{
K^2 \|v_0\|_{\B^{-3\eps}_{p}}^2 \int_0^t  \Ll(s^{- \Ll( \frac 1 4 + \frac{7\eps}{2} + \frac 3 {2p}  \Rr) } \int_0^s (s-u)^{- \Ll(\frac{3}{4} + \frac{3\eps}{2} + \frac 1 {2p}  \Rr)} (K + \| w(u) \|_{L^{3p}}) \, \d u \Rr)^2 \, \d s   
} \qquad & \\
& \ls K^2 \|v_0\|_{\B^{-3\eps}_{p}}^2 \Ll(\int_0^t  \Ll( \int_0^s (s-u)^{- \Ll(\frac{3}{4} + \frac{3\eps}{2} + \frac 1 {2p}  \Rr)} (K + \| w(u) \|_{L^{3p}})\, \d u \Rr)^6 \, \d s  \Rr)^{\frac 1 3} \\
& \ls K^2 \|v_0\|_{\B^{-3\eps}_{p}}^2 \Ll(\int_0^t  (K + \| w(u) \|_{L^{3p}})^6\, \d u   \Rr)^{\frac 1 3} ,
\end{align*}
provided that $p > 18$ and $\eps > 0$ is sufficiently small. This completes the analysis of the first term on the right side of \eqref{e.bound.v2w3p}. For the second term there, we use the bound from Corollary~\ref{c:mult} and \eqref{e.w.3p2} to get
\begin{equation}  
\label{e.almost.same.below}
\|w^{3p-3}\|_{\B^{\frac 1 2 + 2\eps}_{\frac{3p}{3p-2}}} \ls \|w^{3p-4}\|_{L^{\frac{3p}{3p-4}}} \, \|w\|_{\B^{\frac 1 2 + 2\eps}_{\frac {3p}{2}}} \ls \|w\|_{L^{3p}}^{3p-4+ \frac 1 2} \, \|w\|_{\B^{1 + 4\eps}_{p}}^\frac 1 2,
\end{equation}
since $3p-4 \ge 1$.
Appealing again to Theorem~\ref{t:apriori-v}, we deduce that
\begin{align}  
\notag
\lefteqn{
\int_0^t \|v^2(s)\|_{L^{\frac{3p}{2}}} \, \|w^{3p-3}(s)\|_{\B^{\frac 1 2 + 2\eps}_{\frac{3p}{3p-2}}} \, \d s }
\qquad & \\
\label{e.v.square2}
& \ls \int_0^t \Ll[ s^{-\Ll(2\eps + \frac 1 {4p}\Rr)} \|v_0\|_{\B^{-3\eps}_{2p}} + K \int_0^s (s-u)^{-\Ll(\frac 1 2 + \eps\Rr)} (K + \|w(u)\|_{L^{3p}}) \, \d u \Rr]^2 \\
\notag
& \qquad \qquad \times  \|w(s)\|_{L^{3p}}^{3p-\frac 7 2} \, \|w(s)\|_{\B^{1 + 4\eps}_{p}}^\frac 1 2 \, \d s.
\end{align}
We use Young's inequality on the squared term above, and then bound 
\begin{multline*}  
\|v_0\|_{\B^{-3\eps}_{2p}}^2 \int_0^t s^{-\Ll(4\eps + \frac 1 {2p}\Rr)} \|w(s)\|_{L^{3p}}^{3p-\frac 7 2} \, \|w(s)\|_{\B^{1 + 4\eps}_{p}}^\frac 1 2 \, \d s 
\\
 \ls \|v_0\|_{\B^{-3\eps}_{2p}}^2 \Ll(\int_0^t \|w(s)\|_{L^{3p}}^{3p} \, \d s\Rr)^{\frac {6p-7}{6p}} \Ll( \int_0^t \|w(s)\|_{\B^{1 + 4\eps}_p}^p \d s\Rr)^{\frac 1 {2p}},
\end{multline*}
provided that $\eps > 0$ is sufficiently small. Noting that
\begin{equation}  
\label{e.magic.magic}
\frac 1 {3p} \Ll( 2 + \frac{6p-7}{2} + \frac 3 2 \Rr)  = 1,
\end{equation}
and applying Young's inequality with these exponents, we conclude that
the quantity above is bounded by
\begin{equation*}  
\|v_0\|_{\B^{-3\eps}_{2p}}^{3p} +  \int_0^t \|w(s)\|_{L^{3p}}^{3p} \, \d s +  \int_0^t \|w(s)\|_{\B^{1 + 4\eps}_p}^p \d s.
\end{equation*}
We now bound the remaining term from \eqref{e.v.square2}, namely
\begin{equation*}  
\int_0^t \Ll[K\int_0^s (s-u)^{-\Ll(\frac 1 2 + \eps\Rr)} (K + \|w(u)\|_{L^{3p}}) \, \d u\Rr]^2  \|w(s)\|_{L^{3p}}^{3p-\frac 7 2} \, \|w(s)\|_{\B^{1 + 4\eps}_{p}}^\frac 1 2 \,  \d s.
\end{equation*}
We use once more the identity \eqref{e.magic.magic} to apply Young's inequality and get that the integral above is bounded by
\begin{equation*}  
\int_0^t \Ll(\Ll[K\int_0^s (s-u)^{-\Ll(\frac 1 2 + \eps\Rr)} (K + \|w(u)\|_{L^{3p}}) \, \d u\Rr]^{3p}  +  \|w(s)\|_{L^{3p}}^{3p} +  \|w(s)\|_{\B^{1 + 4\eps}_{p}}^p \Rr) \,  \d s.
\end{equation*}
The contribution of the inner integral is bounded using Jensen's inequality. This therefore completes the analysis of the contribution of the term $a_2 v^2$. 

\smallskip

\emph{Step 3.} We finally analyse the contribution of the cross-term $vw$. As in the previous steps, our starting point is the inequality
\begin{equation*}  
\la a_2 vw, w^{3p-3} \ra \ls K \|v w^{3p-2}\|_{\B^{\frac 1 2 +2 \eps}_1}.
\end{equation*}
As above, we apply Proposition~\ref{p:mult} to note that 
\begin{equation}  
\label{e.decomp.vw}
\|v  w^{3p-2}\|_{\B^{\frac 1 2 + 2\eps}_1} \ls \|v\|_{\B^{\frac 1 2 +2 \eps}_{\infty}} \, \|w^{3p-2}\|_{L^1} + \|v\|_{L^{3p}} \, \|w^{3p-2}\|_{\B^{\frac 1 2 + 2\eps}_{\frac {3p}{3p-1}}}.
\end{equation}
Similarly to \eqref{e.interpol.choice}, we use the upper bound
\begin{equation*}  
\|w\|_{L^{3p-2}}^{3p-2} \ls \|w\|_{L^{3p-2}}^\frac {3p-4}{2} \|w\|_{L^{3p}}^{\frac {3p}{2}}
\end{equation*}
to gain some integrability in time. That is, we apply H\"older's inequality to get
\begin{align}  
\notag
\lefteqn{
\int_0^t \|v(s)\|_{\B^{\frac 1 2 + 2\eps}_{\infty}} \, \|w^{3p-2}(s)\|_{L^1} \, \d s
} \qquad & \\
\notag
& \ls 
\Ll(\int_0^t \|v(s)\|_{\B^{\frac 1 2 +2 \eps}_{\infty}}^{2} \, \d s\Rr)^\frac 1 2  \Ll(\sup_{s \le t}  \|w(s)\|_{L^{3p-2}}^\frac {3p-4}{2}\Rr) \Ll(\int_0^t \|w(s)\|_{L^{3p}}^{3p} \, \d s\Rr)^{\frac 1 2 } \\
\label{e.decomp.vw1}
& \ls
\de^{-\kappa} + \de \Ll[
\Ll(\int_0^t \|v(s)\|_{\B^{\frac 1 2 +2 \eps}_{\infty}}^{2} \, \d s\Rr)^\frac{3p}{2} +   \Ll(\sup_{s \le t}  \|w(s)\|_{L^{3p-2}}^{3p-2}\Rr) +  \int_0^t \|w(s)\|_{L^{3p}}^{3p} \, \d s\Rr] ,
\end{align}
where in the last step,  we used Young's inequalities with exponents
\begin{equation*}  
\frac 1 {3p} + \frac{3p-4}{6p-4} + \frac 1 2  < \frac 1 {3p} + \frac 1 2 + \frac{3p-2}{6p} = 1.
\end{equation*}
The first integral on the right side of \eqref{e.decomp.vw1} is very similar to that appearing in \eqref{e.v22}, and can be treated similarly. There remains to estimate the contribution of the last term on the right side of \eqref{e.decomp.vw}. By Corollary~\ref{c:mult} and \eqref{e.w.3p2}, we have
\begin{equation*}  
 \|w^{3p-2}\|_{\B^{\frac 1 2 + 2\eps}_{\frac {3p}{3p-1}}} \ls \|w^{3p-3}\|_{L^{\frac{p}{p-1}}} \, \|w\|_{\B^{\frac 1 2 +2 \eps}_{\frac{3p}{2}}}
\ls \|w\|_{L^{3p}}^{3p-3+\frac 1 2} \, \|w\|_{\B^{1 + 4\eps}_{p}}^\frac 1 2.
\end{equation*}
The analysis of 
\begin{equation*}  
\int_0^t \|v(s)\|_{L^{3p}} \, \|w(s)\|_{L^{3p}}^{3p-3+\frac 1 2} \, \|w(s)\|_{\B^{1 + 4\eps}_{p}}^\frac 1 2 \, \d s
\end{equation*}
then proceeds along very similar lines to that for \eqref{e.v.square2} above, and we therefore omit the details.
\end{proof}


\begin{lem}\label{le:dots2}
Let $p \ge 24$ and $\eps>0$ be sufficiently small. There exists an exponent $\kappa > 0$ depending only on $p$ such that for every $\de \in (0,1]$ and $t \in [0,T)$, we have
\begin{multline*}
 \int_0^t \la  [\ \ldots\ ], w^{3p-3} \ra (s) \, \d s \le  (\de^{-1}K)^{\kappa}\Ll[1+\|v_0\|_{\B^{-3\eps}_{2p}}^{3p} +  \int_0^t \|w(s)\|_{\B^{1+6\eps}_p}^p \, \d s\Rr]\\ + \delta \Ll[  \int_0^t   \|  w(s)\|_{L^{3p}}^{3p} \, \d s + \sup_{s \le t} \|w(s)\|_{L^{3p-2}}^{3p-2}  \Rr] .
\end{multline*}
The dots $\ldots$ represent all the terms left out in \eqref{nlest-5} (spelled out explicitly in \eqref{dot-terms-2} below). 
\end{lem}
\begin{proof}
We need to bound 
\begin{align}
\label{dot-terms-2}
 \int_0^t \la  \Ll[ -3 \msf{com}_2(v+w)   - 3(v+w-\<30>) \pg \<2> + a_0 +a_1(v+w) + c v\Rr], w^{3p-3} \ra (s) \, \d s.
\end{align}
For the first term,
\begin{align*}
 \la \msf{com}_2(v+w)  , w^{3p-3} \ra(s)  & \ls \| \msf{com}_2(v+w)(s) \|_{L^p}
\, \| w(s) \|_{L^{3p}}^{{3p-3}} \\
		& \ls  K^2 \Ll(\| (v+w)(s) \|_{\B^{3\eps}_p} + K\Rr) \, \| w(s) \|_{L^{3p}}^{3p-3},
\end{align*} 
by Proposition~\ref{p:comm1}. We then apply Young's inequality to bound
\begin{equation*}  
 K^2 \| (v+w)(s) \|_{\B^{3\eps}_p} \|w(s)\|_{L^{3p}}^{3p-3}  \le \de \|w(s)\|_{L^{3p}}^{3p} + \Ll( K^2 \de^{-1} \Rr)^p \| (v+w)(s) \|_{\B^{3\eps}_p}^p, 
\end{equation*}
and proceed as before to bound the last term, appealing to Theorem~\ref{t:apriori-v}. For the second term in \eqref{dot-terms-2}, we treat the initial condition for $v$ separately, that is, we first bound
\begin{align*}  
\int_0^t \la (e^{s(\De-c)}  v_0) \pg \<2>, w^{3p-3}(s)  \ra \, \d s 
& 
\le \int_0^t \|(e^{s(\De-c)}  v_0) \pg \<2> \|_{L^{3p-2}} \, \|w(s)\|_{L^{3p-2}}^{3p-3}  \, \d s 
\\
& \ls K \int_0^t \|e^{s(\De-c)}  v_0 \|_{\B^{1+2\eps}_{3p-2}} \, \|w(s)\|_{L^{3p-2}}^{3p-3}  \, \d s.
\end{align*}
We have already estimated this term, see \eqref{e.ahh.the.initial.condition}. We now focus on bounding (dropping the time variable in the notation)
\begin{align*}
& \la   (v - e^{\cdot (\De-c)} v_0 +w-\<30>) \pg \<2> , w^{3p-3} \ra \\
& \qquad 
\ls \|(v- e^{\cdot (\De-c)} v_0 +w-\<30>) \pg \<2>\|_{\B^{-\frac 1 2 - 2\eps}_{\frac {3p}{2}}} \, \|w^{3p-3}\|_{\B^{\frac 1 2 + 3\eps}_{\frac{3p}{p-2}}} \\
& \qquad 
\ls K \Ll(\|(v- e^{\cdot (\De-c)} v_0 +w)\|_{\B^{\frac 1 2 - \eps}_{\frac {3p}{2}}} + K\Rr) \, \|w^{3p-3}\|_{\B^{\frac 1 2 + 3\eps}_{\frac {3p} {3p-2}}}.
\end{align*}
The term $\|w^{3p-3}\|_{\B^{\frac 1 2 + 3\eps}_{\frac {3p} {3p-2}}}$ is the same as that appearing on the \lhs\ of \eqref{e.almost.same.below}, up to a rescaling of $\eps$. Hence,
\begin{equation*}  
\|w^{3p-3}\|_{\B^{\frac 1 2 + 3\eps}_{\frac{3p}{3p-2}}} \ls \|w\|_{L^{3p}}^{3p-4+ \frac 1 2} \, \|w\|_{\B^{1 + 6\eps}_{p}}^\frac 1 2.
\end{equation*}
For the other term, by Proposition~\ref{p:mult},
\begin{equation}
\label{e.time.for.kine}
 \|(v- e^{\cdot (\De-c)} v_0+w-\<30>) \pg \<2>\|_{\B^{-\frac 1 2 - 2\eps}_{\frac{3p}{2}}} \ls K \|v - e^{\cdot (\De-c)} v_0+ w - \<30>\|_{\B^{\frac 1 2  - \eps}_{\frac{3p}{2}}}.
\end{equation}
The term involving $\<30>$ poses no difficulty. For the term involving $w$, we use the interpolation bound
\begin{equation*}  
\|w\|_{\B^{\frac 1 2  - \eps}_{\frac{3p}{2}}} \ls \|w\|_{\B^{-2\eps}_{3p}}^\frac 1 2 \|w\|_{\B^{1}_p}^\frac 1 2 \ls \|w\|_{L^{3p}}^\frac 1 2 \|w\|_{\B^{1}_p}^\frac 1 2
\end{equation*}
to get a bound of the form
\begin{equation*}  
K \int_0^t \|w(s)\|_{L^{3p}}^{3p-3} \|w(s)\|_{\B^{1+6\eps}_p} \, \d s,
\end{equation*}
on which we then apply Young's inequality. The remaining term involving $v-e^{\cdot(\Delta - c)}$ in \eqref{e.time.for.kine} is treated by an appeal to Theorem~\ref{t:apriori-v}.

\smallskip

 The other terms in the \lhs\ of \eqref{dot-terms-2} are only simpler than the quadratic terms covered by the previous lemma, so we omit the details.
\end{proof}

\begin{proof}[Proof of Theorem~\ref{t:apriori-w}]
Fir $p \ge 24$ and $\eps > 0$ sufficiently small, and then $\uc$ such that 
\begin{equation*}  
\uc \ge c_1 10^{(5+15p)} K^{30p},
\end{equation*}
where $c_1$ is given by Lemma~\ref{lem:cubic2}. 
 Combining Proposition~\ref{p:testing} with the bounds derived in Lemmas \ref{lem:cubic2}--\ref{le:dots2} (and with Young's inequality and comparisons of norms), we obtain the existence of an exponent $\kappa > 0$ depending only on $p$, and of a constant $C$ depending only on $p$ and $\eps$, such that for every $t \in (0,T]$,
\begin{multline}
\label{e.herewecome}
 \|w(t)\|_{L^{3p-2}}^{3p-2} - \|w_0\|_{L^{3p-2}}^{3p-2} + \int_0^t  \|w(s)\|_{L^{3p}}^{3p} \, \d s \\
 \le C (cK)^{\kappa} \Ll[ 1 + \| v_0  \|_{\B^{-3\eps}_{2p}}^{3p} + \int_0^t   \|w(s)\|_{\B^{1 + 6 \eps}_p}^p   \, \d s \Rr]  + \frac 1 2 \sup_{s \le t} \|w(s)\|_{L^{3p-2}}^{3p-2}.
\end{multline}
Letting $t$ vary over an interval containing $0$, we can absorb the supremum and thus obtain the announced result.
\end{proof}


\section{Conclusion}
\label{s:conc}


In this section, we combine the bounds derived in the previous sections to derive the final estimate on $v$ and $w$. 
As in the rest of the paper, we assume that $p \ge 24$ and that $\eps$ is sufficiently small. From now on, recalling \eqref{e.redef.uc},
\begin{equation}
\label{e.fix.c}
\textbf{we fix } \ c =  c_0 K^{30p},
\end{equation}
where $c_0$ is the constant (depending on $p$) appearing in Theorem~\ref{t:apriori-w}. 

\begin{thm}  
\label{t:sec8}
Let $p \ge 24$, $\eps > 0$ be sufficiently small, and $c$ be fixed according to~\eqref{e.fix.c}. There exists an exponent $\kappa < \infty$ depending only on $p$ and $\eps$ such that for every $t \in [0,T)$, we have
\begin{equation*}  
\|v(t)\|_{\B^{-3\eps}_{2p}} + \sqrt{t} \, \|w(t)\|_{L^{3p-2}} \ls K^\kappa,
\end{equation*}
where the implicit constant depends only on $p$ and $\eps$. 
\end{thm}

The next lemma combines the bounds obtained in Sections~\ref{s:Gronwall-w} and \ref{s:testing-w} into a single estimate, which we then use as the basis for an ODE-type argument similar to the one sketched below \eqref{e.ode.arg}.
\begin{lem}  
\label{l.key.bound.sec8}
Let $p \ge 24$, $\eps > 0$ be sufficiently small, and $c$ be fixed according to~\eqref{e.fix.c}. There exists an exponent $\kappa < \infty$ depending only on $p$ and $\eps$ such that for every $s,t \in [0,T)$, we have
\begin{multline}
\label{e.key.bound.sec8}
 \|w(t)\|_{L^{3p-2}}^{3p-2}  + \int_s^t  \|w(r)\|_{L^{3p}}^{3p} \, \d r +  
 	\int_s^t \|w(r) \|_{\B^{1+7\eps}_p}^p 	\d r 
 	\\
 \ls K^\kappa \Ll[ 1 + \|w(s)\|_{L^{3p-2}}^{3p-2} + \| v(s)  \|_{\B^{-3\eps}_{2p}}^{3p} +  \| w(s) \|_{\B^{1+7\eps}_p}^{\frac{3p-2}{3}} \Rr],
\end{multline}
where the implicit constant depends only on $p$ and $\eps$.
\end{lem}
\begin{proof}
By Theorem~\ref{t:apriori-w} and Corollary~\ref{cor:high-reg-w}, there exists an exponent $\kappa < \infty$ depending only on $p$ and $\eps$ such that for every $s<t \in [0,T)$, we have
\begin{multline}
\label{e:sec8-1}
 \|w(t)\|_{L^{3p-2}}^{3p-2}  + \int_s^t  \|w(r)\|_{L^{3p}}^{3p} \, \d r \\
  \ls \|w(s)\|_{L^{3p-2}}^{3p-2} +  K^{\kappa} \Ll[ 1 + \| v(s)  \|_{\B^{-3\eps}_{2p}}^{3p} + \int_s^t   \|w(r)\|_{\B^{1 + 6 \eps}_p}^p   \, \d r \Rr]  ,
\end{multline}
as well as
\begin{multline}
\label{e:sec8-2}
	\int_s^t 
		\|w(r) \|_{\B^{1+7\eps}_p}^p 
	\d r 
\\
\lesssim 
	\int_s^t \| e^{\Delta (r-s)}
	w(s) 
	\|_{\B^{1+7\eps}_p}^p \d r +  
K^{\kappa} 	
\Big[
	1+
		\int_s^t 
			\| w(r) \|_{L^{3p}}^{3p}  
		\d r +	 
	\| v(s) \|_{\B_{2p}^{-3\eps}}^{3p}  
\Big],
\end{multline}
where the implicit constants depend only on $p$ and $\eps$. We start by simplifying these estimates and putting them in the most convenient possible form for subsequent analysis. In the estimate \eqref{e:sec8-1}, a term involving $\|w(s)\|_{\B^{1+6\eps}_p}$ appears. This term can be estimated by a power strictly less than $1$ of the quantity on the left-hand side of \eqref{e:sec8-2}. More precisely, by the interpolation bound \eqref{e.interpol.4eps} and Young's inequality, we have
\begin{equation}
\label{e.interpol.4eps.b}
\| w \|_{\B^{1+6\eps}_p} \le \| w \|_{\B^{1+7\eps}_p}^{\frac{1+6\eps}{1+7\eps}} \|w \|_{L^p}^{\frac{\eps}{1+7\eps}} \le \delta^{-\kappa} \| w \|_{\B^{1+7\eps}_p}^{\frac{3+18\eps}{3+20\eps}} 
+ \delta\| w \|_{L^p}^3,
\end{equation}
for some exponent $\kappa < \infty$ depending on $\eps$. Setting
\begin{equation*}  
\sigma := \frac{3+18\eps}{3+20\eps},
\end{equation*}
we deduce from \eqref{e:sec8-1}, \eqref{e:sec8-2} and \eqref{e.interpol.4eps.b} that, after enlarging the exponent $\kappa < \infty$ as necessary,
\begin{multline}
\label{e:sec8-3}
 \|w(t)\|_{L^{3p-2}}^{3p-2}  + \int_s^t  \|w(r)\|_{L^{3p}}^{3p} \, \d r \\
  \ls \|w(s)\|_{L^{3p-2}}^{3p-2} +  K^{\kappa} \Ll[ 1 + \| v(s)  \|_{\B^{-3\eps}_{2p}}^{3p} + \int_s^t   \|w(r)\|_{\B^{1 + 7 \eps}_p}^{\sigma p}   \, \d r \Rr]  .
\end{multline}
After enlarging again the exponent $\kappa < \infty$ as necessary, we infer that uniformly over $\de \in (0,1]$, we have
\begin{multline}
\label{e:sec8-3b}
 \|w(t)\|_{L^{3p-2}}^{3p-2}  + \int_s^t  \|w(r)\|_{L^{3p}}^{3p} \, \d r \\
  \ls \|w(s)\|_{L^{3p-2}}^{3p-2} + K^\kappa \| v(s)  \|_{\B^{-3\eps}_{2p}}^{3p} +\de \int_s^t   \|w(r)\|_{\B^{1 + 7 \eps}_p}^{p}   \, \d r + (\delta^{-1}K)^{\kappa}  .
\end{multline}
We now estimate the term involving the initial datum $w(s)$ in \eqref{e:sec8-2}. We observe that, for $\gamma := (1+7\eps) (1 - \frac{1}{p})$, we have
\begin{equation*}  
\int_s^t 
	\| e^{\Delta(r-s)} w(s) \|_{\B^{1+7\eps}_p}^p \d r 
\lesssim 
\| w(s) \|_{\B^{\gamma}_p}^p
\int_{s}^t 
	\Big( 
		\frac{1}{(r-s)^{\frac{\gamma - (1+7\eps)}{2}}} 
	\Big)^p 
\d r  
\lesssim \| w(s) \|_{\B^{\gamma}_p}^p,
\end{equation*}
since $\frac{1}{2} (\gamma -(1+7\eps)) = \frac{1}{2p}(1+7\eps) < \frac{1}{p}$. We then use interpolation (Proposition~\ref{p:interpol}) and Young's inequality, in the form
\begin{equation*}  
\| w(s) \|_{\B^{\gamma}_p}^p \le \|w(s)\|_{\B^{1+7\eps}_p}^{p-1} \, \|w(s)\|_{L^p} \le \|w(s)\|_{\B^{1+7\eps}_p}^{\frac{3p-2}{3}} + \|w(s)\|_{L^p}^{3p-2} ,
\end{equation*}
and combine this with \eqref{e:sec8-2} to arrive at
\begin{multline}
\label{e:sec8-4}
	\int_s^t \|w(r) \|_{\B^{1+7\eps}_p}^p 	\d r 
\\
\lesssim 
 \| w(s) \|_{\B^{1+7\eps}_p}^{\frac{3p-2}{3}} +
   \| w(s) \|_{L^p}^{3p-2} +
K^{\kappa} 	
\Big[
	1+
		\int_s^t 
			\| w(r) \|_{L^p}^{3p}  
		\d r +	 
	\| v(s) \|_{\B_{2p}^{-3\eps}}^{3p}  
\Big].
\end{multline}
Multiplying the estimate \eqref{e:sec8-3b} by $2K^\kappa$, summing it with \eqref{e:sec8-4} and simplifying, we obtain that for some exponent $\kappa < \infty$ sufficiently large,
\begin{multline*}
 \|w(t)\|_{L^{3p-2}}^{3p-2}  + \int_s^t  \|w(r)\|_{L^{3p}}^{3p} \, \d r +  
 	\int_s^t \|w(r) \|_{\B^{1+7\eps}_p}^p 	\d r 
 	\\
 \ls K^\kappa \Ll[ \de^{-\kappa} + \|w(s)\|_{L^{3p-2}}^{3p-2} + \| v(s)  \|_{\B^{-3\eps}_{2p}}^{3p} +  \| w(s) \|_{\B^{1+7\eps}_p}^{\frac{3p-2}{3}} + \de \int_s^t   \|w(r)\|_{\B^{1 + 7 \eps}_p}^{p}   \, \d r \Rr].
\end{multline*}
Selecting $\de= K^{-\kappa}/(2C)$, where $C$ is the constant implicit in the last $\ls$, then yields the announced result.
\end{proof}
The next lemma exposes the general principle by which, with the help Lemma~\ref{l.key.bound.sec8}, we obtain the sought-after power-law decay of $\|w(t)\|_{L^{3p-2}}$. 
\begin{lem}\label{le:8-3}
Let $\tau > 0$, $\lambda > 1$, $\mathfrak{c} >0$, and let $F \colon [0,\tau) \to [0,\infty)$ be a continuous function such that
for every $s < t \in [0,\tau)$, we have
\begin{align}\label{e:sec8-6A}
\int_s^t F^{\lambda}(r) \d r \leq \mathfrak{c} F(s).
\end{align}
There exist an integer $N \geq 1$ and a sequence $0 = t_0 <  t_1 < t_2 < \ldots < t_{N} = \tau$
such that for every $n \in \{0, \ldots, N-1\}$,
%
\begin{align}\label{e:sec8-6B}
F(t_n)  \lesssim  \mathfrak{c}^{\frac{1}{\lambda-1}} t_{n+1}^{-\frac{1}{\lambda-1}}  ,
\end{align}
where the implicit constant depends only on $\lambda$.
\end{lem}

\begin{proof}
We define $t_0=0$ and 
\begin{align}\notag
t_1^* = \mathfrak{c}2^{\lambda}F(0)^{1-\lambda}.
\end{align}
If $t_1^* \geq \tau$, then we set $N=1$, $t_1 = \tau$, and we verify that \eqref{e:sec8-6B} holds. Otherwise, we  evaluate \eqref{e:sec8-6A} for $s=0$ and $t = t_1^{*}$, writing
\begin{align*}
t_1^{*} \min_{0 \leq r \leq t_1^*} F^{\lambda}(r) \leq \int_0^{t_1^{*}} F^{\lambda}(r) \d r \le \mathfrak{c} F(0).
\end{align*}
By the definition of $t_1^*$, this implies
\begin{align*}
\min_{0 \leq r \leq t_1^*} F^{\lambda}(r) \leq 2^{-\lambda} F^{\lambda}(0).
\end{align*}
We then denote by $t_1$ the smallest value of $r$ for which this minimum is realised, and summarise  this first step of our induction in the bounds
\begin{align}
F(t_1) \leq \frac12  F(0)  \qquad \text{and} \qquad
t_1 \leq \mathfrak{c}2^{\lambda}F(0)^{1-\lambda}. 
\label{e:sec8-7}
\end{align}
We now iterate  this construction, and construct $t_{n+1}$ assuming
that $t_0< t_1 < \ldots < t_n$ have been constructed and that $t_n < \tau$.  
We set
$t_{n+1}^* = t_n + \mathfrak{c}2^{\lambda}F(t_n)^{1-\lambda}$. As before,
 if $t_{n+1}^* \geq \tau$, then we terminate the recursion and set $N= n+1$ and $t_{n+1} = \tau$.  Otherwise, we define $t_{n+1}$ as the smallest value of $r$ for which the minimum 
$\min_{t_n \leq r \leq t_{n+1}^*} F^{\lambda}(r)$ is attained. 
As in the initial step, this implies
\begin{align*}
F(t_{n+1})\leq \frac12 F(t_n) \qquad \text{and} \qquad 
t_{n+1} \leq t_n + \mathfrak{c}2^{\lambda}F^{1-\lambda}(t_n).
\end{align*}
This procedure necessarily terminates after finitely many steps.
Indeed, the first of these estimates can be rewritten as
\begin{align}\label{e:sec8-7B}
2^{\lambda-1} F^{1-\lambda}(t_{n}) \leq F^{1-\lambda}(t_{n+1})   ,
\end{align}
and thus, in each iteration, the proposed time-step $t_{n+1}^* -t_n= \mathfrak{c}2^{\lambda}F(t_n)^{1-\lambda}$ is at least multiplied by a factor of $2^{\lambda-1} > 1$, so that it has to exceed $\tau$ after finitely many steps.
In order to establish \eqref{e:sec8-6B},
we then write, for every $n\in \{0, \ldots, N-1\}$,
\begin{align}
\label{e:sec8-8}
t_{n+1} 
= \sum_{j=0}^n 
	(t_{j+1} - t_j) 
\leq 
\mathfrak{c} 2^{\lambda} \sum_{j=0}^n F^{1-\lambda}(t_j).
\end{align}
The key observation is now that the sum appearing on the right-hand side 
of this identity is dominated by the term $F^{1-\lambda}(t_n)$. Indeed, by induction on \eqref{e:sec8-7B}, we see that for every $j\leq n$,
\begin{align*}
F(t_j)^{1-\lambda} \leq 2^{(1-\lambda)(n-j) } F(t_n).
\end{align*}
Plugging this into the sum on the right-hand side of \eqref{e:sec8-8} yields
\begin{align*}
\sum_{j=0}^n F^{1-\lambda} (t_j)
\leq 
F^{1-\lambda}(t_n) 
\sum_{j=0}^n 
2^{(1-\lambda)(n-j)} 
\leq F^{1-\lambda}(t_n) 
\sum_{j=0}^{\infty} 
2^{(1-\lambda) j}.
\end{align*}
Combining this with \eqref{e:sec8-8} yields
\begin{align*}
t_{n+1} \lesssim \mathfrak{c} F^{1-\lambda}(t_n) ,
\end{align*}
which is equivalent to \eqref{e:sec8-6B} and thus completes the argument.
\end{proof}
\begin{proof}[Proof of Theorem~\ref{t:sec8}]
For every $t \in [0,T)$, we define
\begin{align}\label{e:sec-8-def-F}
F(t) := \|w(t) \|_{\B^{1+7\eps}_p}^{\frac{3p-2}{3}} + \|w(t)\|_{L^{3p-2}}^{3p-2},
\end{align}
as well as 
\begin{align}
\label{e:sec-8-def-tau}
\tau := \inf\{ t \ge 0 \ : \  
F(t) \leq 1
\text{ or } 
\| v(t) \|_{\B^{-3\eps}_{2p}}^{3p} \geq F(t)  
\} \wedge T.
\end{align}
Lemma~\ref{l.key.bound.sec8} implies that, for every $s<t \in [0,\tau)$,
\begin{align}\label{e:sec8-6}
\int_s^t F^{\frac{3p}{3p-2}}(r) \, \d r   \lesssim K^{\kappa} F(s).
\end{align}
Due to our assumption of $v(0) = 0$ and the continuity properties of $v$ and $w$, either $F(0) \leq 1$, or $\tau>0$. If $\tau > 0$, then by Lemma~\ref{le:8-3}, there exists a positive integer $N\ge 1$ and a sequence of times $0= t_0 < t_1 < \ldots < t_N = \tau$ such that for every $n < N$,
\begin{align}\label{e:sec8-6BB}
\|w(t_n) \|_{\B^{1+5\eps}_p}^{\frac{1}{3}}  + \|w(t_n)\|_{L^{3p-2}} \lesssim  K^{\kappa} t_{n+1}^{-\frac{1}{2}}  .
\end{align}  
We now aim to extend this bound to get a control for arbitrary $t \in (0,T)$. We decompose the argument into four steps. 

\smallskip

\emph{Step 1.} 
We consider the case $\tau > 0$ and $t < \tau$. In this situation, there exists a positive integer $n < N$ such that $t_n \leq t < t_{n+1}$, and moreover, for every $s < t$, we have 
\begin{equation*}
\| v(s) \|_{\B^{-3\eps}_{2p}}^{3p} \leq \|w(s) \|_{\B^{1+7\eps}_p}^{\frac{3p-2}{3}} + \|w(s)\|_{L^{3p-2}}^{3p-2}.
\end{equation*}
By Lemma~\ref{l.key.bound.sec8} applied with $s = t_n$ and the previous display, we infer that
\begin{equation*}  
\|w(t)\|_{L^{3p-2}}^{3p-2} \ls K^\kappa \Ll[ 1+ \|w(t_n) \|_{\B^{1+7\eps}_p}^{\frac{3p-2}{3}} + \|w(t_n)\|_{L^{3p-2}}^{3p-2}\Rr] ,
\end{equation*}
and by \eqref{e:sec8-6BB}, we deduce
\begin{equation*}  
\|w(t)\|_{L^{3p-2}} \ls K^{\kappa} t_{n+1}^{-\frac{1}{2}}
\lesssim K^{\kappa} t^{-\frac{1}{2}}.
\end{equation*}

\smallskip

\emph{Step 2.} 
Define
\begin{equation*}  
\tau' := \inf\{s \ge 0 \ : \ \|v(s)\|_{\B^{-3\eps}_{2p}}^{3p} \ge F(s)\} \wedge T. 
\end{equation*}
We clearly have $\tau \le \tau'$. In this step, we study the possibility that $\tau < \tau'$, and aim to cover times $t \in [\tau,\tau')$. Under these conditions, we have $F(\tau) \le 1$ as well as, for every $s < \tau'$,
\begin{equation}  
\label{e.kickass1}
\|v(s)\|_{\B^{-3\eps}_{2p}}^{3p} \le F(s).
\end{equation}
An application of Lemma~\ref{l.key.bound.sec8} with $s = \tau$ then yields
\begin{equation*}  
\|w(t)\|_{L^{3p-2}} \ls K^{\kappa}.
\end{equation*}

\smallskip

\emph{Step 3.} In this step, we consider the remaining case when $t \in [ \tau',T)$. By the result of the previous two steps, we have
\begin{equation*}  
\forall s \le \tau', \ \|w(s)\|_{L^{3p-2}} \ls K^\kappa s^{-\frac 1 2}.
\end{equation*}
By Theorem~\ref{t:apriori-v} and the assumption of $v_0 = 0$, we get
\begin{equation}  
\label{e.now.is.time.to.understand}
\|v(\tau')\|_{\B^{-3\eps}_{2p}} \ls K^\kappa \int_0^{\tau'} \frac{1}{(\tau'-s)^{\frac 1 2-\eps}} \, s^{-\frac 1 2} \, \d s \ls K^\kappa.
\end{equation}
The estimate above is the reason why we were careful to measure $v(\tau')$ in a Besov space with an integrability exponent $2p$ ($3p-2$ would be sufficient, but $3p$ is more problematic). 
By continuity and the definition of $\tau'$, we deduce that
\begin{equation*}  
F(\tau') = \|w(\tau') \|_{\B^{1+5\eps}_p}^{\frac{3p-2}{3}} + \|w(\tau')\|_{L^{3p-2}}^{3p-2} \ls K^\kappa,
\end{equation*}
and by an application of Lemma~\ref{l.key.bound.sec8} with $s = \tau'$, we obtain
\begin{equation*}  
\|w(t)\|_{L^{3p-2}} \ls K^\kappa.
\end{equation*}

\smallskip

\emph{Step 4.} We now conclude. Combining the results of the previous steps yields that for every $t \in (0,T)$,
\begin{equation*}  
\|w(t)\|_{L^{3p-2}} \ls K^\kappa t^{-\frac 1 2}.
\end{equation*}
Arguing as in Step 3, we deduce that for every $t \in [0,T)$,
\begin{equation*}  
\|v(t)\|_{\B^{-3\eps}_{2p}} \ls K^\kappa,
\end{equation*}
and this completes the proof.
\end{proof}

\appendix

\section{products and paraproducts in Besov spaces}
\label{s:Besov}


The goal of this appendix is to collect several estimates used throughout the paper concerning Besov spaces, paraproducts, and the regularizing properties of the heat semigroup. Some of these results appeared independently in \cite{promel}.

\smallskip

We denote by $C^\infty_{\per}$ the space of $\Z^d$-periodic infinitely differentiable functions. For $p \in [1,\infty]$, we write $L^p = L^p([-1,1]^d,\d x)$, with associated norm $\|\, \cdot \, \|_{L^p}$. We write $\langle \cdot, \cdot \rangle$ for the scalar product in $L^2$. We denote by $\|\, \cdot \,\|_{\bar L^p}$ the norm of the space $L^p(\R^d, \d x)$. For $u = (u_n)_{n \in I}$ with $I$ a countable set, we write
$$
\|u\|_{\ell^p} := \Ll(\sum_{n \in I} |u_n|^p \Rr)^\frac 1 p,
$$
with the usual interpretation as a supremum when $p = \infty$. We write $\F f$ or $\hat{f}$ for the Fourier transform (and by $\F^{-1}$ its inverse), which is well-defined for any Schwartz distribution $f$ on $\R^d$, and reads, for $f \in L^1(\R^d)$,
$$
\F f(\zeta) = \hat f(\zeta) = \int e^{-i x \cdot \zeta} f(x) \, \d x.
$$

\subsection{Besov spaces}

We recall briefly a construction of Besov spaces on the torus. Following \cite[Proposition~2.10]{BCD}, there exist $\td{\chi}, \chi \in C^\infty_c$ taking values in $[0,1]$ and such that
\begin{equation}
\label{chi-prop1}
\supp \, \td{\chi} \subset B(0,4/3),  
\end{equation}
\begin{equation}
\label{chi-prop2}
\supp \, {\chi} \subset B(0,8/3) \setminus B(0,3/4),
\end{equation}
\begin{equation}
\label{chi-prop3}
\forall \ze \in \R^d, \ \td{\chi}(\ze) + \sum_{k = 0}^{+\infty} \chi(\ze/2^k) = 1.
\end{equation}
We use this partition of unity to decompose any function $f \in C^\infty_\per$ as a sum of functions with localized spectrum. More precisely, we define
\begin{equation}
\label{e:def:chik}
\chi_{-1} = \td{\chi}, \qquad \chi_k = \chi(\cdot/2^k) \quad (k \ge 0),
\end{equation}
and for $k \ge -1$ integer,
$$
\delta_k f = \F^{-1} \Ll(\chi_k \, \hat{f}\Rr), \qquad S_k f = \sum_{j < k} \delta_j f
$$
(where the sum runs over $j \ge -1$), so that at least formally,
$
f = \sum \delta_k f.
$
We let
\begin{equation}
\label{e:def:etak}
\eta_k = \F^{-1}(\chi_k), \qquad \eta = \eta_0,
\end{equation}
so that for $k \ge 0$, $\eta_k = 2^{kd} \eta(2^k\, \cdot\, )$, and for every $k$,
\begin{equation}
\label{e.dk-convol}
\dk f = \eta_k \star f,
\end{equation}
where $\star$ denotes the convolution. 
For every $\alpha \in \R$, $p,q \in [1,+ \infty]$ and $f \in C^\infty_\per$, we define 
\begin{equation}
\label{e.def.B-norm}
\| f \|_{\B^\al_{p,q}}:= \Ll\| \Ll( 2^{\al k } \| \dk f \|_{L^p} \Rr)_{k \ge -1} \Rr\|_{\ell^q}.
\end{equation}
It is easy to check that this quantity is finite (see \cite[Lemma~3.2]{JCH}).
We define the Besov space $\B^\al_{p,q}$ as the completion of $C^\infty_\per$ with respect to this norm. Outside of this appendix, we use the shorthand notation $\B^\al_{p} := \B^\al_{p,\infty}$.

We first state a duality relation between Besov spaces, see \cite[Proposition~2.76]{BCD}.
\begin{prop}[Duality]
\label{p:dual}
Let $\al \in \R$, and $p,q,p',q' \in [1,\infty]$ be such that 
\begin{equation}
\label{e.conjugate-pq}
\frac 1 p + \frac 1 {p'} = 1, \qquad \frac 1 q + \frac 1 {q'} = 1.
\end{equation} 
The mapping $(f,g) \mapsto \langle f, g \rangle$ (defined for $f,g \in C^\infty_\per$) can be extended to a continuous bilinear form on $\B^\al_{p,q} \times \B^{-\al}_{p',q'}$.
\end{prop}
In particular, we can think of Besov spaces as being all embedded in the space of Schwartz distributions.

Clearly, $\B^\al_{p_1,q_1}$ is continuously embedded in $\B^{\be}_{p_2,q_2}$ if $\be \le \al$, $p_2 \le p_1$ and $q_2 \ge q_1$. We also have the following embeddings (cf.\ \cite[Proposition~3.7]{JCH}).
\begin{prop}[Besov embedding]
\label{p:embed}
Let $\al \le \be \in \R$ and $p \ge r \in [1,\infty]$ be such that
$$
\beta = \al + d \Ll( \frac 1 r - \frac 1 p \Rr).
$$
There exists $C < \infty$ such that
$$
\|f\|_{\B^\al_{p,q}} \le C \|f\|_{\B^\be_{r,q}}.
$$
\end{prop}
\begin{rem}
\label{r:Besov-vs-Lp}
By \cite[Remarks~3.5 and 3.6]{JCH}, there exists $C < \infty$ such that
\begin{equation*}
C^{-1} \, \|f\|_{B^0_{p,\infty}} \le \|f\|_{L^p} \le C \,  \|f\|_{\B^0_{p,1}}.
\end{equation*}
\end{rem}


An application of H\"older's inequality (see \cite[Proposition~3.10]{JCH}) yields the following interpolation result.
\begin{prop}[Interpolation inequalities]
\label{p:interpol}
Let $\al_0, \al_1 \in \R$, $p_0,q_0,p_1,q_1 \in [1,\infty]$ and $\nu \in [0,1]$. Defining $\al = (1-\nu)\al_0 + \nu \al_1$ and $p,q \in [1,\infty]$ such that
$$
\frac{1}{p} = \frac{1-\nu}{p_0} + \frac{\nu}{p_1} \quad \text{ and } \quad \frac{1}{q} = \frac{1-\nu}{q_0} + \frac{\nu}{q_1},
$$
we have
$$
\|f\|_{\B^\al_{p,q}} \le \|f\|^{1-\nu}_{\B^{\al_0}_{p_0,q_0}} \, \|f\|^\nu_{\B^{\al_1}_{p_1,q_1}}.
$$
\end{prop}

The effect of differentiating (in the sense of distributions) an element Besov space is described as follows (see e.g.\ \cite[Proposition~3.8]{JCH}).
\begin{prop}[Effect of differentiating]
\label{p:derivatives}
Let $\al \in \R$ and $p,q \in [1,\infty]$. For every $i \in \{1,\ldots, d\}$, the mapping $f \mapsto \dr_i f$ extends to a continuous linear map from $\B^{\al}_{p,q}$ to $\B^{\al-1}_{p,q}$.
\end{prop}

The following extends \cite[Proposition~3.25]{JCH} by allowing $\al = 1$ and arbitrary values of $p$. 
\begin{prop}[Estimate in terms of $\na f$]
\label{p:sobolev}
Let $\alpha \in (0,1]$ and $p,q \in [1,\infty]$. When $\al = 1$, we also impose $q = \infty$. There exists $C < \infty$ such that
$$
C^{-1} \| f\|_{\B^{\al}_{p,q}} \le  \| f \|_{L^p}^{1- \alpha} \, \| \na f \|_{L^p}^\al  + \| f \|_{L^p}\;.
$$
\end{prop}
\begin{proof}
We decompose the proof into two steps. 

\noindent \emph{Step 1.} We show the result for $\al \in (0,1)$. By comparison of norms, it suffices to show the result for $q = 1$. We assume $p < \infty$, the case $p = \infty$ being similar. Let $f$ be a smooth, one-periodic function. For $\ell \ge 0$, we define the projectors
\begin{align*}
 \Pl f = \sum_{-1 \leq k < \ell} \delta_k f \qquad \text{and} \qquad \Plp f = \sum_{k \ge \ell } \delta_k f \;,
\end{align*}
so that $f = \Pl f + \Plp f$, and by the triangle inequality,
\begin{align*}
\| f\|_{\B^{\al}_{p,1}} \leq  \| \Pl f \|_{\B^{\al}_{p,1}} + \| \Plp f \|_{\B^{\al}_{p,1}} \;.
\end{align*}
For the first term, recalling \eqref{e:def:etak} and \eqref{e.dk-convol}, we have
\begin{equation}
\label{e.Young-torus} 
\| \delta_k f \|_{L^p} = \| \eta_k \star f \|_{{L}^p} \le \| \td \eta_k \|_{L^1} \, \|  f \|_{{L}^p},
\end{equation}
where we used Young's convolution inequality on the torus and set
\begin{equation}
\label{e.def.tdeta}
\td \eta_k := \sum_{y \in (2\Z)^d} \eta_k(\cdot + y). 
\end{equation}
Recall that $\eta_k = 2^{kd} \eta(2^k \cdot)$. By scaling and rapid decay to $0$ at infinity of $\eta$, we have 
\begin{equation}
\label{e.estim-L1}
\sup_{k \ge -1} \|\td \eta_k\|_{L^1} < \infty,
\end{equation}
and thus
%
\begin{equation}
\label{e:goodBound1}
\| \Pl f \|_{\B^{\al}_{p,1}} = \sum_{-1 \leq k < \ell } 2^{k \alpha} \| \delta_k f \|_{L^p}  \ls 2^{\ell \alpha} \| f\|_{{L}^p}.
\end{equation}
On the other hand, using the fact that for $k \geq 0$, the function $\eta_k$ has vanishing integral, we get 
\begin{align*}
\| \Plp f \|_{\B^{\al}_{p,1}} &= \sum_{ k\geq \ell } 2^{k \alpha} \| \delta_k f \|_{L^p} \\
&=  \sum_{k\ge \ell } 2^{-k(1-\alpha)} \Ll( \int_{[-1,1]^d} \Big|\int_{\R^d} 2^k \eta_k(y) \big(f(x-y) -f(x) \big) \,dy \Big|^p  \, dx \Rr)^{\frac 1 p} 
\end{align*}
By H\"older's inequality, the integral above is bounded by
$$
 \||2^k \, \cdot \,|\eta_k\|_{\bar L^1}^{p-1} \, \int_{[-1,1]^d} \int_{\R^d} | 2^{k} y \,  \eta_k(y)|  \frac{|f(x-y) -f(x)|^p}{|y|^p}\, dx  \, dy,
$$
where we recall that $\| \cdot \|_{\bar L^1}$ denotes the $L^1$ norm in the full space $\R^d$.
For every $x,y \in \R^d$,
\begin{align*}
\frac{|f(x-y) -f(x)|^p}{|y|^p}  =& \frac{1}{|y|^p} \Big|\int_{0}^1 \nabla f(x -ty) \cdot y \,  dt\Big|^p \\
\leq& \int_{0}^1 \big| \nabla f(x -ty)\big|^p  \,  dt.
\end{align*}
Therefore,
$$
\int_{[-1,1]^d} \int_{\R^d} | 2^{k} y \,  \eta_k(y)|  \frac{|f(x-y) -f(x)|^p}{|y|^p}\, dx  \, dy \; \le  \| \, |2^k \cdot  | \, \eta_k \|_{\bar L^1} \, \|\na f\|_{{L}^p}^p.
$$
Noting that $\| \, |2^k \cdot  | \, \eta_k \|_{L^1}$ is finite and independent of $k \ge 0$ by scaling, we obtain 
$$
\| \Plp f \|_{\B^{\al}_{p,1}} \ls 2^{-\ell(1-\al)} \|\nabla f\|_{ L^p},
$$
so that uniformly over $\ell \ge 0$,
$$
\| f\|_{\B^{\al}_{p,1}} \ls 2^{\ell \al} \| f \|_{{L}^p} + 2^{-\ell (1-\alpha)} \|\na f\|_{{L}^1}.
$$
The result then follows by optimizing over $\ell$.

\medskip

\noindent \emph{Step 2.} We show the result for $\al = 1$ and $q = \infty$. This is a minor modification of the arguments of the previous step. Indeed, we have
$$
\| \mathcal{P}_0 f \|_{\B^{1}_{p,\infty}} = \|\delta_{-1} f\|_{L^p} \ls  \| f\|_{{L}^p},
$$
while
$$
\| \mathcal{P}_0^{\perp} f \|_{\B^{1}_{p,\infty}} = \sup_{ k\geq 0 } 2^{k} \| \delta_k f \|_{L^p} ,
$$
and we have seen that the latter is bounded by a constant times $\|\nabla f\|_{L^p}$, so the proof is complete.
\end{proof}

\subsection{Paraproducts}
\label{ss.paraprod}
As in \cite{Gubi}, the basis of our analysis rests on the regularity properties of paraproducts. 
For $f, g \in C^\infty_\per$, we define the \emph{paraproduct}
$$
f \pl g = \sum_{j < k-1} \de_j f \ \de_k g = \sum_{k} S_{k-1} f \ \de_k g,
$$
and the \emph{resonant term}
$$
f \pe g = \sum_{|j-k| \le 1} \de_j f \ \de_k g.
$$
We write $f \pg g = g \pl f$. At least formally, we have the \emph{Bony decomposition}
\begin{equation*}
fg = f \pl g + f \pe g + f \pg g.
\end{equation*}
We will also use the symbols $\ple = \pl + \pe$, etc.

The most important estimates for our purpose are summarised in the following proposition (see \cite[Theorems 2.82, 2.85 and Corollary 2.86]{BCD} or \cite[Theorem~3.17 and Corollaries~3.19 and 3.21]{JCH}).
\begin{prop}[paraproduct estimates]
\label{p:mult} 
Let $\al, \be \in \R$ and $p,p_1,p_2,q \in [1,\infty]$ be such that
$$
\frac 1 p = \frac 1 {p_1} + \frac 1 {p_2}.
$$

\noindent $\bullet$ If $\al + \be > 0$, then the mapping $(f,g) \mapsto f \pe g$ extends to a continuous bilinear map from $\B^\al_{p_1,q} \times \B^\be_{p_2,q}$ to $\B^{\al + \be}_{p,q}$. 

\noindent $\bullet$ The mapping $(f,g) \mapsto f \pl g$ extends to a continuous bilinear map from $L^{p_1} \times \B^\be_{p_2,q}$ to $\B^\be_{p,q}$.

\noindent $\bullet$ If $\al < 0$, then the mapping $(f,g) \mapsto f \pl g$ extends to a continuous bilinear map from $\B^{\al}_{p_1,q} \times \B^\be_{p_2,q}$ to $\B^{\al + \be}_{p,q}$.

\noindent $\bullet$ If $\al < 0 < \be$ and $\al + \be > 0$, then the mapping $(f,g) \mapsto fg$ extends to a continuous bilinear map from $\B^\al_{p_1,q} \times \B^\be_{p_2,q}$ to $\B^{\al}_{p,q}$. 

\noindent $\bullet$ If $\al > 0$, then the mapping $(f,g) \mapsto fg$ extends to a continuous bilinear map from $\B^\al_{p_1,q} \times \B^{\al}_{p_2,q}$ to $\B^\al_{p,q}$. Moreover, for $p_3,p_4 \in [1,\infty]$ such that 
\begin{equation*}  
\frac 1 p = \frac 1 {p_1} + \frac 1 {p_2} = \frac 1 {p_3} + \frac 1 {p_4},
\end{equation*}
there exists $C < \infty$ satisfying
\begin{equation}  
\label{e.sharp.mult}
\|fg\|_{\B^\al_{p,q}} \le C \Ll(\|f\|_{L^{p_1}} \, \|g\|_{\B^\al_{p_2,q}} + \|f\|_{\B^\al_{p_3,q}} \, \|g\|_{L^{p_4}}  \Rr) .
\end{equation}
\end{prop}
We also record the following convenient corollary to Proposition~\ref{p:mult}.
\begin{cor}  
\label{c:mult}
Let $\al> 0$ $r \in \N$ and $p,p_1,p_2, q \in [1,\infty]$ be such that 
\begin{equation*}  
\frac 1 p = \frac 1 {p_1} + \frac 1 {p_2}.
\end{equation*}
There exists $C < \infty$ such that
\begin{equation*}  
\|f^{r+1}\|_{\B^{\al}_{p,q}} \le C \|f^r\|_{L^{p_1}} \, \|f\|_{\B^{\al}_{p_2,q}}.
\end{equation*}
\end{cor}
\begin{proof}
The result follows by induction on \eqref{e.sharp.mult}. 
\end{proof}


The next result is our first commutator estimate. It extends \cite[Lemma~2.4]{Gubi} to more general Besov spaces.

\begin{prop}[commutation between $\pl$ and $\pe$]
\label{p:comm1}
Let $\al  < 1$, $\be, \ga \in \R$ and $p$, $p_1$, $p_2$, $p_3 \in [1,\infty]$ be such that
$$
\be + \ga < 0, \qquad  \al + \be + \ga > 0  \qquad \text{and} \qquad  \frac 1 p = \frac 1 {p_1} + \frac 1 {p_2} + \frac 1 {p_3}.
$$ 
The mapping
\begin{equation}
\label{e:comm1}
[\pl,\pe] : (f,g,h) \mapsto (f \pl g) \pe h - f  (g \pe h)
\end{equation}
extends to a continuous trilinear map from $\B^\al_{p_1,\infty} \times \B^\be_{p_2,\infty} \times \B^\ga_{p_3,\infty}$ to $\B^{\al + \be + \ga}_{p,\infty}$. 
\end{prop}
The proof of Proposition~\ref{p:comm1} relies on the following two lemmas.
\begin{lem}
\label{l:com-step1}
For $f,g \in C^\infty$, define
$$
[\dk, f](g) = \dk(fg)- f \, \dk g.
$$
Let $p,p_1,p_2 \in [1,\infty]$ be such that $\frac 1 p = \frac 1 {p_1} + \frac 1 {p_2}$. There exists $C < \infty$ such that for every $k \ge 0$ and $f,g \in C^\infty_\per$,
$$
\|[\dk, f](g)\|_{L^p} \le \frac{C}{2^k} \|\na f\|_{L^{p_1}} \, \|g\|_{L^{p_2}}.
$$
\end{lem}
\begin{proof}
The proof is similar to that of \cite[Lemma~2.97]{BCD}. 
\end{proof}

\begin{lem}
\label{l:com-step2}
For $f, g \in C^\infty$, define
\begin{equation}
\label{e:def:com-dk}
[\dk,\pl](f,g) := \dk(f \pl g) - f  (\dk g).
\end{equation}
Let $p,p_1,p_2 \in [1,\infty]$ be such that $\frac 1 p = \frac 1 {p_1} + \frac 1 {p_2}$, $\al \in(0, 1)$ and $\be \in \R$. There exists $C < \infty$ such that for every $f,g \in C^\infty_\per$,
$$
\|[\dk,\pl](f,g)\|_{L^p} \le C  2^{-k(\al + \be)} \, \|f\|_{\B^{\al}_{p_1,\infty}} \, \|g\|_{\B^{\be}_{p_2,\infty}}.
$$
\end{lem}
\begin{rem}
\label{}
It would perhaps be more natural to define the commutator between $\dk$ and~$\pl$ as 
\begin{equation}
\label{e:com-dk-2}
\dk (f \pl g) - f \pl (\dk g)
\end{equation}
(and similarly for \eqref{e:comm1}).
However, the definition in \eqref{e:def:com-dk} will be more convenient to work with in the proof of Proposition~\ref{p:comm1} (besides matching the choice of \cite{Gubi}). 
\end{rem}
\begin{proof}
We decompose the proof into two steps, the first one being focused on deriving bounds for the quantity in \eqref{e:com-dk-2}. 

\medskip 

\noindent \emph{Step 1}. We show that 
\begin{equation}
\label{e:com-dk-step1}
\|\dk (f \pl g) - f \pl (\dk g)\|_{L^p} \le C  2^{-k(\al + \be)} \, \|\na f\|_{\B^{\al-1}_{p_1,\infty}} \, \|g\|_{\B^{\be}_{p_2,\infty}}.
\end{equation}
(The proof given now shows that \eqref{e:com-dk-step1} is also valid when $\al \le 0$.)
Note that 
\begin{align*}
\dk \Ll(f \pl  g\Rr) - f \pl (\dk g) & = \sum_{i = 0}^{+\infty}  \dk \Ll(S_{i-1}f \, \de_i g \Rr) - S_{i-1}f  \, \de_i \dk g .
\end{align*}
The term $\de_i \de_k g = \dk \de_i g$ vanishes unless $|i-k| \le 1$. Moreover, for any $h$, 
the Fourier spectrum of $S_{i-1} f \, \de_i h$ is contained in $2^i {\mathscr A}$, where ${\mathscr A}$ is the annulus $B(0,10/3) \setminus B(0,1/12)$. Hence, $\dk \Ll( S_{i-1} f \, \de_i  h \Rr)$ vanishes unless $|i-k| \le 5$, and 
$$
\dk \Ll(f \pl  g\Rr) - f \pl (\dk g)  = \sum_{|i-k| \le 5} [\dk,S_{i-1} f](\de_i g).
$$
By Lemma~\ref{l:com-step1}, 
$$
\|[\dk,S_{i-1} f](\de_i g)\|_{L^p} \ls \frac{1}{2^k} \, \|\na S_{i-1}f \|_{L^{p_1}} \, \|\de_i g\|_{L^{p_2}}.
$$
Since we assume $\al < 1$, we have
$$
\|\na S_{i-1}f \|_{L^{p_1}} \le \sum_{j < i-1} \|\de_j \Ll( \na f \Rr) \|_{L^{p_1}} \ls 2^{i(1-\al)} \|\na f\|_{\B^{\al-1}_{p_1,\infty}}.
$$
Using also the fact that $\|\de_i g\|_{L^{p_2}} \le 2^{-i \be} \|g\|_{\B^{\be}_{p_2,\infty}}$, we arrive at
$$
\|\dk \Ll(f \pl  g\Rr) - f \pl (\dk g)\|_{L^p} \ls \sum_{|i-k| \le 5} \frac{2^{i(1-\al-\be)}}{2^k} \|\na f\|_{\B^{\al-1}_{p_1,\infty}} \, \|g\|_{\B^{\be}_{p_2,\infty}},
$$
which proves \eqref{e:com-dk-step1}.

\medskip

\noindent \emph{Step 2.} Recall from Proposition~\ref{p:derivatives} that $\|\na f\|_{\B^{\al - 1}_{p_1,\infty}} \ls \|f\|_{\B^\al_{p_1,\infty}}$. In order to conclude the proof, it thus suffices to show that
\begin{equation}
\label{e:com-dk-step2}
\|f \pge (\dk g)\|_{L^p} \ls 2^{-k(\al + \be)} \|f\|_{\B^\al_{p_1,\infty}} \, \|g\|_{\B^\be_{p_2,\infty}}.
\end{equation}
We have
$$
f \pge (\dk g) = \sum_{i,j :\, i \le j+1} \de_i \dk g \, \de_j f.
$$
As observed previously, $\de_i \dk g$ vanishes unless $|i-k| \le 1$. In this case, by writing~$\delta_i$ as a convolution against $\eta_i$, applying Young's convolution inequality in the form of \eqref{e.Young-torus} and recalling \eqref{e.estim-L1}, we obtain
$$
\|\de_i \dk g\|_{L^{p_1}} \ls \|\dk g\|_{L^{p_1}} \le  2^{-k\be} \|g\|_{\B^\be_{p_2,\infty}}.
$$ 
Since we also have $\|\de_j f\|_{L^{p_2}} \le 2^{-j\al} \|f\|_{\B^\al_{p_2,\infty}}$, we obtain
$$
\|f \pge (\dk g)\|_{L^p} \ls 2^{-k\al} \|g\|_{\B^\be_{p_2,\infty}}\, \sum_{j \ge k-2} 2^{-j\al} \| f\|_{\B^\al_{p_2,\infty}},
$$
and \eqref{e:com-dk-step2} follows since we assume that $\al > 0$. 
\end{proof}

\begin{proof}[Proof of Proposition~\ref{p:comm1}]
Observe that
\begin{align*}
(f \pl g) \pe h & = \sum_{|k-k'| \le 1} \dk (f \pl g) \, \de_{k'} h \\
& = \sum_{i,k,k': \, |k-k'| \le 1} \dk (\de_i f \pl g) \, \de_{k'} h .
\end{align*}
The Fourier spectrum of $\de_i f \pl g$ is contained in $2^{i} \hat {\mathscr A}$, where $\hat {\mathscr A}$ is the annulus $B(0,20/3) \setminus B(0,1/24)$. As a consequence, $\dk (\de_i f \pl g)$ vanishes unless $|k-i| \le 6$, and
\begin{align}
(f \pl g) \pe h & = \sum_{|k-k'| \le 1, i-k \le 6} \dk(\de_i f \pl g) \, \de_{k'} h \notag \\
& =  \sum_{|k-k'| \le 1, i-k \le 6} \de_i f \, \dk g \, \de_{k'} h + \sum_{|k-k'| \le 1, i-k \le 6} [\dk,\pl](\de_i f, g) \, \de_{k'} h.
\label{e:comm1:1}
\end{align}
As a first step, we show that the $\B^{\al + \be + \ga}_{p,\infty}$ norm of the second sum is bounded by a constant times $\|f\|_{\B^\al_{p_1,\infty}} \, \|g\|_{\B^\be_{p_2,\infty}} \, \|h\|_{\B^\ga_{p_3,\infty}}$. For each fixed $k$, the Fourier spectrum of
$$
\msf{co}_k := \sum_{k',i:\, |k-k'| \le 1, i-k \le 6} [\dk,\pl](\de_i f, g) \, \de_{k'} h
$$
is contained in a ball whose radius grows proportionally to $2^k$. By \cite[Lemma~2.84]{BCD} (or \cite[Lemma~3.16]{JCH}), and since $\al + \be + \gamma > 0$, it thus suffices to show that
\begin{equation}
\label{e:comm1:step-rest}
\Ll\| \Ll( 2^{k(\al + \be + \ga)} \| \msf{co}_k \|_{L^p} \Rr)_{k \ge -1}  \Rr\|_{\ell^{\infty}} \ls \|f\|_{\B^\al_{p_1,\infty}} \, \|g\|_{\B^\be_{p_2,\infty}} \, \|h\|_{\B^\ga_{p_3,\infty}}.
\end{equation}
We can rewrite $\msf{co}_k$ as
$$
\sum_{k' :\, |k-k'| \le 1} [\dk,\pl] \Ll( \sum_{i \le k + 6} \de_i f,g \Rr) \, \de_{k'} h.
$$
By Lemma~\ref{l:com-step2} and H\"older's inequality, the $L^p$ norm of $\msf{co}_k$ is thus bounded by
\begin{align*}
& \sum_{k' :\, |k-k'| \le 1} 2^{-k(\al + \be)} \Ll\| \sum_{i \le k + 6} \de_i f \Rr\|_{\B^\al_{p_1,\infty}} \, \|g\|_{\B^\be_{p_2,\infty}} \, \|\de_{k'} h\|_{L^{p_3}} \\
& \qquad \ls 2^{-k(\al + \be + \ga)} \|f\|_{\B^\al_{p_1,\infty}} \, \|g\|_{\B^\be_{p_2,\infty}} \, \|h\|_{\B^\ga_{p_3,\infty}},
\end{align*}
which proves \eqref{e:comm1:step-rest}. 

\medskip

Now that we have controlled the second sum in \eqref{e:comm1:1}, we will argue that the first sum is close to $f(g \pe h)$. We observe that
$$
f(g \pe h) = \sum_{i,k,k' :\, |k-k'| \le 1} \de_i f \, \de_k g \, \de_{k'} h,
$$
so the difference between the first sum in \eqref{e:comm1:1} and $f(g \pe h)$ is given by
$$
\sum_{i,k,k' :\,|k-k'| \le 1, i-k > 6} \de_i f \, \de_k g \, \de_{k'} h.
$$
As above, in order to control the $\B^{\al + \be + \ga}_{p,\infty}$ norm of this term, we observe that for each $i$, the Fourier spectrum of 
$$
\msf s_i := \sum_{k,k' :\,|k-k'| \le 1, k < i-6} \de_i f \, \de_k g \, \de_{k'} h
$$
is contained in a ball whose radius grows proportionnally to $2^i$. Hence, it suffices to show that
\begin{equation}
\label{e:comm1:step-s}
\Ll\| \Ll( 2^{i(\al + \be + \ga)} \| \msf{s}_i \|_{L^p} \Rr)_{i \ge -1}  \Rr\|_{\ell^{\infty}} \ls \|f\|_{\B^\al_{p_1,\infty}} \, \|g\|_{\B^\be_{p_2,\infty}} \, \|h\|_{\B^\ga_{p_3,\infty}}.
\end{equation}
By H\"older's inequality, 
\begin{align*}
\|\msf s_i\|_{L^p} & \le \sum_{k,k' :\,|k-k'| \le 1, k < i-6} \|\de_i f\|_{L^{p_1}} \, \|\de_k g\|_{L^{p_2}} \, \|\de_{k'} h\|_{L^{p_3}} \\
& \le 2^{-\al i} \|f\|_{\B^\al_{p_1,\infty}}  \sum_{k,k' :\,|k-k'| \le 1, k < i-6} 2^{-k \be-k'\ga} \|g\|_{\B^\be_{p_2,\infty}} \, \|h\|_{\B^\ga_{p_3,\infty}} \\
& \le 2^{-i(\al + \be + \ga)}  \|f\|_{\B^\al_{p_1,\infty}}\, \|g\|_{\B^\be_{p_2,\infty}} \, \|h\|_{\B^\ga_{p_3,\infty}} ,
\end{align*}
where we used the fact that $\be + \ga < 0$ in the last step. The proof is thus complete.
\end{proof}

\subsection{Heat flow}

The next proposition quantifies the regularising effect of the heat flow, see e.g.\ \cite[Propositions~3.11 and 3.12]{JCH}.

\begin{prop}[Regularisation by heat flow]
\label{p:smooth-besov} 
Let $\al, \be \in \R$ and $p,q \in [1,\infty]$. 

\noindent $\bullet$ If $\al \ge\be$, then there exists $C < \infty$ such that uniformly over $t > 0$,
$$
\|e^{t\Delta} f\|_{\B^\al_{p,q}} \le C \, t^{\frac{\be-\al}{2}} \, \|f\|_{\B_{p,q}^\be}.
$$

\noindent $\bullet$ If $0 \le \be - \al \le 2$, then there exists $C < \infty$ such that uniformly over $t \ge 0$,
$$
\|(1-e^{t\Delta}) f \|_{\B_{p,q}^\al} \le C t^{\frac{\be-\al}{2}} \|f\|_{\B_{p,q}^\be}.
$$
\end{prop}
\begin{rem}
\label{r:Lpbound}
We also have, for every $p \in [1,\infty]$ and $t \ge 0$,
$$
\|e^{t\Delta} f\|_{L^p} \le \|f\|_{L^p}.
$$
Indeed, the heat kernel has unit $L^1$ norm, so the inequality above follows by Young's convolution inequality.
\end{rem}
\begin{rem}  
\label{r.smooth.withc}
Since, for every $c \ge 0$, 
\begin{equation*}  
(1- e^{t(\Delta-c)}) f = f - e^{t\Delta} f - (1-e^{-ct})e^{t\Delta} f,
\end{equation*}
and $1-e^{-ct} \le ct$, in the setting of the second part of Proposition~\ref{p:smooth-besov}, we also have
\begin{equation*}  
\|(1- e^{t(\Delta-c)}) f\|_{\B_{p,q}^\al} \le C \Ll( t^{\frac{\be-\al}{2}} + ct\Rr) \|f\|_{\B_{p,q}^\be}.
\end{equation*}
\end{rem}

We now turn to our second commutator estimate, which extends Lemma~32 in the first arXiv version of \cite{Gubi} to more general Besov spaces (see also \cite[Lemma~2.5]{catcho}).
\begin{prop}[commutation between $e^{t\De}$ and $\pl$]
\label{p:comm2}
Let $\al < 1$, $\be \in \R$, $\ga \ge \al + \be$, and $p,p_1,p_2 \in [1,\infty]$ such that $1/p = 1/p_1 + 1/p_2$. For every $t\ge 0$, define
$$
[e^{t\De},\pl] : (f,g) \mapsto e^{t\De}(f\pl g) - f \pl(e^{t\De} g).
$$
There exists $C < \infty$ such that uniformly over $t > 0$,
$$
\|[e^{t\De},\pl] (f,g) \|_{\B^{\ga}_{p,\infty}} \le Ct^{\frac {\al + \be -\ga} 2} \, \|f\|_{\B^\al_{p_1,\infty}}\, \|g\|_{\B^\be_{p_2,\infty}}.
$$
\end{prop}
\begin{proof}
We will actually show that 
$$
\|[e^{t\De},\pl] (f,g) \|_{\B^{\ga}_{p,\infty}} \le Ct^{\frac {\al + \be -\ga} 2} \, \|\na f\|_{\B^{\al-1}_{p_1,\infty}}\, \|g\|_{\B^\be_{p_2,\infty}}.
$$
Since $\|\na f\|_{\B^{\al - 1}_{p_1,\infty}} \ls \|f\|_{\B^\al_{p_1,\infty}}$ by Proposition~\ref{p:derivatives}, this implies the proposition.

We decompose $[e^{t\De},\pl](f,g)$ into $\sum_{k = 0}^{+\infty} \msf{h}_k$,
where
$$
\msf h_k := e^{t\De}(S_{k-1} f \, \dk g) - S_{k-1} f \, \dk(e^{t\De} g).
$$
The Fourier spectrum of $\msf h_k$ is contained in $2^k {\mathscr A}$, where we recall that ${\mathscr A}$ is the annulus $B(0,10/3) \setminus B(0,1/12)$. By \cite[Lemma~2.84]{BCD} (or \cite[Lemma~3.16]{JCH}), it thus suffices to show that
$$
\Ll\| \Ll( 2^{k\ga} \|\msf h_k\|_{L^p} \Rr)_{k \ge 0}  \Rr\|_{\ell^{\infty}} \ls t^{\frac {\al + \be -\ga} 2} \, \|\na f\|_{\B^{\al-1}_{p_1,\infty}}\, \|g\|_{\B^\be_{p_2,\infty}}.
$$
Let $\phi \in C^\infty_c$ be supported on an annulus and such that $\phi = 1$ on ${\mathscr A}$, and let 
$$
G_{k,t} = \mathscr{F}^{-1} \Ll( \phi(2^{-k} \, \cdot\, ) \, e^{-t|\cdot|^2} \Rr) .
$$ 
Any function $h$ whose Fourier spectrum lies in $2^k {\mathscr A}$ satisfies
$$
e^{t\De} h = G_{k,t} \star h.
$$
In particular, 
$$
\msf h_k = G_{k,t} \star \Ll( S_{k-1} f \, \dk g \Rr) - S_{k-1} f \Ll( G_{k,t} \star \dk g \Rr) ,
$$
that is,
$$
\msf h_k(x) = \int G_{k,t}(y) \, \dk g(x-y) \Ll( S_{k-1} f(x) - S_{k-1} f(x-y) \Rr)  \, \d y.
$$
We can rewrite the difference of $S_{k-1} f$ at two points in terms of its gradient:
$$
S_{k-1} f(x) - S_{k-1} f(x-y) = - \int_0^1 \na S_{k-1}f(x - sy) \cdot y \, \d s,
$$
so that
$$
\msf h_k(x) = \int_0^1 \int \dk g(x-y) \td G_{k,t}(y) \cdot \na S_{k-1} f(x-sy) \, \d y \, \d s,
$$
where $\td G_{k,t}(y) := y \, G_{k,t}(y)$. Let us denote the inner integral above by $\msf h_{k,s}(x)$. We now show that
\begin{equation}
\label{e:gen-Young}
\|\msf h_k\|_{L^p} \ls \|\td G_{k,t}\|_{L^1} \|\na S_{k-1} f\|_{L^{p_1}} \, \|\dk g\|_{L^{p_2}}.
\end{equation}
We will in fact show that \eqref{e:gen-Young} holds with $\msf h_{k,s}$ in place of $\msf h_k$, uniformly over~$s$. (This inequality is a minor variant of Young's and H\"older's inequalities; in particular, it does not depend on the specific properties of the functions involved, and the implicit multiplicative constant would be $1$ if all functions were real-valued.) We first observe that by H\"older's inequality,
$$
\msf h_{k,s}(x) \ls \|\td G_{k,t}\|_{L^1}^{1 - \frac 1 p} \Ll( \int |\td G_{k,t}(y)| \, |\dk g (x-y)|^p \, |\na S_{k-1} f(x-sy)|^p \, \d y   \Rr)^{\frac 1 p}.
$$
As a consequence,
$$
\|\msf h_{k,s}\|_{L^p}^p  \ls \|\td G_{k,t}\|_{L^1}^{p-1} \iint |\td G_{k,t}(y)| \, |\dk g (x-y)|^p \, |\na S_{k-1} f(x-sy)|^p \, \d y \, \d x   .
$$
By H\"older's inequality, 
$$
\int  |\na S_{k-1} f(x-sy)|^p  \, |\dk g (x-y)|^p \, \d x  \le \|\na S_{k-1} f\|^p_{L^{p_1}} \, \|\dk g\|_{L^{p_2}}^p , 
$$
and we obtain \eqref{e:gen-Young}. 

The remaining step consists in uncovering the size of $\|\td G_{k,t}\|_{L^1}$ in terms of $k$ and~$t$. By symmetry, it suffices to study the $L^1$ norm of the function $y \mapsto y_1 \td G_{k,t}(y)$. Up to a factor $i$, this function is the inverse Fourier transform of 
$$
\zeta \mapsto \dr_1 \Ll( \phi(2^{-k} \zeta ) \, e^{-t|\zeta|^2} \Rr) = \Ll(2^{-k}  \dr_1 \phi(2^{-k} \zeta) - 2 \zeta_1 t \phi(2^{-k} \zeta)\Rr)e^{-t|\zeta|^2} .
$$
We learn from the proof of \cite[Lemma~2.4]{BCD} (or that of \cite[Lemma~2.10]{JCH}) that for every $\td \phi \in C^\infty$ with support in an annulus, there exists $c > 0$ such that 
$$
\Ll\| \mathscr{F}^{-1}\Ll( \td \phi(2^{-k} \, \cdot \,) \, e^{-t|\cdot|^2} \Rr)  \Rr\|_{L^1} \ls e^{-c t 2^{2k}}.
$$
As a consequence, there exists $c > 0$ such that
$$
\|\td G_{k,t}\|_{L^1} \ls 2^{-k} \Ll(1  + t 2^{2k} \Rr) e^{-ct2^{2k}}.
$$
Combining with \eqref{e:gen-Young}, we get
$$
\|\msf h_k\|_{L^p} \ls 2^{-k}\Ll(1  + t 2^{2k} \Rr) e^{-ct2^{2k}} \, \|\na S_{k-1} f\|_{L^{p_1}} \, \|\dk g\|_{L^{p_2}}.
$$
By definition of the Besov norm, we have $\|\dk g\|_{L^{p_2}} \le 2^{-k\be} \|g\|_{\B^\be_{p_2,\infty}}$. Since we assume $\al < 1$, we also have $\|\na S_{k-1} f\|_{L^{p_1}} \ls 2^{k(1-\al)} \|\na f\|_{\B^{\al - 1}_{p_1,\infty}}$, and thus
\begin{align*}
2^{k\ga} \|\msf h_k\|_{L^p} & \ls 2^{k(\ga-\al - \be)}\Ll(1  + t 2^{2k} \Rr) e^{-ct2^{2k}} \, \|\na S_{k-1} f\|_{L^{p_1}} \, \|\dk g\|_{L^{p_2}} \\
& \ls t^{\frac{\al + \be - \ga}{2}}\Ll[\Ll(t 2^{2k}\Rr)^{\frac{\ga-\al - \be}{2}} \Ll(1  + t 2^{2k} \Rr) e^{-ct2^{2k}}\Rr] \, \|\na S_{k-1} f\|_{L^{p_1}} \, \|\dk g\|_{L^{p_2}} .
\end{align*}
The term between square brackets is uniformly bounded, so the proof is complete.
\end{proof}

\bibliographystyle{abbrv}
\bibliography{global}

\end{document}